\documentclass[12pt,oneside]{amsart}
\usepackage[a4paper, hmargin=2.5cm, vmargin=3cm,marginparwidth=2cm]{geometry}
\usepackage{amsthm}
\usepackage{amsmath}
\usepackage{amsfonts}
\usepackage{mathtools}
\usepackage{enumitem}
\usepackage{marginnote}
\usepackage{xcolor}
\usepackage{hyperref}
\usepackage{amssymb}

\setlist[enumerate]{label=(\alph*)}
\numberwithin{equation}{section}

\definecolor{Darkgreen}{rgb}{0,0.4,0}
\hypersetup{
  colorlinks,
  linkcolor=Darkgreen,
  citecolor=Darkgreen,
  breaklinks=true,
  bookmarksnumbered=true,
  plainpages=false,
  unicode=true,
  pdfauthor={Jiří Černý, Ramon Locher},
  pdftitle={An invariance principle for the critical level-set
    percolation of the Gaussian free field on regular trees}
  }

\newtheorem{theorem}{Theorem}[section]
\newtheorem{proposition}[theorem]{Proposition}
\newtheorem{lemma}[theorem]{Lemma}
\newtheorem{corollary}[theorem]{Corollary}

\theoremstyle{definition}

\theoremstyle{remark}
\newtheorem{remark}[theorem]{Remark}

\newcommand{\R}{\mathbb{R}}

\newcommand{\N}{\mathcal{N}}
\newcommand{\F}{\mathcal{F}}

\newcommand{\Q}{Q}
\newcommand{\T}{\mathbb{T}}

\newcommand{\defeq}{\coloneqq}
\newcommand{\eqdef}{\eqqcolon}

\newcommand{\D}{\mathop{}\!\mathrm{d}}

\let\varphi\varphi

\DeclareMathOperator{\spn}{span}

\DeclareMathOperator{\diam}{diam}
\DeclareMathOperator{\skel}{skel}
\DeclareMathOperator{\Var}{Var}
\DeclareMathOperator{\desc}{desc}
\DeclareMathOperator{\sib}{sib}

\DeclarePairedDelimiter{\abs}{\lvert}{\rvert}
\DeclarePairedDelimiter{\ceil}{\lceil}{\rceil}
\DeclarePairedDelimiter{\floor}{\lfloor}{\rfloor}
\DeclarePairedDelimiter{\norm}{\lVert}{\rVert}
\DeclarePairedDelimiter{\set}{\lbrace}{\rbrace}
\DeclarePairedDelimiterX{\inp}[2]{\langle}{\rangle}{#1, #2}

\begin{document}

\title[Scaling limits for the critical GFF]
{Scaling limits for the critical level-set percolation
  of the Gaussian free field on regular trees}

\author{Jiří Černý, Ramon Locher}

\begin{abstract}
  We continue the study of the level-set percolation of the discrete
  Gaussian free field (GFF) on regular trees in the critical regime,
  initiated in \cite{CerLoc23}. First, we derive a sharp asymptotic
  estimate for the probability that the connected component of the
  critical level set containing the root of the tree reaches generation
  $n$. In particular, we show that the one-arm exponent satisfies
  $\rho =1$. Next, we establish a Yaglom-type limit theorem for the
  values of the GFF at generation $n$ within this component. Finally, we
  show that, after a correct rescaling, this component conditioned on
  reaching generation $n$ converges, as $n\to\infty$, to Aldous'
  continuum random tree.
\end{abstract}

\maketitle

\section{Introduction}

The Gaussian free field's level-set percolation, especially on
$\mathbb{Z}^d$, is a significant model in percolation theory that is
characterised by its long-range dependencies. Initial investigations into
this model trace back to the 1980s, with pioneering studies
\cite{BriLebMae87, MolSte83,LebSal86}. Over the last decade, renewed
interest has been ignited by the findings in \cite{RodSzn13},
demonstrating that on $\mathbb{Z}^d$, the model undergoes a distinctive
percolative phase transition at a critical threshold $h^* = h^*(d)$ for
any dimension $d \ge 3$. Follow-up research, including papers
\cite{DreRatSap14, PopRat15, DrePreRod18, Szn19.2, ChiNit20, GosRodSev22,
PanSev22}, has provided a comprehensive understanding of the
model's behaviour in both subcritical and supercritical phases, often
making use of additional natural critical points in order to work in
strongly sub-/super-critical regime. Notably, \cite{DGRS23} confirmed the
alignment of these critical points with $h^*$, indicating a precise phase
transition.

In this paper, we continue the study of level-set percolation of the
discrete Gaussian free field (GFF) on regular trees in the critical
regime, building on the work initiated in \cite{CerLoc23}. Our main
contributions are threefold: First, we derive a sharp asymptotic estimate
for the one-arm probability. Second, we establish a Yaglom-type limit
theorem for the field at vertices located at distance $n$ from the root.
Finally, we show that the connected component of the critical level set,
conditioned to be large, converges to Aldous' Continuum Random Tree.
These results are stated precisely in Theorems~\ref{thm:diameter},
\ref{thm:conv-to-exp}, and~\ref{thm:inv-principle}, respectively.

The considered model was initially studied in \cite{Szn15} where the
critical value $h^*$ was identified in terms of the largest eigenvalue of
a specific integral operator. Additionally, a comparison with random
interlacements was employed to establish bounds on $h^*$, notably
demonstrating that $0 < h^* < \infty$. Subsequently, in \cite{AbaCer19},
sub- and supercritical phases of the model were studied in detail. The
main results of this paper include the continuity of the percolation
probability outside the critical level $h^*$, and accurate estimates on
the size of connected components of the level sets in both phases. Later,
in our previous paper \cite{CerLoc23}, properties of the critical and the
near-critical model were investigated. It was proved there that there is
no percolation at the critical point $h^*$, and that the percolation
probability is continuous also at this point, with a precise asymptotics
formula for this probability in the regime $h\uparrow h^*$. Further, we
provided rather precise estimates on the tail of the size of the
connected component at criticality. Here, we complement these findings
with further results concerning the model at the critical level $h^*$.

As in \cite{Szn15, AbaCer19, CerLoc23}, we will strongly rely on the fact
that the model admits a representation as a branching process with an
uncountable and unbounded type space. For the critical single-type
Galton-Watson process, analogous versions of Theorems~\ref{thm:diameter}
and~\ref{thm:conv-to-exp} go back to Kolmogorov~\cite{Kol38} and
Yaglom~\cite{Yag47}. In the more general case of branching processes with
a finite type space, similar results are also well known. General scaling
limits in the spirit of Theorem~\ref{thm:inv-principle} were first
introduced for critical Galton-Watson processes in~\cite{Ald91}, and
later extended to branching processes with finite or countably infinite
types~\cite{Mie08, Rap17}. More recently, \cite{Pow19} shows a similar
scaling limit result in the setting of critical branching diffusion
on bounded domains.

Recent advances in branching process theory have extended these classical
results in several important directions. \cite{CJP24} obtained refined
convergence rates for Yaglom limits in varying environments, providing
Wasserstein metric bounds that may also be useful for analyses in our
unbounded type-space setting. \cite{BDIM23} established Yaglom-type
theorems for branching processes in sparse random environments, an
intermediate framework that connects the classical Galton-Watson model
with fully random environments. \cite{BFS24arXiv} proved a Yaglom-type
theorem for near-critical branching processes in random environments and
further showed that, under survival conditioning, the genealogical
structure of the population at a fixed time horizon converges to a
time-changed Brownian coalescent point process.

Due to the nature of the branching process appearing in our model (in particular
because its type space is uncountable and unbounded), no previous results are
directly applicable to our setting. In this article, we adapt and extend the
strategy of \cite{Pow19}, addressing the fundamental challenge of unbounded type
spaces.

\section{Model and results}
\label{sec:model}

We start with the definition of the model. Let $\T$ be the infinite
$(d+1)$-regular tree, $d\ge 2$, rooted at an arbitrary fixed vertex
$o\in \T$, endowed with the usual graph distance $d(\cdot,\cdot)$. On $\T$,
we consider the Gaussian free field $\varphi = (\varphi_v)_{v\in \T}$
which is a centred Gaussian process whose covariance function agrees with
the Green function of the simple random walk on $\T$ (see
  \eqref{def:green-func} for the precise definition). We use $P$ to
denote the law of this process on $\R^\T$. For $x \in \R$, we write $P_x$
for the conditional distribution of $\varphi$ given that $\varphi_o=x$,
\begin{equation}
  \label{def:conditional-prob}
  P_x[\, \cdot\, ] \defeq P[\, \cdot\mid \varphi_o = x].
\end{equation}
(For an explicit construction of $P_x$, see \eqref{eq:BP-representation}
  and the paragraph below it.) Furthermore, let $\bar{o} \in \T$ be an
arbitrary fixed neighbour of the root $o$, and define the forward tree
$\T^+$ by
\begin{equation}
  \label{def:Tplus}
  \T^+ \defeq \set{v \in \T :
    \bar{o}\text{~is not contained in the geodesic path from $o$ to $v$}}.
\end{equation}

We analyse the percolation properties of the (super-)level sets of
$\varphi$ above level $h\in \mathbb{R}$,
\begin{equation}
  \label{eq:levelset}
  E_{\varphi}^h \defeq \set{v \in \T : \varphi_v \geq h}.
\end{equation}
In particular, we are interested in the connected component of this set
containing the root~$o$,
\begin{equation}
  \label{eq:Cxh}
  \mathcal{C}_o^h \defeq \set{v \in \T :
    v \text{ is connected to } o \text{ in } E_\varphi^h},
  \quad h\in\R.
\end{equation}
The critical height $h^*$ of the level-set percolation is defined by
\begin{equation}
  \label{def:hstar}
  h^* = h^*(d) \defeq
  \inf \set[\big]{h \in \R : P[\abs{\mathcal{C}_o^h} = \infty] = 0}.
\end{equation}
It is well known that $h^*$ is non-trivial and strictly positive (see
  \cite[Corollary 4.5]{Szn15}). Moreover, as proved in \cite{Szn15}, $h^*$
can be characterized with the help of the operator norms of a certain
family of non-negative operators $(L_h)_{h\in \R}$ acting on the space
$L^2(\nu)$, where $\nu$ is a centred Gaussian measure with variance
$\sigma^2_\nu = d/(d-1)$. We provide more details on this
characterization in Section~\ref{sec:notation-and-results} below. Here,
we only define $\lambda_h$ to be the largest eigenvalue of $L_h$ and
$\chi_h$ the corresponding normed eigenfunction, and recall that $h^*$ is
the unique solution to
\begin{equation}
  \lambda_{h^*} = 1.
\end{equation}
The scalar product on $L^2(\nu )$ will be denoted by $\inp{\cdot}{\cdot}$.

We use $N_n^h$ to denote the set of vertices in $\mathcal C_o^h$ that are
at a distance $n$ from the root,
\begin{equation}
  \label{eq:N_n^h-def}
  N_n^h = \set{v \in \mathcal{C}_o^{h} : d(v, o)=n}.
\end{equation}
and set
\begin{equation}
  \label{eq:N_n^hplus-def}
  N_n^{h,+} = N_n^h \cap \T^+.
\end{equation}
Since we almost exclusively deal with the critical case, we abbreviate
$\chi \defeq \chi_{h^*}$, $\mathcal C_o = \mathcal C_o^{h^*}$,
$L \defeq L_{h^*}$, $N_n \defeq N_n^{h^*}$ and
$N_n^+  \defeq N_n^{h^*,+}$.

Our first result describes the exact asymptotic behaviour of the
probabilities (conditional and unconditional) that $N_n$ and $N_n^+$ are
non-empty, that is, that $\mathcal C_o$ has diameter at least~$n$.
\begin{theorem}
  \label{thm:diameter}
  For every $x \ge h^*$, as $n \to \infty$,
  \begin{align}
    \label{eq:diameter-1}
    P_x[N_n^+ \neq \emptyset ]& = C_1 \chi(x) n^{-1}(1 + o(1)), \\
    \label{eq:diameter-3}
    P[N_n^+\neq \emptyset]
    &= C_1 \inp{1}{\chi}n^{-1}(1 + o(1)),
  \end{align}
  where
  \begin{equation}
    \label{eq:C_1-def}
    C_1 = \frac{2d}{d-1} \frac{1}{\inp{\chi^2}{\chi}}.
  \end{equation}
  If the event $\{N_n^+\neq \emptyset\}$ is replaced by
  $\{N_n\neq \emptyset\}$, the same results hold with $C_1$ replaced by
  $\widetilde{C}_1 = C_1(d+1)/d$.
\end{theorem}

In particular, Theorem~\ref{thm:diameter} proves that the one-arm
exponent of the critical level set, defined by
$\rho = -\lim_{n \to \infty}\log n/\log(P[N_n\neq \emptyset])$ (see,
  e.g., \cite[Section 9.1]{Gri99}) satisfies
\begin{equation}
  \rho = 1.
\end{equation}
This complements the values of two other important critical exponents for
our model given in \cite[(2.22)]{CerLoc23}, where it was shown that
$\delta = 2$ and $\beta = 1$.

The one-arm probability and the associated critical exponent is an
actively studied quantity in several prominent percolation models.
Notably, in the context of level-set percolation of the GFF on the metric
graph of $\mathbb{Z}^d$, recent work has led to a detailed understanding
of the one-arm probability across all dimensions. This includes the
derivation of bounds for $d=3$ in \cite{DPR25}, for $d>6$ in
\cite{CD25}, and for the intermediate regime $3\le d \le 6$ in
\cite{CaiDin24arXiv}.

Our remaining main results consider the critical component conditioned on
being large, more precisely conditioned on the rare event
$\set{N_n^+ \neq \emptyset}$. The first such result is a Yaglom-type
limit theorem for the GFF restricted to $N_n^+$.

\begin{theorem}
  \label{thm:conv-to-exp}
	For $f \in L^2(\nu)$,
  $n\ge 1$, and $x \ge h^*$, let $Z_{n}^{f,x}$, resp.~ $Z_n^f$, be a
  random variable distributed as $n^{-1} \sum_{v \in N_n^+}f(\varphi_v)$
  under the conditional measure $P_x[\,\cdot \,| N_n^+ \neq \emptyset]$,
  resp.~$P[\,\cdot\,| N_n^+ \neq \emptyset]$. Let further $Z$ be an
  exponential random variable with mean one. Then, with $C_1$ as in
  \eqref{eq:C_1-def},
  \begin{equation}
    \label{eq:conv-to-exp}
    \lim_{n\to\infty }Z_{n}^{f,x}  =
    \lim_{n\to\infty}Z_{n}^f = C_1^{-1}\inp\chi f Z \qquad\text{in
      distribution}.
  \end{equation}
\end{theorem}
In the case of the critical single-type Galton-Watson process, results
analogous to Theorem~\ref{thm:conv-to-exp} trace back to the classical
work of Yaglom \cite{Yag47}. For more general branching processes with a
finite type space, similar theorems are also well established; see, for
instance, \cite[Theorem~10.1]{Mode71} and other foundational texts in the
branching process literature. In recent years, Yaglom-type limits have
been proven in various extended settings. These include critical
non-local branching Markov processes \cite{HHK22}, branching Brownian
motion with absorption \cite{MaiSch22}, and branching diffusions in
bounded domains \cite{Pow19}. Additionally, \cite{GLL22} establishes a
Yaglom-type result for critical branching processes in a random Markovian
environment with finite state space.

Our third result concerns a scaling limit for the
critical component $\mathcal{C}_o \cap \mathbb{T}^+$, viewed as a metric
space, under the conditional law
$P_x[\, \cdot \,| N_n^+ \neq \emptyset]$. To this end, let
$(T_{n,x}, d_{n,x})$ be a random
compact metric space whose law coincides with that of
$(\mathcal{C}_o \cap \mathbb{T}^+, n^{-1} d)$ under
$P_x[\, \cdot \,| N_n^+ \neq \emptyset]$ (recall that $d$ denotes the
  distance on $\mathbb T$). We show that the sequence
$(T_{n,x},d_{n,x})$ converges in distribution to a conditioned Brownian
continuum random tree
$(T_{\mathbf{e}}, d_{\mathbf{e}})$, whose contour function $\mathbf e$ is
a Brownian excursion conditioned to reach height at least $1$.

\begin{theorem}
  \label{thm:inv-principle}
  For every $x \ge h^*$, as $n \to \infty$,
  \begin{equation}
    (T_{n, x}, d_{n,x}) \to (T_{\mathbf{e}}, d_{\mathbf{e}})
  \end{equation}
  in distribution, with respect to the Gromov-Hausdorff topology.
\end{theorem}

General scaling limits of critical Galton-Watson processes, in the spirit
of Theorem~\ref{thm:inv-principle}, were first introduced in
\cite{Ald91}. A corresponding result for critical multi-type processes
with finitely many types was established in \cite{Mie08}, and later
extended to processes with a countably infinite type space in
\cite{Rap17}. More recently, \cite{CKM24} extended the classical
continuous random tree convergence to Galton-Watson trees evolving in a
random environment, where each generation has a random offspring
distribution with mean one and finite expected variance. In the context
of critical branching diffusions, \cite{Pow19} proves an invariance
principle under the assumptions of a bounded domain, finite second moment
of the offspring distribution, and an elliptic diffusion generator.
However, for branching diffusions in general (unbounded) domains,
analogous results are not yet available (see \cite[Question 1.8]{Pow19}).

\begin{remark}
  Theorems~\ref{thm:conv-to-exp} and \ref{thm:inv-principle} hold without
  any further change if the conditioning therein is changed from
  $P_x[\,\cdot\, | N_n^+\neq \emptyset ]$ to
  $P_x[\,\cdot\, | N_n\neq \emptyset ]$. For the sake of brevity, we
  refrain from providing detailed proofs of these results.
\end{remark}

We briefly discuss the organisation of this article. In
Section~\ref{sec:notation-and-results}, we collect relevant background on
the GFF on regular trees along with the framework of branching processes
with spines. Section~\ref{sec:proof_of_diameter}
presents the proof of Theorem~\ref{thm:diameter}. In
Section~\ref{sec:proof-of-thm2}, we prove Theorem~\ref{thm:conv-to-exp}
and, in a dedicated subsection, establish additional results for the
model conditioned on the event $\set{N_n^+ \neq\emptyset}$.
Section~\ref{sec:S_n-process} introduces an auxiliary martingale $S_n$
and establishes scaling limit results for this martingale and some
related processes. Section~\ref{sec:H_n-S_n-connection} focuses on the
``height process'' $H_n$, defined via the distance to the origin in a
depth-first traversal of $\mathcal{C}_o \cap \mathbb{T}^+$, and
investigates its connection to the martingale $S_n$. Finally,
Section~\ref{sec:inv-principle} concludes the article with the proof of
Theorem~\ref{thm:inv-principle}, which combines topological arguments
with the results from Sections~\ref{sec:S_n-process} and
\ref{sec:H_n-S_n-connection}.

\section{Notation and useful results}
\label{sec:notation-and-results}

In this section we introduce the notation used throughout the paper and
recall some known facts about the level set percolation of the
Gaussian free field on trees. We then briefly present the formalism of
branching processes with spines and apply it to our model.

As already stated in the introduction, we use $\T$ to denote the $(d+1)$-regular
tree, $d\ge 2$,  that is an infinite tree whose every vertex has exactly
$d+1$ neighbours. For two vertices $v,w \in \T$ we use $d(v, w)$ to
denote the usual graph distance. The tree is rooted at an arbitrary fixed
vertex $o \in \T$, $\bar{o} \in \T$ denotes a fixed neighbour of $o$, and
$\T^+$ stands for the forward tree,  see \eqref{def:Tplus}. We set
$\abs v = d(o,v)$ and write
\begin{equation}
  \label{eq:spheres}
  S_n = \{v\in \T: \abs v = n\}, \qquad S_n^+ = S_n \cap \T^+
\end{equation}
for the spheres with radius $n$ centred at $o$. For every
$v\in \T \setminus \set{o}$  we use $p(v)$ to denote its parent in
$\mathbb T$, that is the only vertex on the geodesic path from $v$ to $o$
with $\abs{p(v)}= \abs v-1$. We write $\desc(v)$ for the set of direct
descendants of $v$, and $\sib(v) = \desc(p(v))$ for the set of its
siblings, including itself. Finally, if $w$ is an ancestor of $v$, that is
$w$ lies on the geodesics from $o$ to $v$, we write $w \preceq v$.

Throughout the paper we use the usual notation for the asymptotic
relation of two functions $f$ and $g$: We will write $f(s) \sim g(s)$ as
$s \to \infty$ if $\lim_{s\to \infty} {f(s)}/{g(s)} = 1$, $f(s) = o(g(s))$
as $s \to \infty$ if $\lim_{s \to \infty} {\abs{f(s)}}/{g(s)} = 0$, and
$f(s) = O(g(s))$ as $s \to \infty$ if
$\limsup_{s \to \infty} {\abs{f(s)}/{g(s)}} < \infty$. We use
$c,c',c_1,\dots$ to denote finite positive constants whose value may
change from place to place and which can only depend on $d$. The
dependence of these constants on additional parameters is explicitly
mentioned.

\subsection{Properties of the GFF}
\label{ss:GFF}

We consider the Gaussian free field
$\varphi = (\varphi_v)_{v\in \mathbb T}$ which is the centred Gaussian
process on $\mathbb T$ whose covariance function is the Green function of
the simple random walk on $\T$,
\begin{equation}
  \label{def:green-func}
  E[\varphi_v \varphi_w] = g(v, w)
  \defeq
  \frac{1}{d+1} \mathbb{E}_v\Big[\sum_{k=0}^\infty 1_{X_k = w}\Big],
  \qquad v,w \in \T,
\end{equation}
where $\mathbb E_v$ stands for the expectation with respect to the simple
random walk $(X_k)_{k \geq 0}$ on $\mathbb T$ starting at $v\in \T$.

We frequently use the fact that the Gaussian free field on $\T$ can
be viewed as a multi-type branching process with a continuous type space
(see \cite[Section~3]{Szn15} and \cite[Section 2.1]{AbaCer19}). To this
end, we define
\begin{equation}
  \label{eq:sigmas}
  \sigma_\nu^2 \defeq \frac{d}{d-1}
  \quad \text{and} \quad
  \sigma_Y^2 \defeq \frac{d+1}{d},
\end{equation}
and let $(Y_v)_{v\in\T}$ be a collection of independent centred Gaussian
random variables on some auxiliary probability space such that
$Y_o \sim \N(0, \sigma_\nu^2)$ and $Y_v \sim \N(0, \sigma_Y^2)$ for
$v \neq o$. We then define another field $\widetilde\varphi $ on
$\mathbb T$ by
\begin{equation}
  \label{eq:BP-representation}
  \begin{split}
    &\text{(a)}\quad \widetilde{\varphi}_o \defeq Y_o, \\ &\text{(b)}
    \quad\text{for $v\neq o$ we recursively set
      $\widetilde{\varphi}_v \defeq d^{-1} \widetilde{\varphi}_{p(v)} +
      Y_v$.}
  \end{split}
\end{equation}
As explained, e.g., in~\cite[(2.9)]{AbaCer19}, the law of
$(\widetilde{\varphi}_v)_{v\in\T}$ agrees with the law $P$ of the
Gaussian free field $\varphi$. Therefore, we will always assume that the
considered Gaussian free field is constructed in this way and will not
distinguish between $\varphi$, and $\widetilde \varphi$ from now on.

Representation \eqref{eq:BP-representation} of $\varphi$ can be used to
give a concrete construction for the conditional probability $P_x$
introduced in \eqref{def:conditional-prob}: It suffices to replace (a) in
\eqref{eq:BP-representation} by $\widetilde \varphi_o= x$. In addition,
\eqref{eq:BP-representation} can directly be used to construct a monotone
coupling of $P_x$ and $P_y$. As a consequence:
\begin{equation}
  \label{eq:stoch_dom}
  \text{If $x<y$, then $P_y$ stochastically dominates $P_x$,}
\end{equation}
that is $E_x[f(\varphi)] \le E_y[f(\varphi)]$ for every bounded increasing
function $f:\mathbb R^{\mathbb T}\to \mathbb R$.

From the construction \eqref{eq:BP-representation} it follows that the
GFF on $\T$ can be viewed as a multi-type branching process where the
type corresponds to the value of the field $\varphi$. The type of the
initial individual $o$ of this branching process is distributed as $Y_o$.
Every individual $v$ in this branching process then has $d$ descendants
($d+1$ if $v = o$) whose types are independently given by
$d^{-1}{\varphi}_{v} + Y$, with $Y\sim N(0,\sigma_Y^2)$. The branching
process point of view can be adapted to the connected component
$\mathcal{C}_o^h$ (defined in \eqref{eq:Cxh}) by considering the same
multi-type branching process but killing immediately all individuals with
type smaller than $h$ (and not allowing them to have descendants
  themselves).

We now recall in more detail the spectral machinery introduced in
\cite{Szn15} in order to characterise the critical value $h^*$. Let $\nu$
be a centred Gaussian measure on $\mathbb R$ with variance $\sigma_\nu^2$
(as defined in \eqref{eq:sigmas}), and let $Y$ be a centred Gaussian
random variable with variance $\sigma_Y^2$. The expectation with respect
to this random variable is denoted by $E_Y$. We consider the Hilbert
space $L^2(\nu) \defeq L^2(\R, \mathcal B(\R), \nu)$, and for every
$h \in \R$, define the operator $L_h$ acting on $L^2(\nu)$ by
\begin{equation}
  \label{eq:L_h}
  \begin{split}
    L_h[f](x)
    &\defeq 1_{[h,\infty)}(x)\, d\, E_Y\Big[1_{[h,\infty)}
      \Big(Y + \frac{x}{d}\Big)f\Big(Y+\frac{x}{d}\Big)\Big]\\
    &= 1_{[h,\infty)}(x)\, d \int_{[h, \infty)} f(y)
    \rho_Y\Big(y - \frac{x}{d}\Big) \D y,
  \end{split}
\end{equation}
where $\rho_Y$ denotes the density of $Y$. From the branching process
construction \eqref{eq:BP-representation} of $\varphi $, it follows that
\begin{equation}
  \label{eq:L_h-BP1}
  L_h[f](x) = E_x\bigg[\sum_{v \in N_1^{h,+}} f(\varphi_v)\bigg],
\end{equation}
where $N_n^{h,+}$ is defined in \eqref{eq:N_n^hplus-def}. Iterating this
expression one also obtains
\begin{equation}
  \label{eq:L_h-BPn}
  L_h^n[f](x) = E_x\bigg[\sum_{v \in N_n^{h,+}} f(\varphi_v)\bigg], \qquad
  n\in \mathbb N.
\end{equation}
Finally, we let $\lambda_h$ stand for the operator norm of $L_h$ in
$L^2(\nu )$,
\begin{equation}
  \label{eq:altdef-lambda_h}
  \lambda_h \defeq \norm{L_h}_{L^2(\nu)\to L^2(\nu)}.
\end{equation}

The following proposition summarises some known properties of the
operator $L_h$ as well as the connection between $L_h$ and the critical
height $h^*$.

\begin{proposition}[\cite{Szn15} Propositions 3.1, 3.3,  Corollary 4.5]
  \label{prop:L_h-chi-connection}
  For all $h\in\R$, $L_h$ is a self-adjoint non-negative Hilbert-Schmidt
  operator on $L^2(\nu)$, $\lambda_h$ is a simple eigenvalue of $L_h$,
  and there exists a unique $\chi_h \geq 0$ with unit $L^2(\nu)$-norm,
  which is continuous, strictly positive on $[h, \infty)$, vanishing on
  $(-\infty,h)$, and such that
  \begin{equation}
    \label{eq:chieigenfunction}
    L_h[\chi_h] = \lambda_h \chi_h.
  \end{equation}
  Additionally, the map $h \mapsto \lambda_h$ is a decreasing homeomorphism
  from $\R$ to $(0, d)$ and $h^*$ is the unique value in $\R$ such that
  $\lambda_{h^*} = 1$. Finally, for every $d\ge 2$,
  \begin{equation}
    0 < h^* < \infty.
  \end{equation}
\end{proposition}

Combining Proposition~\ref{prop:L_h-chi-connection} with
\eqref{eq:L_h-BPn} gives that for every $n \in \mathbb{N}$,
\begin{equation}
  \label{eq:chi-is-eigenfunc}
  E_x\bigg[\sum_{w \in N_n^{h,+}} \chi_h(\varphi_w)\bigg]
  = \lambda_h^n \chi_h(x).
\end{equation}

We will require estimates on the norms of $L_h[f]$, which follow from the
hypercontractivity of the Ornstein-Uhlenbeck semigroup. For part (a) of
the following proposition, we refer to (3.14) in \cite{Szn15} or (2.14)
in \cite{AbaCer19}. Part (b) then follows directly from part (a) in
combination with generalized Hölder's inequality.

\begin{proposition}
  \label{prop:hypercontractivity}
  (a) For every $f \in L^2(\nu)$, $h \in \mathbb{R}$, $1 < p < \infty$
  and $q \le (p-1)d^2+1$,
  \begin{equation}
    \norm{L_h[f]}_{L^q(\nu)} \le d \norm{f}_{L^p(\nu)}.
  \end{equation}

  (b) For every $1 \le k \le (d^2+1)/2$ and $f_1,\dots,f_k \in L^2(\nu)$
  \begin{equation}
    \norm[\Big]{\prod_{i=1}^k L_h[f_i]}_{L^2(\nu)}
    \le d^k \prod_{i=1}^k\norm{f_i}_{L^2(\nu)}.
  \end{equation}
\end{proposition}

The next proposition recalls known properties of the critical component
$\mathcal C_o$ from \cite{CerLoc23}.

\begin{proposition}[Theorem~2.1 and Theorem~2.3 in \cite{CerLoc23}]
  \label{pro:cl23}
  There is $C\in (0,\infty)$ such that
  for every $x \in \R$, as $n \to \infty$,
  \begin{equation}
    P_x[\abs{\mathcal{C}_o \cap \mathbb{T}^+} > n]
    = C \chi(x)n^{-1/2}(1+o(1)).
  \end{equation}
  In particular, for every $x\in \mathbb R$,
  \begin{equation}
    P_x[\abs{\mathcal{C}_o \cap \T^+} = \infty]
    = 0.
  \end{equation}
\end{proposition}

We will also need the following three properties of $\chi$, the first one
is proved in Remark~2.5 of \cite{CerLoc23}, the second is a direct
consequence of Proposition~3.1 in \cite{AbaCer19}, and the last one is
proved in Appendix~\ref{app:proof-chi-lipsch}:
\begin{gather}
  \label{eq:chi_non-dec}
  x \mapsto \chi(x) \text{ is non-decreasing},\\
  \label{eq:crit-chi_bounds}
  c_1 x \le \chi(x) \le c_2 x \quad \text{for all $x \ge h^*$ and some
    $c_1,c_2\in (0,\infty)$},\\
  \label{eq:chi-Lipschitz}
  x \mapsto \chi(x) \text{ is Lipschitz on $[h^*, \infty)$.}
\end{gather}

\subsection{Branching processes with spines}
\label{ss:spined-bp}

We now recall the machinery of branching processes with spines which is
frequently used in the theory of branching processes, and specialize it
to our model, in order to study the critical component $\mathcal C_o$. We
then state many-to-few formulas that express certain moments for the
original branching process in terms of the dynamics along the spines, see
Proposition~\ref{pro:many-to-few_model} below. Later, in
Section~\ref{subsec:further-cond-res}, we will see that the processes
with spines can be used to describe the distribution of $\mathcal C_o$
conditioned on being infinite. The content of this section is mostly
based on \cite{HarRob17}.

The main idea of the machinery is to designate several lines of
descent in the branching process, called spines. These spines are then
used to introduce a certain change of measure, under which the vertices
on the spine exhibit modified branching behaviour, while the non-spine
vertices behave as in the original process.

For the construction, we need some notation. For $x\in \mathbb R$ and
$k\in \mathbb N$, we first introduce measures $P^k_x$ under which the
behaviour of the field is the same as under $P_x$, but in addition there
are $k$ distinguished spines. Formally, the measure $P^k_x$  is the
distribution of a
$(\mathbb{R}\cup \set{\dagger}) \times \set{0, 1, \dots, k}$-valued
stochastic process $(\varphi_v,l_v)_{v\in \T^+}$. This process assigns to
every $v \in \mathbb{T}^+$ a field value
$\varphi_v \in \R \cup \set{\dagger}$ (where $\dagger$ is a cemetery
  state) and a number $l_v \in \set{0, \dots, k}$ that represents the
number of spine marks on $v$. Under $P^k_x$ , the law of
$(\varphi_v)_{v \in \mathbb{T}^+}$ is a straightforward modification of
the original measure $P_x$, the only difference is that we set
$\varphi_v = \dagger$ for every $v \not\in \mathcal{C}_o$. The random
variables $(l_v)_{v\in \mathbb T^+}$ are independent of the field values
$(\varphi_v)_{v\in \mathbb T^+}$. The root node $o$ has exactly $k$
marks, $l_o = k$. The remaining random variables $(l_v)_{v\neq o}$ are
constructed recursively as follows: If a node $v\in\T^+$ carries $j$
marks, then each of its $j$ marks `moves' to one of its $d$ direct
descendants independently uniformly at random. (Note that nodes in the
  cemetery state can carry marks.) As consequence, under $P^k_x$, in
every generation $n$, there are exactly $k$ marks present, that is
$\sum_{v\in S^+_n } l_v = k$. We use $P^k_x$ also for the corresponding
expectations.

For $i = 1, \dots, k$, we denote by $\sigma_n^i$ the node that carries
the $i$-th mark in generation $n$ (that is $\abs{\sigma_n^i} = n$), and
set $\xi_n^i = \varphi_{\sigma_n^i}$ to be its type.
We also define $\skel(n) = \{v\in \T^+: \abs v \le n, l_v \ge 1\}$ to be
the set of nodes of generation at most $n$ having at least one mark. We let
$\F_n$ stand for the natural filtration of the branching process, and
$\F_n^k$ for the filtration containing in addition the information about
the $k$ spine marks,
\begin{equation}
  \label{eq:filtrations}
  \mathcal{F}_n = \sigma(\varphi_v : v\in \mathbb T^+, \abs v \le n)
  \quad \text{and} \quad
  \mathcal{F}^k_n
  = \sigma(\varphi_v, l_v :v\in \mathbb T^+, \abs v \le n).
\end{equation}
Any $f:\R\to\R$ is extended to $\R \cup \set{\dagger}$ by setting
$f(\dagger) = 0$. Then, by definition of $P^k_x$,
\begin{equation}
  \label{eq:N^prime-N_equality}
  E_x\bigg[\sum_{v \in N_n^+} f(\varphi_v)\bigg]
  = P_x^k\bigg[\sum_{v \in S_n^+} f(\varphi_v)\bigg].
\end{equation}

We now define another measure $\Q_x^k$, where the nodes without a spine
mark behave as under $P_x^k$ but the nodes with a mark have a modified
branching behaviour: Under $\Q_x^k$ the movement of the marks, and thus
the distribution of $(l_v)_{v\in \T^+}$, is exactly the same as under
$P_x^k$: If a node $v$ carries $k$ marks, each of the marks is given to
one of its $d$ direct descendants independently uniformly at random. To
describe the distribution of the field $\varphi $ under $\Q_x^k$, we first
define a transition kernel (recall $\rho_Y$ from \eqref{eq:L_h})
\begin{equation}
  \label{eq:K-def}
  \mathcal{K}(x, \D y)
  = d \frac{\chi(y)}{\chi(x)}\rho_Y\Big(y-\frac{x}{d}\Big)\D y,
  \qquad x\ge h^*, y\in \mathbb R.
\end{equation}
Note that, by \eqref{eq:L_h} and
Proposition~\ref{prop:L_h-chi-connection}, for every $x\ge h^*$,
$\mathcal K(x,\cdot)$ is a probability measure with support
$[h^*, \infty)$. Conditionally on the marks $(l_v)_{v\in \T^+}$, the
field $(\varphi_v)_{v \in \mathbb{T}^+}$ under $\Q^k_x$ is recursively
constructed by
\begin{equation}
  \label{eq:Q-BP}
  \parbox{14.5cm}{
    \begin{enumerate}
      \item $\varphi_o \defeq x$,
      \item If $v \neq o$ and $l_v= 0$, then
      $\varphi_v = d^{-1} \varphi_{p(v)}+Y_v$ (as under $P_x^k$).
      \item If $v \neq o$ and $l_v\ge 1$, then
      $\varphi_v$ is $\mathcal K(\varphi_{p(v)}, \cdot)$-distributed,
      independently of previous randomness.
    \end{enumerate}
  }
\end{equation}
To simplify notation, we write $\Q_x$ instead of $\Q_x^1$; in this case
we also set $\sigma_n = \sigma^1_n$ and $\xi_n = \xi^1_n$.

Note that, unlike under $P_x^k$, under the measure $\Q_x^k$ the nodes in
the cemetery state cannot carry any mark. Consequently, $\Q_x^k$-a.s.,
there are nodes not in the cemetery state in every generation.

By construction, under $\Q_x^k$, the process $(\xi^i_n)_{n\in \mathbb N}$
recording the value of the field along the $i$-th spine is a Markov chain
with transition kernel $\mathcal K$. This chain never enters the cemetery
state. The following lemma determines its invariant distribution.

\begin{lemma}
  \label{lem:xi_inv-measure}
  The Markov chain $(\xi_n)_{n\in \mathbb N}$ with the transition kernel
  $\mathcal K$ has a unique
  invariant distribution $\pi $ given by
  \begin{equation}
    \label{eq:inv-measure}
    \pi(\D x) = \chi(x)^2\nu(\D x).
  \end{equation}
\end{lemma}

\begin{proof}
  We first show that $\pi$ is invariant for $\mathcal K$. Comparing
  \eqref{eq:L_h} and \eqref{eq:K-def} yields that
  $\mathcal{K}(x, A) = L[1_A \chi](x)/\chi(x)$ for every $x \ge h^*$.
  Therefore, writing the action of $\mathcal{K}$ on $\pi$ as inner
  product on $L^2(\nu)$, using that $L$ is self-adjoint and $\chi$ is its
  eigenfunction with eigenvalue $1$ (see
    Proposition~\ref{prop:L_h-chi-connection}),
  \begin{equation}
    (\pi\mathcal{K})(A) = \inp{\mathcal{K}(\cdot,A)}{\chi^2}_{\nu}
    = \inp{L[1_A \chi]}{\chi}_{\nu} = \inp{1_A \chi}{L[\chi]}_{\nu}
    = \inp{1_A \chi}{\chi}_{\nu} = \pi(A),
  \end{equation}
  which shows that $\pi$ is an invariant measure. The uniqueness follows
  from the irreducibility of $(\xi_n)_{n\in \mathbb N}$.
\end{proof}

Next, we state several moment formulas that are frequently used throughout
the paper. Such formulas are well understood in the theory of branching
processes, see, e.g., \cite{HarRob17} and references therein. The proof
of the following proposition is based on Lemma 8 of that paper and can be
found in Appendix~\ref{sec:B}.

\begin{proposition}
  \label{pro:many-to-few_model}
  For all functions $f, g: [h^*, \infty) \to \R$ for which the
  expectations below are well defined,
  \begin{align}
    \label{eq:many-to-few_model1}
    E_x\Big[\sum_{v\in N^+_n} f(\varphi_v)\Big]
    &= \Q_x\Big[f(\xi_n) \frac{\chi(x)}{\chi(\xi_n)}\Big],
    \\
    \label{eq:many-to-few_model2}
    E_x\Big[\sum_{v,w\in N_n^+} f(\varphi_v)  g(\varphi_w)\Big]
    &= \begin{aligned}[t]&\chi(x) \frac{d-1}{d} \sum_{k=0}^{n-1}
      \Q_x\bigg[\chi(\xi_{k})
        \Q_{\xi_k}\Big[\frac{f(\xi_{n-k})}{\chi(\xi_{n-k})}\Big]
        \Q_{\xi_k}\Big[\frac{g(\xi_{n-k})}{\chi(\xi_{n-k})}\Big]\bigg]
      \\&\quad + \chi (x) \Q_x\Big[\frac{f(\xi_n)g(\xi_n)}{\chi
          (\xi_n)}\Big].
    \end{aligned}
  \end{align}
\end{proposition}

\subsection{Asymptotic behaviour of the moments}

The main result of this section is Proposition~\ref{pro:Harris_L2} giving
precise asymptotic estimates on the quantities appearing in
\eqref{eq:many-to-few_model1}, \eqref{eq:many-to-few_model2}. To prove
them, we could, in principle, use this proposition together with the
known results on the convergence of Markov chains. However, it is easier,
and for our purposes slightly more practical, to use
formula~\eqref{eq:L_h-BPn} together with $L^2$-estimates on the
operator~$L$.

Since $L$ is a self-adjoint Hilbert-Schmidt operator (see
  Proposition~\ref{prop:L_h-chi-connection}), $L^2(\nu)$ has an
orthonormal basis consisting of the eigenfunctions $\set{e_k}_{k\geq 1}$
of $L$ corresponding to the eigenvalues $\set{\lambda_k}_{k\geq 1}$. By
Proposition~\ref{prop:L_h-chi-connection} we may assume that
$1=\lambda_1 > \abs{\lambda_2} \geq \abs{\lambda_3}\geq \dots $, and
$e_1 = \chi$. We set $\gamma  = \abs{\lambda_2}\in (0,1)$. By decomposing
$f \in L^2(\nu)$ as
\begin{equation}
  \label{eq:f-decomp}
  f = \sum_{k\geq 1} \inp{e_k}{f} e_k
  = \inp{\chi}{f}\chi + \sum_{k\geq 2} \inp{e_k}{f} e_k
  \eqdef \inp{\chi}{f}\chi + \beta[f],
\end{equation}
for every $n\in \mathbb N$, it holds
\begin{equation}
  \label{eq:decomp_Ly}
  L^n[f] = \sum_{k\geq 1} \lambda_k^n \inp{e_k}{f} e_k
  = \inp{\chi}{f}\chi + \sum_{k\geq 2} \lambda_k^n \inp{e_k}{f} e_k
  = \inp{\chi}{f}\chi + L^n[\beta[f]],
\end{equation}
with
\begin{equation}
  \label{eq:norm-beta}
  \norm{L^n[\beta[f]]}
  \le \gamma^n \norm{\beta[f]} \le \gamma^n \norm f.
\end{equation}

The following simple lemma will later be used to deduce the pointwise
convergence from the $L^2(\nu )$-one.
\begin{lemma}
  \label{lem:L-pointwise}
  There is a function $q: [h^*, \infty) \to [0, \infty)$
  such that for every $x \ge h^*$ and $f\in L^2(\nu )$
  \begin{equation}
    \abs{L[f](x)} \le \norm{f}q(x).
  \end{equation}
\end{lemma}

\begin{proof}
  Using the definition \eqref{eq:L_h} of $L$ and the Cauchy--Schwarz
  inequality,
  \begin{equation}
    \begin{split}
      \abs{L[f](x)}
      &\le d\int_{\R}\abs{f(y)}\rho_Y\Big(y-\frac xd\Big)\D y
      = d\int_{\R}\abs{f(y)}\frac{\rho_Y(y-x/d)}{\rho_\nu(y)} \nu(\D y)
      \\&\le d\norm{f}_{L^2(\nu)}
      \norm[\Big]{\frac{\rho_Y(\boldsymbol{\cdot} - x/d)}
        {\rho_\nu(\boldsymbol{\cdot})}}_{L^2(\nu)}
      \eqdef \norm{f}_{L^2(\nu)} q(x).
    \end{split}
  \end{equation}
  By \eqref{eq:sigmas}, $\sigma_\nu  > \sigma_Y$. Using this one can
  easily check by a direct computation that $ q(x) < \infty$ for every
  $x\ge h^*$.
\end{proof}

\begin{remark}
  The same computation also shows that $q(x) \in L^2(\nu )$, but we will
  not use this fact.
\end{remark}

We can now give the asymptotic estimates on moments appearing in
Proposition~\ref{pro:many-to-few_model}.

\begin{proposition}
  \label{pro:Harris_L2}
  (a) For every $f \in L^2(\nu)$, $x\ge h^*$, and $n\in \mathbb N$,
  \begin{equation}
    \label{eq:Harris_L2a}
    E_x\Big[\sum_{v \in N_n^+} f(\varphi_v)\Big]
    = \chi(x) \inp{\chi}{f} + \varepsilon_n^f(x),
  \end{equation}
  where the error term $\varepsilon_n^f$ satisfies
  (with $q$ as in Lemma~\ref{lem:L-pointwise})
  \begin{equation}
    \label{eq:errorsa}
    \norm{\varepsilon_n^f} \le
    \gamma^n \norm{f}
    \qquad \text{and}\qquad
    \abs{\varepsilon_n^f(x)} \le \gamma^{n-1}\norm{f}q(x).
  \end{equation}

  (b) For every  $f,g \in L^2(\nu )$, $x\ge h^*$, and $n\in \mathbb N$
  \begin{equation}
    \begin{split}
      \label{eq:Harris_L2b}
      E_x\Big[\sum_{v,w\in N_n^+} f(\varphi_v) g(\varphi_w)\Big]
      &- E_x\Big[\sum_{v \in N_n^+} f(\varphi_v)g(\varphi_v)\Big]
      \\&= \chi(x)\frac{d-1}{d}\inp{\chi}{f}
      \inp{\chi}{g}\inp{\chi^2}{\chi} n + \varepsilon_n^{f,g}(x),
    \end{split}
  \end{equation}
  where the error term $\varepsilon_n^{f,g}$ satisfies
  \begin{equation}
    \norm{\varepsilon_n^{f,g}} \le C \norm{f}\norm{g}
    \qquad\text{and}\qquad
    \varepsilon_n^{f,g}(x) \le \norm{f}\norm{g}\tilde q(x)
  \end{equation}
  with $\tilde q(x) \le C ( q(x) + \chi (x) )<\infty$.
\end{proposition}

\begin{proof}
  (a) By \eqref{eq:L_h-BPn}, the left-hand side of \eqref{eq:Harris_L2a}
  equals $L^n[f]$. Therefore, by \eqref{eq:decomp_Ly},
  $\varepsilon^f_n = L^n[\beta[f]]$. \eqref{eq:norm-beta} then implies the
  first claim in \eqref{eq:errorsa}. The second one follows from the
  first one and Lemma~\ref{lem:L-pointwise}.

  (b) Combining \eqref{eq:many-to-few_model2} with
  \eqref{eq:many-to-few_model1} and \eqref{eq:L_h-BPn}, we obtain that
  \begin{equation}
    \label{eq:Harris_doublesum-equation}
    E_x\Big[\sum_{v,w \in N^+_n}f(\varphi_v) g(\varphi_w)\Big]
    = \frac{d-1}{d} \sum_{k=0}^{n-1} L^{k}\Big[L^{n-k}[f]L^{n-k}[g]\Big](x)
    + L^n[fg](x).
  \end{equation}
  By \eqref{eq:Harris_L2a},
  \(
    L^{n-k}[f]L^{n-k}[g] = (\inp{\chi}{f}\chi +
      \varepsilon^f_{n-k})(\inp{\chi}{g}\chi + \varepsilon^g_{n-k})
  \)
  and thus, again by \eqref{eq:Harris_L2a},
  \begin{equation}
    \begin{split}
      \label{eq:Harris_L^k-term}
      L^{k}\Big[L^{n-k}[f]L^{n-k}[g]\Big]
      &= \inp{\chi}{L^{n-k}[f]L^{n-k}[g]}\chi
      + \varepsilon_{k}^{L^{n-k}[f]L^{n-k}[g]} \\
      &= \Big(\inp{\chi}{f}\inp{\chi}{g}\inp{\chi}{\chi^2} +
        \inp{\chi}{f}\inp{\chi}{\chi \varepsilon_{n-k}^g}
        + \inp{\chi}{g}\inp{\chi}{\chi \varepsilon_{n-k}^f}  \\
        &\quad\quad
        + \inp{\chi}{\varepsilon_{n-k}^f \varepsilon_{n-k}^g} \Big) \chi
      + \varepsilon_{k}^{L^{n-k}[f]L^{n-k}[g]}.
    \end{split}
  \end{equation}
  Note also that the $L^n[fg](x)$ summand in
  \eqref{eq:Harris_doublesum-equation} equals
  $E_x[\sum_{v \in N_n^+} f(\varphi_v)g(\varphi_v)]$. Therefore, the
  error term in \eqref{eq:Harris_L2b} satisfies
  \begin{equation}
    \begin{split}
      \label{eq:Harris_L2b-error}
      \varepsilon_n^{f,g}(x)
      &= \frac{d-1}{d}\sum_{k = 0}^{n-1}
      \Big(\big(\inp{\chi}{f}\inp{\chi}{\chi \varepsilon_{n-k}^g}
          + \inp{\chi}{g}\inp{\chi}{\chi \varepsilon_{n-k}^f}
          + \inp{\chi}{\varepsilon_{n-k}^f \varepsilon_{n-k}^g}\big) \chi(x)
        \\&\quad\quad + \varepsilon_{k}^{L^{n-k}[f]L^{n-k}[g]}(x)\Big).
    \end{split}
  \end{equation}
  To bound this expression, note that  $\abs{\inp{\chi}{f}} \le \norm{f}$ and
  \(
    \abs{\inp{\chi}{\chi \varepsilon_{n-k}^g}}
    \le \norm{\chi^2}\norm{\varepsilon_{n-k}^g}
    \le C \gamma^{n-k}\norm{g}
  \),
  by part (a). Therefore, the first summand in
  \eqref{eq:Harris_L2b-error} satisfies
  \(
    \abs{\inp{\chi}{f}\inp{\chi}{\chi \varepsilon_{n-k}^g}}
    \le C \gamma^{n-k}\norm{f}\norm{g}.
  \)
  Analogously,
  \(
    \abs{\inp{\chi}{g}\inp{\chi}{\chi \varepsilon_{n-k}^f}}
    \le C \gamma^{n-k}\norm{f}\norm{g}
  \).
  To estimate the third summand,  we observe that by
  \eqref{eq:Harris_L2a} $\varepsilon_{n-k}^f = L^{n-k}[\beta[f] ]$.
  Therefore, using Proposition~\ref{prop:hypercontractivity}(b),
  \(
    \norm{\varepsilon_{n-k}^f \varepsilon_{n-k}^g}
    \le d^2 \norm{\varepsilon_{n-k-1}^f}
    \norm{\varepsilon_{n-k-1}^g}
  \)
  and thus
  \begin{equation}
    \inp{\chi}{\varepsilon_{n-k}^f\varepsilon_{n-k}^g}
    \le d^2 \norm{\varepsilon_{n-k-1}^f} \norm{\varepsilon_{n-k-1}^g}
    \le d^2 \gamma^{2(n-k-1)}\norm{f}\norm{g}.
  \end{equation}
  Finally, for the fourth summand, using again
  Proposition~\ref{prop:hypercontractivity}(b) and part (a), since the
  operator norm of $L$ equals $1$,
  \begin{equation}
    \begin{split}
      \norm{\varepsilon_k^{L^{n-k}[f]L^{n-k}[g]}}
      &\le \gamma^k \norm{L^{n-k}[f]L^{n-k}[g]}
      \\&\le d^2 \gamma^k \norm{L^{n-k-1}[f]}\norm{L^{n-k-1}[g]}
      \le C \gamma^k \norm{f}\norm{g}.
    \end{split}
  \end{equation}
  Taking the norm in \eqref{eq:Harris_L2b-error} and making use of the above
  estimates yields
  \begin{equation}
    \norm{\varepsilon_n^{f,g}} \le C \norm{f}\norm{g}
    \sum_{k=0}^{n-1} (\gamma^{n-k} +  \gamma^{2(n-k-1)} +  \gamma^k)
    < C \norm f \norm g,
  \end{equation}
  since $\gamma < 1$. This establishes the bound on
  $\norm{\varepsilon_n^{f,g}}$.

  To establish the claimed pointwise bound for
  $\abs{\varepsilon_n^{f,g}(x)}$, observe that by part (a), we have
  \(
    \abs[\big]{\varepsilon_k^{L^{n-k}[f]L^{n-k}[g]}(x)}
    \le \gamma^k \norm{L^{n-k}[f]L^{n-k}[g]}q(x)
    \le C \gamma^k \norm{f}\norm{g} q(x)
  \).
  Combining this with \eqref{eq:Harris_L2b-error} and the previously
  derived bounds on the inner products appearing there, we easily
  conclude.
\end{proof}

\section{Proof of Theorem~\ref{thm:diameter}}
\label{sec:proof_of_diameter}

In this section we prove Theorem~\ref{thm:diameter} which describes
the tail behaviour of the diameter of the critical cluster. We
also show several estimates that will later be used in the proof of
Theorem~\ref{thm:conv-to-exp}.

Let $\mathbb F$ be the set of all non-increasing functions
$f: \mathbb R\to [0,1)$. For every $f \in \mathbb F$ and $n\in \mathbb N$,
we introduce
\begin{equation}
  \label{eq:def-u_n^f}
  u_n^f(x)\defeq \begin{cases}
    E_x\big[1_{\set{N_n^+ \neq \emptyset}}
      \big(1-\prod_{u \in N_n^+}f(\varphi_u)\big)\big],
    \quad & \text{ for $x\ge h^*$}, \\
    0,   & \text{ for $x < h^*$}.
  \end{cases}
\end{equation}
Note that $u_n^f(x) \in [0,1]$ and by \eqref{eq:stoch_dom} it is
increasing in $x$. Taking $f\equiv 0$,
\begin{equation}
  \label{eq:uzeroP}
  u_n^0(x) = P_x[N_n^+ \neq \emptyset] ,
\end{equation}
which explains the relevance of this definition for the proof of
Theorem~\ref{thm:diameter}.

We will use the following inequality, which often allows us to consider
the special case $f\equiv 0$ only: For every $f\in \mathbb F$
\begin{equation}
  \label{eq:u_n^f-bound}
  u_n^0(x) (1-f(h^*)) \le u_n^f(x) \le u_n^0(x).
\end{equation}
Indeed, the second inequality follows directly from the definition of
$u_n^f$, since $f\ge 0$. To see the first one, note that since $f$ is
non-increasing, $\prod_{v \in N_n^+} f(\varphi_v) \le f(h^*)$ when
$N_n^+ \neq \emptyset$, and thus
\(
  E_x[1_{\set{N_n^+ \neq \emptyset}}(1-\prod_{v \in N_n^+}f(\varphi_v))]
  \ge E_x[1_{\set{N_n^+ \neq \emptyset}}(1-f(h^*))]
  = (1-f(h^*))u_n^0(x)
\).

Similarly to \eqref{eq:f-decomp} we decompose $u_n^f$ as
\begin{equation}
  \label{eq:anbn}
  u_n^f = \inp{\chi}{u_n^f}\chi + \beta[u_n^f]
\end{equation}
and define $a_n^f \defeq \inp{\chi}{u_n^f}$ and
$b_n^f \defeq \norm{\beta[u_n^f]}$, so that
\begin{equation}
  \norm{u_n^f}^2 = (a_n^f)^2 + (b_n^f)^2.
\end{equation}

Due to \eqref{eq:uzeroP}, in order to show Theorem~\ref{thm:diameter}, we
need precise asymptotic estimates on $a_n^f$ and $b_n^f$. These will be
proved step by step in the following several lemmas. The theorem is then
shown at the end of the section.

By Proposition~\ref{pro:cl23},
$P_x[\abs{\mathcal{C}_o \cap \mathbb{T}^+} = \infty] = 0$. Therefore,
$P_x[N_n^+ \neq \emptyset] \to 0$ as $n \to \infty$. As a consequence,
using also \eqref{eq:u_n^f-bound}, for every $f\in \mathbb F$,
\begin{equation}
  \label{eq:u_n^f(x)-limit}
  u_n^f(x) \to 0 \text{ as $n \to \infty$, pointwise and thus in
    $L^2(\nu )$}.
\end{equation}
As a consequence, for every $f\in \mathbb F$,
\begin{equation}
  \lim_{n\to\infty} a_n^f = \lim_{n\to\infty} b_n^f = 0.
\end{equation}

The following lemma provides a basic recursive relation for $u_n^f$,
based on the branching process representation. This relation will be
important to obtain a more precise description of the asymptotic
behaviour of $a_n^f$ and $b_n^f$ as $n \to \infty$.

\begin{lemma}
  \label{lem:u_n-recursion}
  For every non-increasing $f:\mathbb R\to [0,1)$ and $n\in \mathbb N$,
  \begin{equation}
    \label{eq:u_n-recursion-p1}
    u_n^f(x) = 1 - \Big(1-\frac{1}{d} L[u_{n-1}^f](x)\Big)^d
    \quad \text{for $x \in \R$}.
  \end{equation}
  As a consequence,
  \begin{equation}
    \label{eq:u_n^0-rec}
    u_n^f(x) = L[u_{n-1}^f](x) - g(L[u_{n-1}^f](x)),
  \end{equation}
  with $g(x) = 1-x-(1-d^{-1}x)^d$. This function satisfies, for
  $c_1,c_2\in (0,\infty)$,
  \begin{equation}
    \label{eq:u_n-recursion-p2}
    c_1 x^2 \le g(x) \le c_2 x^2
    \quad \text{for all $x \in [0,1]$.}
  \end{equation}
\end{lemma}

\begin{proof}
  For $x < h^*$, both sides of \eqref{eq:u_n-recursion-p1} are
  trivially zero by definition of $u_n^f$ and $L$. For $x \ge h^*$, by
  the branching process construction of $\varphi$, recalling
  $\mathcal F_n$ from \eqref{eq:filtrations}, for $n_0 < n$,
  \begin{equation}
    \begin{split}
      \label{eq:prod_n-f-recrusion}
      E_x\Big[\prod_{v \in N_n^+} f(\varphi_v) \Big| \F_{n_0}\Big]
      &= \prod_{v \in N^+_{n_0}}
      E_{\varphi_v}\Big[\prod_{w\in N^+_{n-n_0}}f(\varphi_w)\Big] \\
      &= \prod_{v \in N^+_{n_0}}
      \bigg(1-E_{\varphi_v}\Big[1-\prod_{w\in N^+_{n-n_0}}
          f(\varphi_w)\Big]\bigg) \\
      &= \prod_{v \in N^+_{n_0}}\big(1- u_{n-n_0}^f(\varphi_v)\big).
    \end{split}
  \end{equation}
  Therefore, setting $n_0 = 1$,
  \begin{equation}
    \label{eq:u_n-recursion_1}
    u_n^f(x)
    = E_x\Big[1 - E_x\Big[\prod_{v \in N^+_n} f(\varphi_v)
        \Big| \F_{1}\Big]\Big]
    = 1 - E_x\Big[\prod_{v \in N^+_{1}}
      \big(1- u_{n-1}^f(\varphi_v)\big)\Big].
  \end{equation}
  Using the conditional independence of the $(\varphi_v:v\in S_1^+)$
  given $\varphi_o = x$, and the fact that $u_n^f(x) = 0$ for $x < h^*$,
  \begin{equation}
    \label{eq:u_n-recursion_2}
    E_x\Big[\prod_{v \in N^+_1}\big(1- u_{n-1}^f(\varphi_v)\big)\Big]
    = E_x\Big[\prod_{v\in S^+_1}\big(1- u_{n-1}^f(\varphi_{v})\big)\Big]
    = \prod_{v\in S_1^+} E_x[1- u_{n-1}^f(\varphi_{v})].
  \end{equation}
  Using \eqref{eq:L_h-BP1},
  $E_x[u_{n-1}^f(\varphi_{v})] = d^{-1}L[u_{n-1}^f](x)$ for every
  $v\in S_1^+$. Together with \eqref{eq:u_n-recursion_1} and
  \eqref{eq:u_n-recursion_2}, this proves \eqref{eq:u_n-recursion-p1} and
  \eqref{eq:u_n^0-rec}. Inequality \eqref{eq:u_n-recursion-p2} is proved
  in \cite[Lemma 5.3]{CerLoc23}.
\end{proof}

The following lemma provides a rough lower bound for $u_n^0$ and $a_n^0$.

\begin{lemma}
  \label{lem:u_n^0-lower-bound}
  There exists a constant $c>0$ such that for all $n \ge 1$
  \begin{equation}
    \norm{u_n^0} \ge a_n^0 \ge c n^{-1}.
  \end{equation}
\end{lemma}

\begin{proof}
  The statement will follow if we show
  \begin{equation}
    \label{eq:u_n^0-lowbound}
    0 \le g(x) n^{-1} \le u_n^0(x) \qquad \text{for all $n\ge 1$}
  \end{equation}
  for some non-trivial function $g:\mathbb R\to[0, 1]$. To show
  \eqref{eq:u_n^0-lowbound} we use the estimate on the volume of
  $\mathcal C_o$ from Proposition~\ref{pro:cl23}. By this proposition and
  \eqref{eq:stoch_dom} there is a positive constant $c_1$ such that, for
  all $m \ge 1$ and $x\ge h^*$,
  \begin{equation}
    P_x[\abs{\mathcal{C}_o \cap \mathbb{T}^+} > m] \ge c_1 m^{-1/2}.
  \end{equation}
  Since $|\mathcal{C}_o \cap \mathbb{T}^+| = \sum_{k\geq 0} \abs{N_k^+}$,
  this implies for all $m,n\ge 1$
	\begin{align}
    c_1 m^{-1/2}
    &\le P_x\Big[\sum_{k \geq 0} \abs{N_k^+} > m, N_n^+ = \emptyset\Big]
    + P_x\Big[\sum_{k \geq 0} \abs{N_k^+} > m, N_n^+ \neq \emptyset\Big] \\
    &\le P_x\Big[\sum_{k = 0}^{n-1} \abs{N_k^+} > m\Big]
    + P_x[N_n^+ \neq \emptyset].
  \end{align}
  By the Markov inequality,
  \(
    P_x[\sum_{k = 0}^{n-1} \abs{N_k^+} > m]
    \le \frac{1}{m} \sum_{k=0}^{n-1} E_x[\abs{N_k^+}]
  \),
  where, by Proposition~\ref{pro:Harris_L2}, $E_x[\abs{N_k^+}] \le c_2(x)$
  for some $x$-dependent constant $c_2(x)>0$. Thus,
  \begin{equation}
    c_1m^{-1/2}\le c_2(x) \frac{n}{m} + P_x[N_n^+ \neq \emptyset].
  \end{equation}
  Recalling \eqref{eq:uzeroP} and choosing
  $m = \lfloor \delta^2 n^2 \rfloor$ with $\delta = 2 c_1 c_2(x)^{-1}$,
  it follows that
  \begin{equation}
    u_n^0(x) \ge \frac{c_1^2}{2 c_2(x)} n^{-1} \quad \text{for all
      $n \ge 1$},
  \end{equation}
  showing \eqref{eq:u_n^0-lowbound} and thus the lemma.
\end{proof}

The next lemma gives upper bounds on $a_n^f$ and $b_n^f$.

\begin{lemma}
  \label{lem:u_n^0-upper-bound}
  There is $c<\infty$ such that for every $f\in \mathbb F$
  and $n \ge 1$
  \begin{equation}
    \label{eq:a_n,b_n-upper-toshow}
    a_n^f \le c n^{-1}
    \qquad \text{and} \qquad
    b_n^f \le c n^{-2}.
  \end{equation}
\end{lemma}

\begin{proof}
  To prove the first statement in \eqref{eq:a_n,b_n-upper-toshow} we
  recall the recursive relation \eqref{eq:u_n^0-rec} and project it on
  $\spn\set{\chi}$. Using
  $\inp{\chi}{L[u_n^f]} = \inp{\chi}{u_n^f} = a_n^f$, and the lower bound
  on $g$ from~\eqref{eq:u_n-recursion-p2}, this yields
  \begin{equation}
    \label{eq:u_n^0-ineq1}
    a_{n+1}^f = \inp{\chi}{u_{n+1}^f}
    = \inp{\chi}{L[u_n^f]} - \inp{\chi}{g(L[u^f_n])}
    \le a_n^f - c_1\inp{\chi}{L[u^f_n]^2}.
  \end{equation}
  By Proposition~\ref{prop:L_h-chi-connection} and
  \eqref{eq:chi_non-dec}, $\chi(x) \ge c > 0$ for all $x\in [h^*, \infty)$.
  Hence,
  \(
    \inp{\chi}{L[u^f_n]^2} \ge c \inp{1}{L[u^f_n]^2}
    = c \norm{L[u^f_n]}^2 \ge c (a_n^f)^2
  \).
  Applied to \eqref{eq:u_n^0-ineq1}, this gives
  \begin{equation}
    a_{n+1}^f \le a_n^f - c (a_n^f)^2.
  \end{equation}
  This implies that $a_n^f$ is decreasing in $n$ and, after rearranging, also
  \begin{equation}
    c \le \frac{a_n^f - a_{n+1}^f}{(a_n^f)^2}
    \le \frac{a_n^f - a_{n+1}^f}{a_{n+1}^f a_n^f}
    = \frac{1}{a_{n+1}^f} - \frac{1}{a_n^f}.
  \end{equation}
  Summing this over $n$ running from $0$ to $n-1$ yields
  \begin{equation}
    c n
    \le \sum_{k=0}^{n-1} \Big(\frac{1}{a_{k+1}^f} - \frac{1}{a_{k}^f}\Big)
    = \frac{1}{a_n^f} - \frac{1}{a_0^f}
    \le \frac{1}{a_n^f},
  \end{equation}
  proving the first part of \eqref{eq:a_n,b_n-upper-toshow}.

  To prove the second part we project \eqref{eq:u_n^0-rec} onto the
  orthogonal complement of $\spn{\set{\chi}}$ and take norms. With the
  triangle inequality, this gives
  \begin{equation}
    \label{eq:b_n^0-ineq1}
    b_{n+1}^f \le \norm{\beta[L[u_n^f]]} + \norm{\beta[g(L[u_n^f])]}.
  \end{equation}
  By \eqref{eq:norm-beta},
  \(
    \norm{\beta[L[u_n^f]]} = \norm{L[\beta[u_n^f]]}
    \le \gamma \norm{\beta[u_n^f]} = \gamma b_n^f
  \).
  By the upper bound of $g$ from \eqref{eq:u_n-recursion-p2} and
  Proposition~\ref{prop:hypercontractivity}, since $\beta$ is a
  projection,
  \(
    \norm{\beta[g(L[u_n^f])]} \le \norm{g(L[u_n^f])}
    \le c_2 \norm{L[u_n^f]^2} \le c \norm{u_n^f}^2
  \).
  Applied to \eqref{eq:b_n^0-ineq1}, this gives
  \begin{equation}
    \label{eq:b_n-diffeq0}
    b_{n+1}^f \le \gamma b_n^f + c \norm{u_n^f}^2
    = \gamma b_n^f + c (a_n^f)^2 + c(b_n^f)^2.
  \end{equation}

  Taking now $f = 0$, since $b_n^0\to 0$ as $n \to \infty$, there is
  $\gamma' \in (\gamma ,1)$ and $n' <\infty$ such that
  $\gamma b_n^0 + c(b_n^0)^2 \le \gamma' b_n^0$ for all $n\ge n'$.
  Together with the first part of \eqref{eq:a_n,b_n-upper-toshow}, this
  implies
  \begin{equation}
    \label{eq:b_n-diffeq}
    b_{n+1}^0 \le \gamma' b_n^0 + c n^{-2} \quad \text{for $n \ge n'$}.
  \end{equation}
  Applying this recursively yields
  \begin{equation}
    b_{n+1}^0
    \le \gamma' \big(\gamma' b_{n-1}^0 + c (n-1)^{-2}\big) + c n^{-2}
    \le \dots \le (\gamma')^{n+1-n'} b_{n'}^0 + c
    \sum_{k=0}^{n-n'}(\gamma')^k (n-k)^{-2}.
  \end{equation}
  For $n\ge 2n'$, by splitting the sum at $k=n/2$, using also that
  $b_n^0 \le 1$ for all $n$, this can be bounded by
  $c (\gamma ')^{n/2} + c (n/2)^{-2} \le c n^{-2}$, proving the second
  half of \eqref{eq:a_n,b_n-upper-toshow} for $f = 0$.

  For a general $f\in \mathbb F$, by \eqref{eq:u_n^f-bound} and the
  already proven statements,
  $\norm{u_n^f}^2 \le \norm{u_n^0}^2 = (a_n^0)^2 + (b_n^0)^2 \le c n^{-2}$.
  Inserting this into \eqref{eq:b_n-diffeq0} gives
  \begin{equation}
    b_{n+1}^f \le \gamma b_n^f + c n^{-2},
  \end{equation}
  which looks like \eqref{eq:b_n-diffeq}. The second part of
  \eqref{eq:a_n,b_n-upper-toshow} for general $f\in \mathbb F$ then
  follows by the same arguments as for $f=0$.
\end{proof}

When $f\in \mathbb F$ is fixed, Lemmas~\ref{lem:u_n^0-lower-bound} and
\ref{lem:u_n^0-upper-bound} show that
$\lim_{n\to\infty}u_n^f/\norm{u_n^f} = \chi$ in $L^2(\nu )$ and
$c< n \norm{u_n^f} < c'$. In the special case $f=0$, this is almost
sufficient to prove Theorem~\ref{thm:diameter}, we only need to improve
the estimate on $\norm{u_n^f}$. In contrast, in the proof of
Theorem~\ref{thm:conv-to-exp} we will need to consider functions $f$
varying with $n$. There our estimates are not sufficient, since the lower
bound in Lemma~\ref{lem:u_n^0-lower-bound} holds only for $f=0$ and
cannot be true uniformly over $f\in \mathbb F$. We now provide tools
allowing us to deal with this case as well.

As it turns out, it is enough to prove the uniformity over a certain
family of non-increasing functions. Specifically, let
\begin{equation}
  \hat{\mathbb{F}} = \set{f_\lambda: \lambda \in [0, \infty)} \subset
  \mathbb F,
\end{equation}
where
\begin{equation}
  \label{eq:f_k-def}
  f_0 \equiv 0
  \quad\text{and}\quad
  f_\lambda(x) := \exp\Big( - \frac{\chi(x)}{\lambda} \Big)
  \quad\text{for $\lambda >0$.}
\end{equation}
To simplify the notation, we define (recall \eqref{eq:def-u_n^f},
  \eqref{eq:anbn})
\begin{equation}
  u_n^\lambda = u_n^{f_\lambda},
  \qquad a_n^\lambda = a_n^{f_\lambda} = \inp \chi {u_n^\lambda},
  \qquad b_n^\lambda = b_n^{f_\lambda} = \norm{\beta [u_n^\lambda]}.
\end{equation}

The first preliminary step in proving the uniformity over
$\hat{\mathbb F}$ is the following lemma. In the special case
$\lambda = 0$, this already follows from
Lemmas~\ref{lem:u_n^0-lower-bound} and~\ref{lem:u_n^0-upper-bound}.

\begin{lemma}
  \label{lem:b_n^k/a_n^k-boundedness}
  There is a constant $c < \infty$ so that
  \begin{equation}
    \label{eq:b_n^k/a_n^k-boundedness}
    b_n^\lambda \le c a_n^\lambda
    \quad \text{and} \quad
    \norm{u_n^\lambda} \le c{a_n^\lambda}
    \quad
    \text{for all $n\in \mathbb N_0$ and $\lambda \ge 0$}.
  \end{equation}
\end{lemma}

\begin{proof}
  As the case $\lambda = 0$ already follows from
  Lemmas~\ref{lem:u_n^0-lower-bound} and~\ref{lem:u_n^0-upper-bound}, it
  is enough to show \eqref{eq:b_n^k/a_n^k-boundedness} with
  $\lambda \ge 0$ replaced by $\lambda >0$.

  By \eqref{eq:b_n-diffeq0} from the proof of
  Lemma~\ref{lem:u_n^0-upper-bound},
  \begin{equation}
    b_{n+1}^\lambda \le \gamma b_{n}^\lambda + c \norm{u_n^\lambda}^2,
  \end{equation}
  with $c$ that is uniform over $f\in \mathbb F$.
  An iterative application of this inequality yields
  \begin{equation}
    \label{eq:b_n^k-iteration}
    b_{n+1}^\lambda \le \gamma^n b_0^\lambda
    + c \sum_{l = 0}^n \gamma^{n-l} \norm{u_l^\lambda}^2.
  \end{equation}

  To continue, we argue that
  \begin{equation}
    \label{eq:u_l^k-bound}
    \norm{u_l^\lambda} \le \lambda^{-1}
    \quad \text{for every $l \in \mathbb{N}_0$, $\lambda > 0$}.
  \end{equation}
  Indeed, by Lemma~\ref{lem:u_n-recursion},
  $u_{n+1}^\lambda(x) \le L[u_n^\lambda](x)$. Applying this recursively,
  taking norms, and using that the operator norm of $L$ is one, we obtain
  \begin{equation}
    \label{eq:u_l^k-bound-s1}
    \norm{u_l^\lambda} \le \norm{L^l[u_0^\lambda]} \le \norm{u_0^\lambda}.
  \end{equation}
  Moreover, since by definition
  $0\le u_0^\lambda = 1 - \exp(- \lambda^{-1} \chi) \le \lambda^{-1} \chi$,
  we know that
  $\norm{u_0^\lambda} \le \lambda^{-1} \norm{\chi} = \lambda^{-1}$.
  Together with \eqref{eq:u_l^k-bound-s1}, this shows
  \eqref{eq:u_l^k-bound}.

  Since also $b_0^\lambda \le \norm{u_0^\lambda} \le \lambda^{-1}$ and
  $\gamma <1$, \eqref{eq:b_n^k-iteration} and \eqref{eq:u_l^k-bound}
  imply
  \begin{equation}
    b_{n+1}^\lambda \le \gamma^n \lambda^{-1} + c \lambda^{-2}.
  \end{equation}
  On the other hand, by Lemma~\ref{lem:u_n^0-upper-bound},
  $b_n^\lambda \le c n^{-2}$, and thus
  \begin{equation}
    b_n^\lambda \le c \min\Big( \gamma ^n \lambda^{-1}
      +  \lambda^{-2},  \frac{1}{n^2}\Big).
  \end{equation}
  Since
  $1-f_\lambda(h^*) \ge c \lambda^{-1}$, \eqref{eq:u_n^f-bound} implies
  that
  \begin{equation}
    a_n^\lambda = \inp{u_n^\lambda}{\chi} \ge
    (1-f_\lambda (h^*))\inp{u_n^0}{\chi } \ge c \lambda^{-1}
    a_n^0 \ge c \lambda^{-1}n^{-1},
  \end{equation}
  where the last inequality follows from
  Lemma~\ref{lem:u_n^0-lower-bound}. Therefore,
  \begin{equation}
    \frac{b_n^\lambda}{a_n^\lambda}
    \le \frac{ c \min\big( \gamma ^n \lambda^{-1} + \lambda^{-2},
        n^{-2}\big)}{\lambda^{-1} n^{-1}}
    \le c \min\Big(c \gamma^n n +  \frac{n}{\lambda},
      \frac{\lambda}{n}\Big) \le c\gamma^n n + c,
  \end{equation}
  which is bounded from above by some constant $c$ independent of
  $\lambda$ or $n$. This proves the first statement. The second follows
  from the first one and
  $\norm{u_n^\lambda}^2 = (a_n^\lambda)^2 + (b_n^\lambda)^2$.
\end{proof}

We now improve the result of Lemma~\ref{lem:b_n^k/a_n^k-boundedness} and
show that $b_n^{\lambda} = o(a_n^{\lambda})$, uniformly in $\lambda$.

\begin{proposition}
  \label{prop:u_n^fk-L2-asymtote}
  Uniformly over $\lambda \ge 0$,
  \begin{equation}
    \label{eq:u_n^k_toshow}
    \lim_{n\to\infty}\frac{u_n^\lambda}{\inp{\chi}{u_n^\lambda}} = \chi
    \quad \text{in $L^2(\nu)$.}
  \end{equation}
\end{proposition}

\begin{proof}
  The proof consists of two steps. First we show that, for a suitable
  choice of $n_0(n)$, which slowly diverges with $n$,
  \begin{equation}
    \label{eq:firststep-L2}
    \frac{u_n^\lambda}{\inp{\chi}{u_{n-n_0(n)}^\lambda}} \to \chi \quad
    \text{in $L^2(\nu )$ as $n \to \infty$, uniformly in $\lambda \ge 0$}.
  \end{equation}
  In a second step, we then show that
  $\inp{\chi}{u_{n-n_0(n)}^\lambda}/\inp{\chi}{u_n^\lambda} \to 1$ as
  $n \to \infty$ uniformly in $\lambda \ge 0$, implying the claim of the
  proposition.

  By \eqref{eq:prod_n-f-recrusion} from the proof of
  Lemma~\ref{lem:u_n-recursion},
  \begin{equation}
    \begin{split}
      \label{eq:u_n^f-L2}
      &u_n^\lambda(x)
      = E_x\bigg[E_x\Big[1 - \prod_{v \in N^+_n} f_\lambda (\varphi_v)
          \Big| \F_{n_0}\Big]\bigg]
      = E_x\Big[1 - \prod_{v \in N^+_{n_0}}
        \big(1- u_{n-n_0}^\lambda(\varphi_v)\big)\Big] \\
      &= E_x\Big[\sum_{v\in N^+_{n_0}} u_{n-n_0}^\lambda(\varphi_v)\Big]
      - E_x\Big[\prod_{v \in N^+_{n_0}}
        \big(1- u_{n-n_0}^\lambda(\varphi_v)\big)
        - 1 + \sum_{v\in N^+_{n_0}} u_{n-n_0}^\lambda(\varphi_v) \Big].
    \end{split}
  \end{equation}
  By Proposition~\ref{pro:Harris_L2}(a), the first expectation on the
  right-hand side satisfies
  \begin{equation}
    \label{eq:u_n^f_1-L2}
    E_x\Big[\sum_{u \in N^+_{n_0}} u_{n-n_0}^\lambda(\varphi_u)\Big]
    = \chi(x) \inp{\chi}{u_{n-n_0}^\lambda}
    +\varepsilon_{n_0}^{\lambda,n-n_0}(x)
  \end{equation}
  with
  \(
    \norm{\varepsilon_{n_0}^{\lambda,n-n_0}} \leq
    \gamma^{n_0} \norm{u_{n-n_0}^\lambda}
  \).

  To estimate the second expectation we need the following inequality:
  For any finite index set $I$ and $a_i \in [0,1]$, $i \in I$,
  \begin{equation}
    \label{eq:aiinequality-L2}
    0 \leq \prod_{i\in I}(1-a_i) - 1
    + \sum_{i \in I}a_i \leq \frac{1}{2} \sum_{i\neq j\in I} a_i a_j.
  \end{equation}
  This inequality follows easily from Bonferroni inequalities for
  independent events $A_i$ with $P(A_i) = a_i$. Indeed,
  \begin{equation}
    1-\prod_{i\in I} (1-a_i) = P(\cup_{i\in I} A_i)
    \le \sum_{i\in I}P(A_i) = \sum_{i\in I}a_i
  \end{equation}
  implies the lower bound in \eqref{eq:aiinequality-L2}, and
  \begin{equation}
    1-\prod_{i\in I} (1-a_i) = P(\cup_{i\in I} A_i)
    \ge \sum_{i\in I}P(A_i) - \frac12\sum_{i\neq j\in I}P(A_i\cap A_j)
    = \sum_{i\in I}a_i - \frac{1}{2} \sum_{i \neq j\in I} a_i a_j
  \end{equation}
  implies the upper bound.
  Inequality \eqref{eq:aiinequality-L2} and
  Proposition~\ref{pro:Harris_L2}(b) imply
  \begin{equation}
    \begin{split}
      \label{eq:u_n^f_2-L2}
      0 &\leq E_x\Big[\prod_{u \in N^+_{n_0}}
        \big(1-u_{n-n_0}^\lambda(\varphi_u)\big) - 1
        + \sum_{u \in N^+_{n_0}} u_{n-n_0}^\lambda(\varphi_u) \Big]
      \\ & \leq c \chi(x)\inp{\chi}{u_{n-n_0}^\lambda}^2
      \inp{\chi}{\chi^2} n_0
      + \bar{\varepsilon}_{n_0}^{\lambda, n-n_0}(x),
    \end{split}
  \end{equation}
  with
  \(
    \norm{\bar{\varepsilon}_{n_0}^{\lambda, n-n_0}}
    \le C\norm{u_{n-n_0}^\lambda}^2
  \).

  Dividing \eqref{eq:u_n^f-L2} by $\inp{\chi}{u_{n-n_0}^\lambda}$ and
  combining it with \eqref{eq:u_n^f_1-L2} and \eqref{eq:u_n^f_2-L2}, we
  obtain
  \begin{equation}
    \begin{split}
      \label{eq:u_n-n0^f-ineq-L2}
      \norm[\bigg]{
        &\frac{u_{n}^\lambda}{\inp{\chi}{u_{n-n_0}^\lambda}} - \chi}
      \\&\leq
      \frac{\norm{\varepsilon_{n_0}^{\lambda,n-n_0}}}
      {\inp{\chi}{u_{n-n_0}^\lambda}}
      + \frac{c \norm{\chi\inp{\chi}{u_{n-n_0}^\lambda}^2
          \inp{\chi}{\chi^2} n_0}}
      {\inp{\chi}{u_{n-n_0}^\lambda}}
      + \frac{\norm{\bar{\varepsilon}_{n_0}^{\lambda, n-n_0}(x)}}
      {\inp{\chi}{u_{n-n_0}^\lambda}}
      \\&\leq
      \gamma^{n_0} \frac{\norm{u_{n-n_0}^\lambda}}
      {\inp{\chi}{u_{n-n_0}^\lambda}}
      +  c\inp{\chi}{\chi^2} \inp{\chi}{u_{n-n_0}^\lambda} n_0
      + c\frac{\norm{u_{n-n_0}^\lambda}}
      {\inp{\chi}{u_{n-n_0}^\lambda}}\norm{u_{n-n_0}^\lambda}
      \\&\leq
      c \gamma^{n_0}  + \frac{c \inp{\chi}{\chi^2}n_0}{n-n_0}
      + \frac{c}{n-n_0},
    \end{split}
  \end{equation}
  where in the last inequality we used that
  $\norm{u_{n}^\lambda}/\inp{\chi}{u_n^\lambda}$ is uniformly bounded by
  Lemma~\ref{lem:b_n^k/a_n^k-boundedness}, and both
  $\inp{\chi}{u_n^\lambda}$ and $\norm{u_n^\lambda}$ are uniformly
  bounded above by $c n^{-1}$ by Lemma~\ref{lem:u_n^0-upper-bound}. This
  establishes \eqref{eq:firststep-L2} with, for example,
  $n_0(n) = \floor{\log(n)}$.

  Next we show that \eqref{eq:u_n-n0^f-ineq-L2} also implies
  \begin{equation}
    \label{eq:u_n^f_lefttoshow-L2}
    \abs[\bigg]{
      \frac{\inp{\chi}{u_{n}^\lambda}}{\inp{\chi}{u_{n-n_0(n)}^\lambda}} - 1
    } \to 0
    \quad \text{as $n\to \infty$, uniformly in $\lambda \ge 0$}.
  \end{equation}
  To see this, we use that $1 = \inp{\chi}{\chi}$ and
  $\abs{\inp{\chi}{f}} \le \norm{f}$ to get
  \begin{equation}
    \begin{split}
      \abs[\bigg]{
        \frac{\inp{\chi}{u_{n}^\lambda}}{\inp{\chi}{u_{n-n_0(n)}^\lambda}}
        - 1
      }
      &= \abs[\bigg]{
        \frac{\inp{\chi}{u_{n}^\lambda}} {\inp{\chi}{u_{n-n_0(n)}^\lambda}}
        - \inp{\chi}{\chi}}
      = \abs[\bigg]{
        \inp[\bigg]{ \chi}{\frac{u_{n}^\lambda}
          {\inp{\chi}{u_{n-n_0(n)}^\lambda}} - \chi}  }
      \\ &\leq \norm[\bigg]{
        \frac{u_n^\lambda}{\inp{\chi}{u_{n-n_0(n)}^\lambda}} -\chi},
    \end{split}
  \end{equation}
  which by \eqref{eq:u_n-n0^f-ineq-L2} converges to zero uniformly in
  $\lambda$. This shows \eqref{eq:u_n^f_lefttoshow-L2}, which together
  with \eqref{eq:firststep-L2} finishes the proof.
\end{proof}

Using the hypercontractivity of $L$ from
Proposition~\ref{prop:hypercontractivity}, it is easy to slightly improve
Proposition~\ref{prop:u_n^fk-L2-asymtote}. This will be useful
to ensure that certain products still lie in~$L^2(\nu)$.

\begin{lemma}
  \label{lem:u_n^k-L2-asymptote}
  Uniformly over $\lambda \ge 0$,
  \begin{equation}
    \label{eq:rmk-u_n^k-Lp}
    \lim_{n\to\infty}\frac{u_n^\lambda}{\inp{\chi}{u_n^\lambda}} = \chi
    \quad \text{in $L^{5/2}(\nu)$ and pointwise over $[h^*,\infty)$}.
  \end{equation}
\end{lemma}

\begin{proof}
  We will first show \eqref{eq:rmk-u_n^k-Lp}. We set $p = 5/2$. By
  Lemma~\ref{lem:u_n-recursion},
  $u_n^\lambda = L[u_{n-1}^\lambda] + \sum_{l=2}^d c_l L[u_{n-1}^\lambda]^l$.
  Therefore, also using that $L[\chi] = \chi$,
  \begin{equation}
    \label{eq:rmk-L^p-conv-eq1}
    \norm[\Big]{
      \frac{u_n^\lambda}{\inp{\chi}{u_n^\lambda}} - \chi}_{L^{p}(\nu)}
    \le \norm[\bigg]{
      L \Big[\frac{u_{n-1}^\lambda}{\inp{\chi}{u_n^\lambda}}
        - \chi\Big]}_{L^{p}(\nu)}
    + \frac{c}{\inp{\chi}{u_n^\lambda}}
    \sum_{l=2}^{d}\norm{L[u_{n-1}^\lambda]^l}_{L^p(\nu)}.
  \end{equation}
  Since $0\le u_n^\lambda \le 1$, it holds that $0 \le L[u_n^\lambda] \le d$.
  Therefore,
  \begin{equation}
    \norm{L[u_{n-1}^\lambda]^l}_{L^p(\nu)}
    \le d^{l-2}\norm{L[u_{n-1}^\lambda]^2}_{L^p(\nu)}
    = d^{l-2}\norm{L[u_{n-1}^\lambda]}_{L^{2p}(\nu)}^{2}
    \le c \norm{u_{n-1}^\lambda}_{L^2(\nu)}^2,
  \end{equation}
  where in the last inequality we used
  Proposition~\ref{prop:hypercontractivity}(a) with $q = 2p = 5$ and
  $p = 2$. Therefore, using also \eqref{eq:u_n^f_lefttoshow-L2} (with
    $n_0 = 1$), Lemma~\ref{lem:b_n^k/a_n^k-boundedness}, and the fact
  that $\norm{u_n^\lambda} \le \norm{u_n^0} \to 0$ as $n\to \infty$, we get
  \begin{equation}
    \label{eq:rmk-L^p-conv-eq2}
    \frac{\norm{L[u_{n-1}^\lambda]^l}_{L^p(\nu)}}{\inp{\chi}{u_n^\lambda}}
    \le
    \frac{\norm{u_{n-1}^\lambda}^2_{L^2(\nu)}}{c\inp{\chi}{u_{n-1}^\lambda}}
    \to 0
    \quad \text{as $n \to \infty$ uniformly in $\lambda \ge 0$}.
  \end{equation}
  Using Proposition~\ref{prop:hypercontractivity}(a) again for the
  remaining term on the right-hand side of \eqref{eq:rmk-L^p-conv-eq1},
  \begin{equation}
    \label{eq:rmk-L^p-conv-eq3}
    \norm[\bigg]{
      L \Big[\frac{u_{n-1}^\lambda}{\inp{\chi}{u_n^\lambda}} - \chi\Big]
    }_{L^{p}(\nu)}
    \le d \norm[\bigg]{
      \frac{u_{n-1}^\lambda}{\inp{\chi}{u_n^\lambda}} - \chi
    }_{L^{2}(\nu)}
    \to 0
    \quad \text{as $n \to \infty$ uniformly in $\lambda \ge 0$}.
  \end{equation}
  Here, the convergence follows from
  Proposition~\ref{prop:u_n^fk-L2-asymtote}. Combining
  \eqref{eq:rmk-L^p-conv-eq1} with \eqref{eq:rmk-L^p-conv-eq2} and
  \eqref{eq:rmk-L^p-conv-eq3} shows \eqref{eq:rmk-u_n^k-Lp}.

  We now show the pointwise convergence. By similar steps as in
  \eqref{eq:rmk-L^p-conv-eq1}, using Lemma~\ref{lem:L-pointwise}, for
  $x \ge h^*$,
  \begin{equation}
    \begin{split}
      \label{eq:rmk-ptw-conv-eq1}
      \abs[\Big]{
        \frac{u_n^\lambda(x)}{\inp{\chi}{u_n^\lambda}} - \chi(x)}
      &\le \abs[\bigg]{
        L \Big[\frac{u_{n-1}^\lambda}{\inp{\chi}{u_n^\lambda}}
          - \chi\Big](x)}
      + \frac{1}{\inp{\chi}{u_n^\lambda}}\sum_{l=2}^{d}
      \abs{L[u_{n-1}^\lambda](x)}^l
      \\&\le \norm[\bigg]{
        \frac{u_{n-1}^\lambda}{\inp{\chi}{u_n^\lambda}} - \chi} q(x)
      + \frac{1}{\inp{\chi}{u_n^\lambda}}
      \sum_{l=2}^{d}\norm{u_{n-1}^\lambda}^l q(x)^l.
    \end{split}
  \end{equation}
  Using \eqref{eq:rmk-L^p-conv-eq2} and \eqref{eq:rmk-L^p-conv-eq3}, it
  is immediate that the right-hand side of \eqref{eq:rmk-ptw-conv-eq1}
  converges to $0$ as $n \to \infty$, uniformly in $\lambda \ge 0$. This
  shows the pointwise convergence and finishes the proof.
\end{proof}

The following lemma provides the final ingredient for the proof of
Theorem~\ref{thm:diameter}. Its result allows us to determine the exact
asymptotic behaviour of $a_n^{\lambda}$ as $n \to \infty$.

\begin{lemma}
  \label{lem:a^f-asymtote}
  Let $C_1$ be as in \eqref{eq:C_1-def}. Then,
  \begin{equation}
    \label{eq:a^f-asymptote-toprove}
    \lim_{n\to\infty}\frac{1}{n}
    \Big(\frac{1}{a_n^\lambda} - \frac{1}{a_0^\lambda}\Big)
    = C_1^{-1} \quad\text{uniformly in $\lambda \ge 0$}.
  \end{equation}
\end{lemma}

\begin{proof}
  We show that
  \begin{equation}
    \label{eq:a_f-difference-eq}
    \lim_{n\to\infty}\Big(\frac{1}{a_{n}^\lambda}
      -\frac1{a_{n-1}^\lambda}\Big)
    =\lim_{n\to\infty}\frac{a_{n-1}^\lambda
      - a_n^\lambda}{a_{n-1}^\lambda a_{n}^\lambda}
    = C_1^{-1}
  \end{equation}
  uniformly in $\lambda \ge 0$. The claim of the lemma then follows from
  \eqref{eq:a_f-difference-eq} by telescoping the difference therein.

  By Lemma~\ref{lem:u_n-recursion} again, this time writing the prefactor
  of the quadratic term explicitly,
  \begin{equation}
    \label{eq:a^f-1}
    \begin{split}
      a_n^\lambda = \inp{\chi} {u_n^\lambda}
      &= \inp{\chi}{L[u_{n-1}^\lambda]}
      - \binom{d}{2}\frac{1}{d^2}\inp{\chi}{L[u_{n-1}^\lambda]^2}
      - \sum_{l=3}^d c_l \inp{\chi}{L[u_{n-1}^\lambda]^l}.
    \end{split}
  \end{equation}
  Using that $L$ is self-adjoint and $L[\chi]=\chi$, we get
  \(
    \inp{\chi}{L[u_{n-1}^\lambda]}
    = \inp{\chi}{u_{n-1}^\lambda} = a_{n-1}^\lambda
  \).
  Hence, after dividing by $(a_{n-1}^\lambda)^2$, \eqref{eq:a^f-1} implies
  \begin{equation}
    \begin{split}
      \label{eq:a^f-frac}
      \frac{a_{n-1}^\lambda - a_{n}^\lambda}{(a_{n-1}^\lambda)^2}
      &= \binom{d}{2} d^{-2} \inp[\Big]{\chi}
      {\frac{L[u_{n-1}^\lambda]^2}{(a_{n-1}^\lambda)^2}}
      + \sum_{l=3}^d c_l \inp[\Big]{\chi}
      {\frac{L[u_{n-1}^\lambda]^l}{(a_{n-1}^\lambda)^2}}
      \\&= \binom{d}{2} d^{-2}
      \inp[\Big]{\chi}{L\Big[\frac{u_{n-1}^\lambda}{a_{n-1}^\lambda}\Big]^2}
      + \sum_{l=3}^d c_l
      \inp[\Big]{\chi}
      {L\Big[\frac{u_{n-1}^\lambda}{a_{n-1}^\lambda}\Big]^2
        L[u_{n-1}^\lambda]^{l-2}}.
    \end{split}
  \end{equation}
  By Propositions~\ref{prop:u_n^fk-L2-asymtote}
  and~\ref{prop:hypercontractivity},
  $L[u_{n-1}^\lambda/a_{n-1}^\lambda]^2 \to \chi^2$ in $L^2(\nu )$,
  uniformly in $\lambda$.
  By \eqref{eq:u_n^f(x)-limit} and \eqref{eq:u_n^f-bound},
  $0 \leq u_{n}^\lambda \leq u_n^0 \to 0$ in $L^2(\nu)$ as $n \to \infty$.
  Therefore, using that $L[u_{n-1}^\lambda] \le d$ and
  the boundedness of $L$,
  $L[u_{n-1}^\lambda]^{l-2} \le d^{l-3}L[u_{n-1}^\lambda] \to 0$ in
  $L^2(\nu )$ uniformly in $\lambda$. Using this in \eqref{eq:a^f-frac}
  then yields
  \begin{equation}
    \label{eq:a_n^f-diffeq2}
    \lim_{n\to\infty}
    \frac{a_{n-1}^\lambda - a_{n}^\lambda}{(a_{n-1}^\lambda)^2}
    = \binom{d}{2}d^{-2}\inp{\chi}{\chi^2}
    = C_1^{-1}
    \quad \text{uniformly in $\lambda \ge 0$}.
  \end{equation}
  From this \eqref{eq:a_f-difference-eq} follows, if we show
  \begin{equation}
    \label{eq:a^f-lefttoproof}
    \lim_{n\to\infty}\frac{a_n^\lambda}{a_{n-1}^\lambda} = 1
    \quad \text{uniformly in $\lambda \ge 0$.}
  \end{equation}
  To see this, note that by multiplying \eqref{eq:a_n^f-diffeq2} by
  $a_{n-1}^\lambda$,
  \begin{equation}
    1 - a_n^\lambda/a_{n-1}^\lambda
    = a_{n-1}^\lambda C_1^{-1} + o(a_{n-1}^\lambda)
    = o(1),
  \end{equation}
  where the $o(1)$ term is independent of $\lambda$ since
  $0\le a_n^\lambda \leq a_n^0 \to 0$ for $n \to \infty$. This shows
  \eqref{eq:a^f-lefttoproof}, and together with \eqref{eq:a_n^f-diffeq2}
  completes the proof of \eqref{eq:a_f-difference-eq}.
\end{proof}

We are now ready to prove Theorem~\ref{thm:diameter}.

\begin{proof}[Proof of Theorem~\ref{thm:diameter}]
  Recall from \eqref{eq:uzeroP} that
  $u_n^0(x) = P_x[N^+_n \neq \emptyset ]$. Therefore, by
  Proposition~\ref{prop:u_n^fk-L2-asymtote} and
  Lemma~\ref{lem:u_n^k-L2-asymptote},
  \begin{equation}
    \label{eq:thm-diameter-step1}
    \lim_{n\to\infty}
    \frac{P_{\boldsymbol{\cdot}}[N^+_n \neq \emptyset]}{\inp{\chi}{u_n^0}}
    = \chi
    \quad\text{in $L^2(\nu)$ and pointwise}.
  \end{equation}
  Moreover, by Lemma~\ref{lem:a^f-asymtote},
  \begin{equation}
    \label{eq:an0}
    {\inp{\chi}{u_n^0}} = a_n^0 = C_1 n^{-1} (1+o(1)).
  \end{equation}
  Combining these two facts one obtains \eqref{eq:diameter-1}, and by
  integrating over $\varphi (o)$ which is $\nu$-distributed also
  \eqref{eq:diameter-3}.

  To show the last statement of the theorem, let
  $x_1, \dots, x_{d+1} = \bar o$ be the neighbours of $o$ in $\mathbb T$,
  and let $A_{n,i}$ be the event ``the subtree of $x_i$ intersects $N_n$''.
  By the branching process construction, for every $n\ge 1$, the events
  $A_{n,1}, \dots, A_{n,d+1}$ are independent and have the same
  probability $p=p(n,x)$. Further,
  $P_x[N_n\neq \emptyset] = P_x[\cup_{i=1}^{d+1} A_{n,i}] = 1-(1-p)^{d+1}$
  and similarly $P_x[N^+_n\neq \emptyset] = 1-(1-p)^{d}$ and thus
  \begin{equation}
    P_x[N_n\neq \emptyset] = 1-(1-P_x[N_n^+\neq
        \emptyset])^{\frac{d+1}{d}} = \frac{d+1}{d} P_x[N_n^+ \neq
      \emptyset] + O( P_x[N_n^+ \neq \emptyset]^2).
  \end{equation}
  This directly implies the statement for $P_x[N_n\neq \emptyset]$. The
  statement for $P[N_n \neq \emptyset]$ is again obtained by integration
  over $\varphi_{o}$.
\end{proof}

\section{Proof of Theorem~\ref{thm:conv-to-exp}}
\label{sec:proof-of-thm2}
The goal of this section is twofold: We show
Theorem~\ref{thm:conv-to-exp}, and then, in
Section~\ref{subsec:further-cond-res}, provide further results
concerning the behaviour of $\mathcal C_o$ conditioned on
$\set{N_n^+ \neq \emptyset}$. These results are, more or less, direct
consequences of Theorems~\ref{thm:diameter} and~\ref{thm:conv-to-exp},
and will later be used to show Theorem~\ref{thm:inv-principle}.

We start with a preparatory lemma which is a special case of
Theorem~\ref{thm:conv-to-exp} for functions $f$ orthogonal to $\chi$.

\begin{lemma}
  \label{lem:conv-to-exp_prep}
  For every $f \in L^2(\nu)$ such that $\inp{\chi}{f} = 0$ there
  is a sequence $\varepsilon_n^f \in L^2(\nu)$ converging to zero in
  $L^2(\nu)$ and pointwise,
  such that for
  every $\delta > 0$, $x \ge h^*$, and $n\ge 1$,
  \begin{equation}
    \label{eq:conv-to-exp_prep}
    P_x\Big[\abs[\Big]{\frac{1}{n}\sum_{v \in N^+_n} f(\varphi_v)}
      > \delta\, \Big|\, N^+_n\neq\emptyset\Big]
    \le c \delta ^{-2} \varepsilon_n^f(x).
  \end{equation}
\end{lemma}

\begin{proof}
  By the conditional Markov inequality
  \begin{equation}
    P_x\Big[\abs[\Big]{\frac 1n \sum_{v \in N^+_n} f(\varphi_v)}
      >\delta \Big|
      N^+_n\neq\emptyset\Big]
    \leq \delta^{-2}
    E_x\Big[\Big(\frac{1}{n} \sum_{v \in N^+_n} f(\varphi_v)\Big)^2
      \Big| N^+_n\neq\emptyset\Big].
  \end{equation}
  The conditional expectation on the right-hand side satisfies
  \begin{equation}
    E_x\Big[\Big(\frac{1}{n} \sum_{v \in N^+_n} f(\varphi_v)\Big)^2
      \Big| N^+_n\neq\emptyset\Big]
    = \Big(\frac{1}{n P_x[N^+_n\neq\emptyset]}\Big)
    \Big(\frac{1}{n}E_x\Big[\Big(\sum_{v \in N^+_n}
          f(\varphi_v)\Big)^2\Big]\Big).
  \end{equation}
  By Proposition~\ref{pro:Harris_L2}(a,b), since $\inp{\chi}{f} = 0$,
  \begin{equation}
    \frac 1n E_x\Big[\Big(\sum_{v \in N_n^+} f(\varphi_v)\Big)^2\Big]
    = \varepsilon_n^f(x),
  \end{equation}
  with $\varepsilon_n^f \to 0$ in $L^2$ and pointwise. Further, by the
  stochastic domination \eqref{eq:stoch_dom}, using
  Theorem~\ref{thm:diameter},
  \(
    (n P_x[N^+_n\neq\emptyset])^{-1}
    \le (n P_{h^*}[N^+_n \neq \emptyset])^{-1} \le c
  \)
  and the statement of the lemma follows.
\end{proof}

\begin{proof}[Proof of Theorem~\ref{thm:conv-to-exp}]
  We first consider the special case $f = \chi$; the general case will
  then easily follow using Lemma~\ref{lem:conv-to-exp_prep}.

  We start by showing that for every $\alpha > 0$,
  \begin{equation}
    \label{eq:conv-to-exp_p1}
    E\big[ e^{-\alpha Z_n^{\chi ,x}}\big]
    = \frac{C_1}{C_1+\alpha} + \varepsilon_n^\alpha(x),
  \end{equation}
  with $\varepsilon_n^\alpha \to 0$ pointwise and in $L^2(\nu)$. Using
  the definitions \eqref{eq:def-u_n^f}, \eqref{eq:f_k-def} of $u_n^f$ and
  $f_\lambda$, setting as before $u_n^\lambda = u_n^{f_\lambda}$, and
  recalling \eqref{eq:uzeroP}, we obtain
  \begin{equation}
    \begin{split}
      \label{eq:Px[exp(...)]-start}
      E\big[ e^{-\alpha Z_n^{\chi ,x}}\big]
      = E_x\Big[\exp\Big(-\frac{\alpha}{n}&\sum_{v \in
            N_n^+}\chi(\varphi_v)\Big) \Big| N^+_n\neq\emptyset\Big]
      = 1 - \frac{u_n^{n/\alpha}(x)}{u_n^0(x)}.
    \end{split}
  \end{equation}
  By Lemma~\ref{lem:u_n^k-L2-asymptote} and \eqref{eq:an0},
  \begin{equation}
    u_n^0(x) = C n^{-1} \chi (x) (1 + \varepsilon (x)),
  \end{equation}
  where $\varepsilon \to 0$ pointwise. Since
  Lemma~\ref{lem:u_n^k-L2-asymptote} holds uniformly in $\lambda$, we can
  apply it with $\lambda=n/\alpha$ to obtain
  \begin{equation}
    \label{eq:unn}
    u_n^{n/\alpha}(x) = a_n^{n/\alpha} \chi (x) (1 + \tilde\varepsilon (x)),
  \end{equation}
  again with $\tilde \varepsilon \to 0$ pointwise. Moreover, by
  Lemma~\ref{lem:a^f-asymtote}, again using the uniformity in $\lambda$,
  \begin{equation}
    \label{eq:na_n}
    (n a_n^{n/\alpha})^{-1} = (n a_0^{n/\alpha})^{-1} + C_1^{-1} + o(1)
    \quad \text{as $n \to \infty$}.
  \end{equation}
  By the definitions of $a_n^\lambda$, $u_n^\lambda$ and $f_\lambda$, and
  by the monotone convergence theorem,
  \begin{equation}
    \begin{split}
      \label{eq:na_n^n-asym}
      n a_0^{n/\alpha} &= n \inp{\chi}{u_0^{n/\alpha}}
      = n \inp[\big]{\chi}
      { E_{\boldsymbol{\cdot}}\big[1-f_{n/\alpha}(\varphi_o) \big]}
      \\&= n \inp\chi {1-f_{n/\alpha}}
      = \inp[\big] \chi
      {n\big(1- e^{-\alpha \chi(\boldsymbol{\cdot})/n}\big)}
      \xrightarrow{n\to\infty}
      \alpha \inp \chi \chi   = \alpha.
    \end{split}
  \end{equation}
  Inserting this into~\eqref{eq:na_n} yields
  \begin{equation}
    \label{eq:na_n2}
    a_n^{n/\alpha} =  \frac{\alpha C_1(1+o(1))}{n (\alpha +  C_1)}
    \quad \text{as $n \to \infty$}.
  \end{equation}
  Combining \eqref{eq:Px[exp(...)]-start}--\eqref{eq:unn},
  \eqref{eq:na_n2} gives
  \begin{equation}
    1- E\big[e^{-\alpha Z_n^{\chi ,x}}\big]
    = \frac {\alpha }{\alpha +C_1} (1+o(1)),
  \end{equation}
  which shows the pointwise convergence in \eqref{eq:conv-to-exp_p1}. The
  $L^2(\nu)$-convergence follows by the dominated convergence theorem,
  since the left-hand side of \eqref{eq:conv-to-exp_p1} is bounded by $1$.

  By Lévy's continuity theorem for the Laplace transform,
  \eqref{eq:conv-to-exp_p1} implies the statement \eqref{eq:conv-to-exp}
  of the theorem in the case $Z^{f,x}_n$ with $f = \chi$. For general
  $f \in L^2(\nu)$ we write $f = \inp{\chi}{f}\chi + \beta[f]$ as usual.
  Then $Z_n^{f,x} = \inp \chi f Z_n^{\chi,x} + Z_n^{\beta [f],x}$. The
  statement for $Z_n^{f,x}$ then directly follows, since the first
  summand converges to an exponential random variable with mean
  $C_1^{-1}\inp \chi f$, by the first step of the proof, and the second
  summand converges to $0$ in probability, by
  Lemma~\ref{lem:conv-to-exp_prep}, since $\inp{\chi}{\beta[f]} = 0$.

  We now show the statement of the theorem for $Z_n^{f}$. Let $\nu_n$ be the
  law of $\varphi_o$ under $P[\,\cdot\, | N_n^+ \neq \emptyset]$. By
  integrating \eqref{eq:conv-to-exp_p1} over $\nu_n (\D x)$
  \begin{equation}
    \label{eq:e^-Z_n^chi-1}
    E\big[e^{-Z_n^\chi }\big]
    = \frac{C_1}{C_1+\alpha} + \int \varepsilon_n^\alpha(x) \nu_n(\D x).
  \end{equation}
  We need to show that the last integral is $o(1)$. Observe that
  \(
    \nu_n(\D x)
    = P[\varphi_o \in \D x | N_n^+ \ne \emptyset]
    = P[\varphi_o \in \D x, N_n^+ \ne \emptyset]P[N_n^+ \neq \emptyset]^{-1}
    = P_x[N_n^+ \neq \emptyset]P[N_n^+ \neq \emptyset]^{-1} \nu(\D x)
  \).
  Therefore, using the Cauchy--Schwarz inequality and that
  $\varepsilon_n^\alpha \to 0$ in $L^2(\nu)$, the integral in
  \eqref{eq:e^-Z_n^chi-1} is bounded above by
  \begin{equation}
    \label{eq:e^-Z_n^chi-2}
    o(1) \Big(
      \int \frac{P_x[N_n^+ \neq \emptyset]^2}{P[N_n^+ \neq \emptyset]^2}
      \nu(\D x) \Big)^{1/2}.
  \end{equation}
  By \eqref{eq:thm-diameter-step1}, $P_x[N_n^+ \neq \emptyset]P[N_n^+ \neq
  \emptyset]^{-1} = \chi(x)\inp{1}{\chi}^{-1} + \varepsilon_n(x)$, with
  $\varepsilon_n(x) \to 0$ in $L^2(\nu)$ as $n \to \infty$. This implies that
  the integral in \eqref{eq:e^-Z_n^chi-2} is $O(1)$ and together with
  \eqref{eq:e^-Z_n^chi-1} and again Lévy's continuity theorem shows statement
  \eqref{eq:conv-to-exp} for $Z_n^f$ with $f=\chi$. For $Z_n^f$ and general
  $f \in L^2(\nu)$ the statement then follows by integrating
  \eqref{eq:conv-to-exp_prep} over $\nu_n(\D x)$ and applying similar
  arguments as above.
\end{proof}

\subsection{Further results for the conditioned model}
\label{subsec:further-cond-res}

The following few results which all concern the model conditioned to
$\set{N_n^+ \neq \emptyset}$ will later be used in the proof of
Theorem~\ref{thm:inv-principle}. The first one explains the role of the
measure $Q_x$ introduced using the spine construction in
Section~\ref{ss:spined-bp}. Similar results are well established in
branching processes literature,  see for instance
\cite[Theorem~5]{HarHar07} or \cite[Theorem 4]{ChaRou90}. Proposition~1.5
in \cite{Pow19} can be seen as the analogous result for branching
diffusion on bounded domains.

\begin{proposition}
  \label{prop:P-conditioned-is-Q}
  For all $K \in \mathbb{N}$, $x \ge h^*$ and $B \in \F_{K}$ (see
    \eqref{eq:filtrations})
  \begin{equation}
    \label{eq:P-conditioned-is-Q_toshow}
    \lim_{n\to\infty} P_x[B | N^+_n \neq \emptyset] = \Q_x[B].
  \end{equation}
\end{proposition}

\begin{proof}
  We follow similar steps as in the proof of \cite[Proposition~1.5]{Pow19}.
  For every $n\ge K$,
  \begin{equation}
    \label{eq:PBtoq} P_x[B | N^+_n \neq \emptyset]
    = E_x\Big[ \frac{1_B P_x[N^+_n\neq\emptyset | \F_K]}
      {P_x[N^+_n\neq\emptyset]}\Big ].
  \end{equation}
  Writing $N^+_K = \set{v_1, \dots, v_{\abs{N^+_K}}}$ and defining the
  events $A_i = \{v_i$ is root of a subtree in $\mathcal C_o$ with height
    of at least $n-K\}$,
  \begin{equation}
    \label{eq:aidec}
    \set{N^+_n\neq\emptyset} = \bigcup_{k=1}^{\abs{N_K^+}}A_k
    = \bigcup_{k=1}^{\abs{N_K^+}}\Big(A_k  \cap \bigcap_{j<k}A_j^c\Big),
  \end{equation}
  where the last union is disjoint. Since the $A_i$ are independent
  conditionally on $\mathcal F_K$,
  \begin{equation}
    P_x[N^+_n\neq\emptyset | \F_K]
    = \sum_{i=1}^{\abs{N^+_K}} P_{\varphi_{v_i}}[N_{n-K}^+\neq\emptyset]
    \prod_{j<i} P_{\varphi_{v_j}}[N^+_{n-K}=\emptyset].
  \end{equation}
  By Theorem~\ref{thm:diameter},
  \(
    P_{\varphi_{v_i}}[N^+_{n-K}\neq \emptyset]
    \sim C_1(n-K)^{-1}\chi(\varphi_{v_i})\to 0
  \)
  and $P_x[N^+_n\neq\emptyset] \sim C_1n^{-1}\chi(x)$ as $n\to\infty$.
  Therefore,
  \begin{equation}
    \label{eq:Px[Nn>0|FK]}
    \lim_{n\to\infty}
    \frac{1_B P_x[N^+_n\neq\emptyset| \F_K]}{P_x[N^+_n\neq\emptyset]}
    = 1_B \frac{\sum_{v \in N_K}\chi(\varphi_v)}{\chi(x)},
    \qquad P_x\text{-a.s.}
  \end{equation}
  In order to insert this into \eqref{eq:PBtoq}, we use the generalised
  dominated convergence theorem. To this end we bound the fraction on the
  left-hand side of \eqref{eq:Px[Nn>0|FK]} by a function $g_n$ satisfying
  $g_n \to g$ in $L^1(P_x)$ and $P_x$-a.s.: Using \eqref{eq:aidec} and
  Theorem~\ref{thm:diameter} (or \eqref{eq:thm-diameter-step1},
    \eqref{eq:an0} from its proof),
  \begin{equation}
    P_x[N^+_n\neq\emptyset |\F_K]
    \leq \sum_{v \in N^+_K}P_{\varphi_v}[N^+_{n-K}\neq\emptyset]
    = \sum_{v \in N^+_K} C_1(n-K)^{-1}
    (\chi(\varphi_v)+\bar{\varepsilon}_{n-K}(\varphi_v))
  \end{equation}
  with $\bar{\varepsilon}_n \to 0$ pointwise and in $L^2(\nu)$ as
  $n \to \infty$. For the denominator, again by
  Theorem~\ref{thm:diameter},
  $P_x[N^+_n\neq\emptyset] \geq P_{h^*}[N^+_n\neq\emptyset] \geq cn^{-1}$
  for some constant $c$. Thus, for $n>K$,
  \begin{equation}
    \begin{split}
      \frac{1_B P_x[N^+_n\neq\emptyset | \F_K]}
      {P_x[N^+_n\neq\emptyset]}
      &\leq \frac{cn}{n-K}\sum_{v \in N_K^+}
      (\chi(\varphi_v) + \bar{\varepsilon}_{n-K}(\varphi_v))
      \\&\le c' \sum_{v \in N_K^+} (\chi(\varphi_v)
        + \bar{\varepsilon}_{n-K}(\varphi_v)) =: g_n,
    \end{split}
  \end{equation}
  which converges a.s.~to $g := c' \sum_{v \in N_K^+} \chi(\varphi_v)$.
  Moreover, by Proposition~\ref{pro:Harris_L2},
  \begin{equation}
    \label{eq:l1dom}
    E_x[\abs{g_n-g}]
    \le E_x\big[
      c' \sum_{v\in N_K^+} \abs{ \bar{\varepsilon}_{n-K}(\varphi_v)}
      \big]
    =\inp{\chi}{ \abs{\bar{\varepsilon}_{n-K}}}\chi(x)
    + \varepsilon_K^{\abs{\bar{\varepsilon}_{n-K}}}(x).
  \end{equation}
  Since $\norm{\bar \varepsilon_n}\to 0$, the bounds established in
  \eqref{eq:errorsa} imply that the right-hand side of \eqref{eq:l1dom}
  converges to $0$, that is $g_n \to g$ in $L^1(P^x)$. Therefore,
  by the generalised dominated convergence theorem,
  using \eqref{eq:PBtoq}, \eqref{eq:Px[Nn>0|FK]},
  \begin{equation}
    \lim_{n\to\infty }P_x[B | N^+_n\neq\emptyset]
    = E_x\Big[1_B \frac{\sum_{v\in N_K}\chi(\varphi_v)}{\chi(x)}\Big]
    = Q_x[B],
  \end{equation}
  where the last equality follows from Lemma~\ref{lem:many-to-few_general}
  with $k=1$ and $Y(v_1)=1_B \chi(\varphi_{v_1})$.
\end{proof}

The following lemma and its corollary follow almost directly from the
results established in the proof of Theorem~\ref{thm:conv-to-exp}.

\begin{lemma}
  \label{lem:sum_f_average}
  For every $f \in L^2(\nu)$,  $x \ge h^*$ and $\delta >0$,
  \begin{equation}
    \label{eq:sum_f_average-toshow}
    \lim_{n\to\infty}P_x\bigg[\abs[\bigg]{
        \frac{\sum_{v\in N_n^+} f(\varphi_v)}
        {\sum_{v \in N_n^+}\chi(\varphi_v)}
        - \inp{\chi}{f}}
      > \delta  \,\bigg|\,
      N_n^+\neq \emptyset \bigg] =0.
  \end{equation}
\end{lemma}

\begin{proof}
  Observe that
  \begin{equation}
    \begin{split}
      \label{eq:sum_f_avg-simplify}
      \frac{\sum_{v\in N^+_n} f(\varphi_v)}
      {\sum_{v \in N^+_n}\chi(\varphi_v)}
      - \inp{\chi}{f}
      = \frac{\sum_{v\in N^+_n}(f(\varphi_v)
          -\inp{\chi}{f}\chi(\varphi_v))}
      {\sum_{v \in N^+_n}\chi(\varphi_v)}
      = \frac{n^{-1}\sum_{v\in N^+_n} \beta[f](\varphi_v)}
      {n^{-1}\sum_{v \in N^+_n}\chi(\varphi_v)}.
    \end{split}
  \end{equation}
  For any $A,B\in \mathbb R$ and $\varepsilon ,\delta >0$,
  \(
    \set{\abs{A/B} > \varepsilon} \subset \set{\abs A>\delta}
    \cup \set{\abs B <\delta/\varepsilon}
  \),
  which together with \eqref{eq:sum_f_avg-simplify} allows us to bound the
  left-hand side of \eqref{eq:sum_f_average-toshow} by
  \begin{equation}
    \label{eq:sum_bound}
    P_x\Big[\abs[\Big]{n^{-1}\sum_{v \in N^+_n}\beta[f](\varphi_v)} >
      \bar{\delta}_n
      \,\Big|\, N^+_n\neq\emptyset\Big]
    + P_x\Big[n^{-1} \sum_{v \in N^+_n}\chi(\varphi_v) <
      \bar{\delta}_n/\delta
      \,\Big|\, N^+_n\neq\emptyset\Big].
  \end{equation}
  We now show that there is a sequence $\bar{\delta}_n \to 0$ such that
  both summands converge to zero. By Lemma~\ref{lem:conv-to-exp_prep},
  the first summand is bounded by
  $\bar{\delta}_n^{-2}\bar{\varepsilon}_n^f(x)$, where the
  $\bar{\varepsilon}_n^f$-term is independent of $\bar\delta_n$ and
  converges to zero. Therefore, if $\bar{\delta}_n \to 0$ sufficiently
  slowly, then also $\bar{\delta}_n^{-2}\bar{\varepsilon}_n^f(x) \to 0$.
  For the
  second summand, fix $\varepsilon >0$. Then for $n$ large enough so that
  $\bar \delta_n/\delta \le \varepsilon$,
  \begin{equation}
    P_x\Big[n^{-1} \sum_{v \in N^+_n}\chi(\varphi_v) <
      \bar{\delta}_n/\delta
      \,\Big|\, N^+_n\neq\emptyset\Big]
    \le
    P_x\Big[n^{-1} \sum_{v \in N^+_n}\chi(\varphi_v) <
      \varepsilon
      \,\Big|\, N^+_n\neq\emptyset\Big].
  \end{equation}
  By Theorem~\ref{thm:conv-to-exp}, the right-hand side converges to
  $P(Z < C_1 \varepsilon)$ which can be made arbitrarily small by
  choosing $\varepsilon$ small. This implies that the second summand in
  \eqref{eq:sum_bound} converges to zero and completes the proof.
\end{proof}

\begin{corollary}
  \label{prop:f-average}
  For every $x\ge h^*$, $\delta >0$ and $f,g\in L^2(\nu)$ with $g>0$,
  \begin{equation}
    \label{eq:foverg}
    \lim_{n\to\infty}P_x\bigg[\abs[\bigg]{
        \frac{\sum_{v \in N^+_n}f(\varphi_v)}
        {\sum_{v \in N^+_n} g(\varphi_v)}
        - \frac{\inp{\chi}{f}}{\inp{\chi}{g}}
    } > \delta \,\bigg|\, N^+_n\neq\emptyset \bigg] = 0.
  \end{equation}
  In particular, setting $g \equiv 1$,
  \begin{equation}
    \label{eq:f-average_implication}
    \lim_{n\to\infty}P_x\bigg[\abs[\bigg]{
        \frac{1}{\abs{N_n^+}}\sum_{v \in N_n^+}f(\varphi_v)
        - \frac{\inp{\chi}{f}}{\inp{\chi}{1}}
    } > \delta \,\bigg|\, N^+_n\neq\emptyset \bigg] = 0.
  \end{equation}
\end{corollary}

\begin{proof}
  It is easy to see that there is $\delta' = \delta'(\delta ,f,g)>0$ such
  that the event on the left-hand side of \eqref{eq:foverg} is contained
  in the union of
  $\{\abs{\sum f(\varphi_v)/\sum \chi(\varphi_v)-\inp{\chi}{f}}\ge \delta'\}$
  and
  $\{\abs{\sum g(\varphi_v)/\sum \chi(\varphi_v)-\inp{\chi}{g}}\ge \delta'\}$.
  The statement then follows by Lemma~\ref{lem:sum_f_average}.
\end{proof}

The final result of this section is the following technical lemma which
can be seen as a somewhat stronger version of
\eqref{eq:f-average_implication} in Corollary~\ref{prop:f-average}. It is
tailored to be used in the proof of Proposition~\ref{prop:Q-limit}.

\begin{lemma}
  \label{lem:f-average-tech}
  Write $N^+_n = \set{v_1, v_2, \dots, v_{\abs{N_n^+}}}$ and set
  $N^+_{n,M} = \set{v_1, v_2, \dots, v_M}$. For $\delta, \rho >0$ and
  $f:\mathbb R\to\mathbb R$ bounded, define the event
  \begin{equation}
    B_n(f,\delta,\rho) \defeq
    \set{\rho n \le \abs{N^+_n}}
    \cap \bigg(\bigcup_{\rho n \le M \le \abs{N^+_n}}
      \bigg\{\abs[\bigg]{
        \frac{\sum_{v \in N^+_{n,M}}f(\varphi_v)}{M}
        - \frac{\inp{\chi}{f}}{\inp{\chi}{1}}}
      >\delta\bigg\}\bigg).
  \end{equation}
  Then for every $x \ge h^*$,
  \begin{equation}
    \label{eq:B_n-toshow}
    \lim_{n\to\infty}P_{x}\big[B_n(f,\delta,\rho)
      \big| N_n^+ \neq\emptyset \big] =0.
  \end{equation}
\end{lemma}

\begin{proof}
  We start by outlining the strategy of the proof. We will define events
  $A_n = A_n(n_0)$ (see \eqref{eq:A_n-def}) and $A'_n = A'_n(n_0, q)$
  (see \eqref{eq:A'_n-def}), so that when $n_0 = n_0(n)$ and $q=q(n)$ are
  chosen correctly,
  \begin{equation}
    \begin{gathered}
      \label{eq:B_n-conditions}
      \lim_{n\to\infty}
      P_x\big[A_n(n_0(n)) \big| N_n^+ \neq\emptyset\big] = 0,
      \\
      \lim_{n\to\infty}
      P_x\big[A'_n(n_0(n), q(n)) \big| N_n^+ \neq\emptyset\big] =0,
      \\
      B_n(f,\delta, \rho) \subseteq A_n(n_0(n)) \cup A'_n(n_0(n), q(n))
      \quad \text{for $n$ large enough}.
    \end{gathered}
  \end{equation}
  The statement of the lemma follows directly from these claims by a
  union bound.

  To define $A_n = A_n(n_0)$, we fix $0 \leq n_0 \leq n$, write
  $N_{n_0}^+ = \set{w_1, \dots, w_{\abs{N_{n_0}^+}}}$, and set
  $N_n^i = \set{v \in N_n^+ : \text{$w_i$ is ancestor of $v$}}$. For
  $1\le i\le \abs{N_{n_0}^+}$, we set
  \begin{equation}
    m_i = \begin{cases}
      \abs{N_n^i}^{-1} \sum_{v \in N_n^i}f(\varphi_v),\qquad
      &\text{if }\abs{N_n^i}>0,\\
      \inp{\chi}{f}/\inp{\chi}{1}, &\text{otherwise},
    \end{cases}
  \end{equation}
  and define the events
  \begin{equation}
    \label{eq:A_n-def}
    A_n^i = \Big\{
      \abs[\Big]{m_i - \frac{\inp{\chi}{f}}{\inp{\chi}{1}}} >
    \frac{\delta}{2} \Big\}
    \quad\text{and}\quad
    A_n = A_n(n_0) = \bigcup_{i=1}^{\abs{N_{n_0}^+}} A_n^i.
  \end{equation}

  We now prove an upper bound on $P[A_n | N_n^+ \neq\emptyset]$. By a
  union bound, using that
  $A_{n}^i \subset\{N_n^i \neq \emptyset\}\subset\{N_n^+ \neq\emptyset\}$,
  \begin{equation}
    \label{eq:techaa}
    \begin{split}
      P_x[A_n | N_n^+ \neq\emptyset ]
      &= P_x[ N_n^+ \neq\emptyset]^{-1}
      P_x[A_n \cap \{N_n^+ \neq\emptyset\}]
      \\&\le P_x[ N_n^+ \neq\emptyset]^{-1}
      E_x\bigg[\sum_{i=1}^{\abs{N_{n_0}^+}}
        P_x[A_n^i \cap \{N_n^i \neq \emptyset\} |
          \mathcal F_{n_0}]\bigg].
    \end{split}
  \end{equation}
  Using the branching process properties of the GFF,
  \begin{equation}
    \label{eq:techbb}
    \begin{split}
      P_x[&A_n^i \cap \set{N_n^i \neq\emptyset} | \F_{n_0}]\\
      &= P_{\varphi_{w_i}}\bigg[\abs[\bigg]{
          \frac 1{\abs{N_{n-n_0}^+}} \sum_{v \in N_{n-n_0}^+} f(\varphi_v)
          - \frac{\inp{\chi}{f}}{\inp{\chi}{1}}}
        > \frac{\delta}{2}\, \bigg|\, N_{n-n_0}^+ \neq\emptyset\bigg]
      P_{\varphi_{w_i}}[N_{n-n_0}^+ \neq\emptyset].
    \end{split}
  \end{equation}
  By Corollary~\ref{prop:f-average}(b), the first probability on the
  right-hand side is bounded by some  $\varepsilon_{n-n_0}(\varphi_{w_i})$
  satisfying $1\ge \varepsilon_n \to 0$ pointwise, and thus in any
  $L^q(\nu)$ by the bounded convergence theorem. By
  Proposition~\ref{prop:u_n^fk-L2-asymtote} and
  Lemmas~\ref{lem:u_n^k-L2-asymptote},~\ref{lem:a^f-asymtote}, the second
  probability equals
  \begin{equation}
    P_{\varphi_{w_i}}[N_{n-n_0}^+ \neq\emptyset]
    = C_1 (n-n_0)^{-1}\big(\chi(\varphi_{w_i})
      + \bar \varepsilon_{n-n_0}(\varphi_{w_i})\big),
  \end{equation}
  where $\bar \varepsilon \to 0$ pointwise and in $L^{5/2}(\nu)$.
  Combining these statements with \eqref{eq:techaa} and
  \eqref{eq:techbb}, and using then Proposition~\ref{pro:Harris_L2} in
  the numerator, we obtain
  \begin{equation}
    \label{eq:f-average-tech-step1}
    \begin{split}
      P_x[A_n | \set{N_n^+ \neq\emptyset}]
      &\le \frac
      {E_x \Big[\sum_{w \in N_{n_0}^+} \varepsilon_{n-n_0}(\varphi_{w})C_1
          (n-n_0)^{-1}(\chi(\varphi_{w})
            + \bar\varepsilon_{n-n_0}(\varphi_{w}))
          \Big]}
      {C_1 n^{-1} (\chi (x) + \bar \varepsilon_n(x))}
      \\ &=
      \frac{c(x) n}{n-n_0}
      \inp{\chi}{\varepsilon_{n-n_0}(\chi + \bar \varepsilon_{n-n_0})}\chi(x)
      + \tilde \varepsilon_{n_0}(x),
    \end{split}
  \end{equation}
  where $\tilde \varepsilon_n\to 0$ pointwise. Finally, using Hölder's
  inequality on the inner product, since $\chi, \chi^2 \in L^2(\nu)$ (by
    \eqref{eq:crit-chi_bounds}), $\varepsilon_n \to 0$ in $L^5(\nu)$ and
  $\bar \varepsilon \to 0$ in $L^{5/2}(\nu)$, it follows that for every
  $x\ge h^*$,
  \begin{equation}
    \label{eq:An}
    \text{if $n_0\to\infty$ and $n-n_0\to\infty$, then }
    P_x[A_n(n_0) | \set{N_n^+ \neq\emptyset}] \to 0.
  \end{equation}

  We now turn to the events $A'_n$. For given $n_0 \le n$ and
  $q$, let
  \begin{equation}
    \label{eq:A'_n-def}
    A'_n = A'_n(n_0, q) = \bigcup_{i=1}^{\abs{N_{n_0}^+}}
    \set{\abs{N_n^i} > q}.
  \end{equation}
  To bound the probability $P[A'_n | N_n^+\neq\emptyset]$, we write
  \begin{equation}
    \label{eq:A'_n-1}
    P_x[A'_n \cap \set{N_n^+\neq\emptyset}]
    \le P_x[A'_n]
    \le E_x\Big[\sum_{i=1}^{\abs{N_{n_0}^+}}
      P_{\varphi_{w_i}}[\abs{N_{n-n_0}^+} > q] \Big].
  \end{equation}
  By the Markov inequality and Proposition~\ref{pro:Harris_L2}(b) (with
    $f = g = 1$),
  \begin{equation}
    P_{\varphi_{w_i}}[\abs{N_{n-n_0}^+} > q]
    \le q^{-2}E_{\varphi_{w_i}}[\abs{N_{n-n_0}^+}^2]
    = Cq^{-2} (\chi(\varphi_{w_i})(n-n_0)
      + \varepsilon_{n-n_0}(\varphi_{w_i})),
  \end{equation}
  where $\norm{\varepsilon_{n-n_0}} \le c$. Thus, by
  Proposition~\ref{pro:Harris_L2}(a),
  \begin{equation}
    \begin{split}
      \label{eq:Aprimen}
      P_x[A'_n \cap \set{N_n^+\neq\emptyset}]
      &\leq q^{-2}E_x\Big[\sum_{v \in N_{n_0}^+}
        (c \chi(\varphi_{v})(n-n_0)
          + \varepsilon_{n-n_0}(\varphi_v))\Big]
      \\&\le cq^{-2}\big((n-n_0+1)\chi(x) +
        \varepsilon'_{n_0}(x)\big).
    \end{split}
  \end{equation}
  with $\varepsilon'_{n}\to 0$ pointwise. Together with
  $P_x[N_n^+\neq \emptyset ] \ge c n^{-1}$, by
  Theorem~\ref{thm:diameter}, this implies
  \begin{equation}
    \label{eq:A'_n-final}
    P_x[A'_n | N_n^+\neq\emptyset]
    \le \frac{n}{q^2}\Big((n-n_0+C)\chi(x) + \varepsilon '_{n_0}(x)\Big).
  \end{equation}

  We will now give sufficient conditions on the functions $n_0(n)$ and
  $q(n)$ so that for large enough $n$,
  \begin{equation}
    \label{eq:B_n-inclusion}
    B_n(f,\delta,\rho) \subset A_n \cup A'_n, \text{ or equivalently }
    B_n(f,\delta,\rho)^c \supset (A_n)^c \cap (A'_n)^c.
  \end{equation}
  By definition,
  $\set{\rho n \le \abs{N_n^+}}^c \subseteq B_n(f,\delta,\rho)^c$.
  Therefore, \eqref{eq:B_n-inclusion} is implied by
  \begin{equation}
    \label{eq:B_n^c-inclusion}
    (A_n)^c \cap (A'_n)^c \cap \set{\rho n \le \abs{N_n^+}}
    \subseteq B_n(f,\delta,\rho)^c.
  \end{equation}
  To prove this, we assume that
  $(A_n)^c \cap (A'_n)^c \cap \set{\rho n \le \abs{N_n^+}}$ holds. For
  $M \le \abs{N_n^+}$, let
  \begin{equation}
    k(M) \defeq \inf\Big\{\sum_{i=1}^\ell  \abs{N_n^i} :
      \ell\in \{0,\dots, \abs{N_{n_0}^+}\}
      \text{ such that }\sum_{i=1}^\ell  \abs{N_n^i} \ge M\Big\}
    \ge M.
  \end{equation}
  On $(A_n)^c$,
  \(
    \abs[\big]{\abs{N_n^i}^{-1} \sum_{v \in N_n^i} f(\varphi_v)
      -\inp{\chi}{f}/\inp{\chi}{1}} \leq \delta/2
  \)
  for all $i = 1, \dots, \abs{N_{n_0}^+}$. Therefore,
  for all $M$ with $\rho n \le M \le \abs{N_n^+}$
  \begin{equation}
    \label{eq:A_n-bound}
    \abs[\bigg]{\frac{\sum_{v \in N_{n,k(M)}^+} f(\varphi_v)}{k(M)}
      -\frac{\inp{\chi}{f}}{\inp{\chi}{1}}}
    \leq \frac{\delta}{2}.
  \end{equation}
  Writing
  \(
    \sum_{v \in N_{n,M}^+}f(\varphi_v)
    = \sum_{v \in N_{n,k(M)}^+} f(\varphi_v)
    - \sum_{v \in N_{n,k(M)}^+\setminus N_{n,M}^+} f(\varphi_v)
  \),
  and using that the absolute value of the second sum is bounded by
  $(k(M)-M)\sup\abs{f}$, we obtain that for every $M\ge \rho n$
  \begin{equation}
    \begin{split}
      \label{eq:pre-B}
      \abs[\bigg]{&\frac{\sum_{v \in N_{n,M}^+} f(\varphi_v)}{M}
        - \frac{\sum_{v \in N_{n,k(M)}^+}f(\varphi_v)}{k(M)}
      }
      \\&= \abs[\bigg]{
        \frac{\sum_{v \in N_{n,k(M)}^+} f(\varphi_v)}{M}
        - \frac{\sum_{v \in N_{n,k(M)}^+\setminus N_{n,M}^+} f(\varphi_v)}{M}
        - \frac{\sum_{v \in N_{n,k(M)}^+}f(\varphi_v)}{k(M)}
      }
      \\&\le \abs[\bigg]{
        \sum_{v \in N_{n,k(M)}^+} f(\varphi_v)\Big(\frac{1}{M}
          - \frac{1}{k(M)}\Big)}
      + \frac{k(M)-M}{M} \sup\abs{f}\\
      &= \frac{k(M) - M}{M}
      \bigg(\abs[\bigg]{
          \frac{\sum_{v \in N_{n,k(M)}^+} f(\varphi_v)}{k(M)}}
        + \sup\abs{f}\bigg)
      \le c(f,\delta)  \frac q{\rho n},
    \end{split}
  \end{equation}
  where in the last inequality we used \eqref{eq:A_n-bound} and the fact
  that $k(M)-M \leq q$ on $(A'_n)^c$. If the right-hand side of
  \eqref{eq:pre-B} is smaller than $\delta /2$, then together with
  \eqref{eq:A_n-bound}, this implies that $B_n(f,\delta,\rho)^c$
  holds, implying \eqref{eq:B_n^c-inclusion} and thus
  \eqref{eq:B_n-inclusion}.

  To finish the proof of \eqref{eq:B_n-conditions}, we must choose
  $n_0=n_0(n)$ and $q =q(n)$ so that \eqref{eq:An} applies and the
  left-hand sides of \eqref{eq:Aprimen}, \eqref{eq:pre-B} tend to zero.
  This is easily done by setting, e.g., $q (n) = n^{3/4}$ and
  $n_0 = n - n^{1/4}$.
\end{proof}

\section{\texorpdfstring{The $S_n$ martingale}{The Sn martingale}}
\label{sec:S_n-process}

The remaining three sections of this paper are dedicated to proving
Theorem~\ref{thm:inv-principle}. In this section, we will introduce a
martingale based on the depth-first traversal of a sequence of copies of
$\mathcal C_o \cap \mathbb T^+$, and prove that it satisfies an
invariance principle, see Proposition~\ref{prop:S_n-scaling-limit}. This
martingale can be seen as an analogue to the Lukasiewicz path used to
study critical Galton-Watson trees. The second part of the section then
demonstrates another two scaling limit results,
Proposition~\ref{prop:joint-scaling-limit} and
Proposition~\ref{prop:conditioned-scaling-limit} which are both
consequences of Proposition~\ref{prop:S_n-scaling-limit}.

To define the martingale, we consider an i.i.d.~sequence
$\mathbf{T} = ((T^1, \varphi^1), (T^2, \varphi^2), \dots)$ where every
$(T^i, \varphi^i)$ is distributed as
$(\mathcal C_o \cap \mathbb T^+, \varphi |_{\mathcal C_o \cap \mathbb T^+})$
under $P_x$. To keep the notation simple, we keep using $P_x$ for the
probability measure associated with the whole sequence $\mathbf{T}$, and
for $v \in T^i$ we write $\varphi_v$ instead of $\varphi^i_v$. We further
use $o^i$ to denote the root of $T^i$, and set
$N_n^{+,i} = \{v\in T^i: d(o^i,v) = n\}$. Note that $\abs{N_n^{+,i}}$ has
the same distribution as $\abs{N_n^+}$ which was studied in detail in the
previous sections. In particular, we know that all $T^i$ are a.s.~finite.

Throughout this section, we use the notation for the parent $p(v)$,
direct descendants $\desc(v)$, and siblings $\sib(v)$ of $v\in T^i$
relative to the rooted tree $T^i$ (not $\mathbb T$),  $w \preceq v$
means that $w$ is an ancestor of $v$, cf.~below~\eqref{eq:spheres}.

We now describe the depth-first traversal $\mathbf{v} = (v_1, v_2, \dots)$
of $\mathbf{T}$. It starts at the root of $T^1$, that is $v_1 = o^1$, and
then explores the tree $T^1$ in a depth-first manner. After
visiting all vertices of $T^1$, it proceeds to $o^2$, explores $T^2$ in a
depth-first manner, and so forth. $v_i$ denotes the $i$-th vertex visited
during this traversal. The notation $v <_{\mathbf{v}} w$ indicates that
$v$ precedes $w$ in $\mathbf{v}$ (with $\le_{\mathbf{v}}$ representing
  the reflexive version). For any $v \in \cup_i T^i$ we write $\Lambda(v)$
to denote the index of the tree which $v$ belongs to (that is
  $v \in T^{\Lambda(v)}$). We define
\begin{equation}
  \label{eq:def-Lambda,H}
  \Lambda_n = \Lambda(v_n)
  \quad \text{and} \quad
  H_n = \abs{v_n} =  d(v_n, o^{\Lambda_n}),
\end{equation}
that is $\Lambda_n$ is the index of the tree which is explored at step
$n$, and $H_n$ is the ``height'' of the $n$-th explored vertex. For every
$v \in \cup_i T^i$, we define the set
$Y(v)$ as
the union of all siblings of some ancestor of $v$ that appear later in
$\mathbf{v}$ than this ancestor itself, that is,
\begin{equation}
  \label{eq:Y-def}
  Y(v) = \bigcup_{w \preceq v} \set{u \in \sib(w) : w <_{\mathbf{v}} u}.
\end{equation}

With this notation we can introduce the key object of this section, the
process
\begin{equation}
  \label{eq:S_n-def}
  S_n = \chi(\varphi_{v_n}) - \sum_{i \le \Lambda_n} \chi(\varphi_{o^i})
  + \sum_{w \in Y(v_n)} \chi(\varphi_w),
  \qquad n\in \mathbb N,
\end{equation}
which is adapted to the filtration
\begin{equation}
  \mathcal{H}_n = \sigma\Big((v, \varphi_v): v\in \bigcup_{i=1}^n \sib
      (v_i)\Big), \qquad n\in \mathbb N.
\end{equation}

The next proposition shows that $S_n$ converges to a Brownian motion with
variance
\begin{equation}
    \label{eq:sigma-def}
  \sigma^2 = \inp{\chi}{\mathcal{V}}/\inp{\chi}{1},
\end{equation}
where
\begin{equation}
  \label{eq:V_chi-def}
  \mathcal{V}(x)
  := P_{x}\Big[\Big(\sum_{w \in N_1^+} \chi(\varphi_w)\Big)^2\Big]
  - P_{x}\Big[\sum_{w \in N_1^+} \chi(\varphi_w)\Big]^2
  = \Var_{x}\Big(\sum_{w \in N_1^+} \chi(\varphi_w)\Big).
\end{equation}
Here and below $(B_t)_{t\ge 0}$ denotes the standard Brownian motion.

\begin{proposition}
  \label{prop:S_n-scaling-limit}
  For every $x\ge h^*$, in $P_x$-distribution w.r.t.~the Skorokhod topology,
  \begin{equation}
    \lim_{n\to\infty}\Big(\frac{1}{\sqrt{n}}S_{\floor{nt}}\Big)_{t \ge 0}
    = (\sigma B_t)_{t \ge 0}.
  \end{equation}
\end{proposition}

We need a few preparatory steps to show this proposition.
We start by proving that $S$ is a martingale.

\begin{lemma}
  \label{lem:S_n-is-mtg}
  The process $S$ is a $\mathcal{H}$-martingale under $P_x$.
\end{lemma}

\begin{proof}
  We first show that
  \begin{equation}
    \label{eq:S-diff}
    S_{k+1} - S_k
    = -\chi(\varphi_{v_k}) + \sum_{w \in \desc(v_k)} \chi(\varphi_w).
  \end{equation}
  To this end we need to distinguish between several possible scenarios:

  (1) If $v_k$ is not a leaf, that is $\desc(v_k) \neq \emptyset $, then
  $v_{k+1}$ is the first child of $v_k$, and thus
  $Y(v_{k+1}) = (\desc(v_k)\setminus \set{v_{k+1}}) \cup Y(v_k)$ where
  the union is disjoint. From this \eqref{eq:S-diff} follows.

  (2) If $v_k$ is a leaf, that is $\desc(v_k) = \emptyset$, then there
  are three possible cases for $v_{k+1}$: Either it is the next sibling
  (with respect to $<_{\mathbf{v}}$) of $v_k$, or the next sibling of
  some ancestor of $v_k$, or it is the root of the next tree in the tree
  sequence. In the first two cases,
  $Y(v_k) = Y(v_{k+1}) \cup \set{v_{k+1}}$ where the union is disjoint,
  and thus $S_{k+1} - S_k = -\chi(\varphi_{v_k})$ and thus
  \eqref{eq:S-diff} holds. In the last case, when $v_{k+1}$ is the root
  of the next tree, both sets $Y(v_k)$ and $Y(v_{k+1})$ must be empty and
  $\Lambda_{k+1} = \Lambda_k +1$, which implies \eqref{eq:S-diff} also in
  this case.

  From \eqref{eq:S-diff} and the branching process properties of
  $\varphi $ it follows that
  \begin{equation}
    \begin{split}
      E_x&[S_{k+1}-S_k | \mathcal{H}_k]
      = E_x\Big[-\chi(\varphi_{v_k})
        + \sum_{w \in \desc(v_k)} \chi(\varphi_w)
        \Big| \mathcal{H}_k\Big] \\
      &= -\chi(\varphi_{v_k}) + E_{\varphi_{v_k}}
      \Big[\sum_{w \in N_1^+} \chi(\varphi_w)\Big]
      = -\chi(\varphi_{v_k}) + L\chi (\varphi_{v_k}) = 0,
    \end{split}
  \end{equation}
  since $\chi $ is the eigenfunction of $L$ (see
    \eqref{eq:chi-is-eigenfunc}). This finishes the proof.
\end{proof}

The next four simple lemmas will be used to control the quadratic
variation of $S$.

\begin{lemma}
  \label{lem:S_k-square-diff}
  For every $k\ge 1$, with~$\mathcal V$ as in \eqref{eq:V_chi-def},
  \begin{equation}
    E_x[(S_{k+1}- S_k)^2 | \mathcal{H}_k] = \mathcal{V}(\varphi_{v_k}).
  \end{equation}
\end{lemma}

\begin{proof}
  By \eqref{eq:S-diff}
  the conditional expectation $E_x[(S_{k+1}-S_k)^2 | \mathcal{H}_k]$
  equals
  \begin{equation}
    \begin{split}
      &\chi(\varphi_{v_k})^2
      - 2\chi(\varphi_{v_k})
      E_x\Big[
        \sum_{w \in \desc(v_k)}\chi(\varphi_w) \Big| \mathcal{H}_k \Big]
      + E_x\Big[\Big(\sum_{w \in \desc(v_k)} \chi(\varphi_w)\Big)^2
        \Big| \mathcal{H}_k\Big] \\
      &= \chi(\varphi_{v_k})^2
      - 2\chi(\varphi_{v_k})E_{\varphi_{v_k}}
      \Big[\sum_{w \in N_1^+}\chi(\varphi_w)\Big]
      + E_{\varphi_{v_k}}\Big[
        \Big(\sum_{w \in N_1^+} \chi(\varphi_w)\Big)^2\Big].
    \end{split}
  \end{equation}
  From this the statement follows using \eqref{eq:chi-is-eigenfunc} again.
\end{proof}

\begin{lemma}
  \label{lem:var-bound}
  There is $c<\infty$ such that $\mathcal V(x)<c$ for all $x\ge h^*$.
\end{lemma}

\begin{proof}
  Let $\set{w_1, \dots, w_d}$ be the children of the root in
  $\mathbb T^+$. Since $(\varphi_{w_i})_{i=1}^d$ are independent
  under $P_x$ and $\chi(x) = 0$ for $x < h^*$,
  \(
    \Var_x(\sum_{v \in N_1^+} \chi(\varphi_v))
    = \Var_x(\sum_{i=1}^d \chi(\varphi_{w_i}))
    = \sum_{i=1}^d \Var_x(\chi(\varphi_{w_i}))
    = d \Var_Y(\chi(Y + x/d))
  \),
  where $Y \sim \mathcal N(0, \sigma_Y^2)$, see
  \eqref{eq:BP-representation}. Since $\chi$ is Lipschitz
  on $[h^*, \infty)$ by~\eqref{eq:chi-Lipschitz}, the statement follows.
\end{proof}

To state the next result, let $U_n$ be a random variable distributed
uniformly on $\{1,\dots,n\}$, defined on the same probability space as the
sequence $\mathbf T$, independent of $\mathbf T$. We write
\begin{equation}
  \label{eq:def_Lambda,H}
  H_n^* = H_{U_n}
  \quad \text{and} \quad \Lambda_n^* = \Lambda_{U_n},
\end{equation}
for the height and the tree index of a vertex
chosen uniformly amongst the first $n$ explored vertices.

\begin{lemma}
  \label{lem:H_n>Csqrt(n)-tozero}
  For all $x \ge h^*$,
  \begin{equation}
    \lim_{u\to\infty}\sup_{n\ge 1} P_x [H_n^* \ge u \sqrt{n}] = 0.
  \end{equation}
\end{lemma}

\begin{proof}
  By Theorem~\ref{thm:diameter} and Proposition~\ref{pro:cl23},
  there is $c(x)<\infty$ such that for all $n\ge 1$,
  \begin{equation}
      \label{eq:K(x)-cond}
      P_x[N_n^+ \neq \emptyset ] \le \frac{c(x)}{n}
      \qquad\text{and}\qquad
      E_x[\Lambda_n] \le c(x) \sqrt{n}.
  \end{equation}
  To see the second inequality in \eqref{eq:K(x)-cond}, we note that by
  the independence of the trees $T^i$ and Proposition~\ref{pro:cl23},
  \(
    P_x[\Lambda_n \ge k]
    \le P_x[|T^i| < n, i=1,\dots, k-1]
    = P_x[{\abs{T^1} < n}]^{k-1}
    \le (1-c(x) n^{-1/2})^{k-1}
  \)
  and thus
  \(
    E_x[\Lambda_n] = \sum_{k=1}^\infty P_x[\Lambda_n \ge k]
    \le c(x)^{-1}n^{1/2}
  \)
  as claimed.

  Since
  \(\set{H_n^* \ge u \sqrt{n}}
    \subseteq \set{\exists i \le \Lambda_n : N_{\ceil{u\sqrt{n}}}^{+,i}
      \neq \emptyset  }
  \),
  by a union bound,
  \begin{equation}
    P_x [H_n^* \ge u \sqrt{n}]
    \le \sum_{i\ge 1}
    P_x\big[i\le \Lambda_n, N_{\ceil{u \sqrt{n}}}^{+,i}\neq \emptyset \big].
  \end{equation}
  The events $\set{i \leq \Lambda_n}$ and
  $\set{\abs{N_{\ceil{u\sqrt{n}}}^{+,i}}>0}$ are independent. Hence, using
  the first half of \eqref{eq:K(x)-cond}, this is bounded by
  $c(x) u^{-1} n^{-1/2}\sum_{i \ge 1} P_x[i \le \Lambda_n]$. Since
  $\sum_{i \ge 1} P_x[{i \le \Lambda_n }] =   E_x[\Lambda_n] \le
  c(x)\sqrt{n}$, by the second half of \eqref{eq:K(x)-cond}, the claim
  follows.
\end{proof}

\begin{lemma}
  \label{lem:P[H*_t<R(x,n)]>delta/3}
  For every $\delta>0$ and $x \ge h^*$ there exists a sequence
  $R(n)$ such that $\lim_{n\to\infty}R(n)=\infty $  and
  \begin{equation}
    \sup_{n\ge 1} P_x[H^*_n \le R(n)] \le \delta.
  \end{equation}
\end{lemma}

\begin{proof}
  We will show that for every $K<\infty$
  \begin{equation}
    \lim_{n\to\infty }P_x [H_n^* \le K] = 0,
  \end{equation}
  which implies the statement of the lemma.

  For fixed $K>0$, by definition of  $H^*_n$,
  \begin{equation}
    P_x[H_n^* \le K] = n^{-1}\sum_{l = 1}^n P_x [H_l \le K]
    = n^{-1} E_x\Big[\sum_{h=0}^K \sum_{l=1}^n 1_{H_l = h}\Big].
  \end{equation}
  Since
  $\sum_{l=1}^n 1_{H_l = h} \le \sum_{i=1}^{\Lambda_n} \abs{N_{h}^{+,i}}$
  and thus
  \(
    \sum_{h=0}^K \sum_{l=1}^n 1_{H_l = h}
    \le \sum_{i=1}^{\Lambda_n} \sum_{h = 0}^K \abs{N_{h}^{+,i}}
  \).
  Moreover, $(\sum_{h=0}^{K} \abs{N_{h}^{+,i}})_{i=1}^{\Lambda_n}$ are i.i.d.
  and with finite mean. Noting that $\Lambda_n$ is a stopping time with
  respect to $\mathcal{G}_i = \sigma(T^j: j\le i)$ and using Wald's
  equation (see e.g. \cite[Theorem 2.6.2]{Durrett19}),
  \begin{equation}
    \label{eq:H_n^*<K-bound}
    P_x[H_n^* \le K]
    \le \frac{1}{n} E_x\Big[\sum_{i=1}^{\Lambda_n}
      \sum_{h=0}^K |N_{h}^{+,i}|\Big]
    = \frac{1}{n} E_x[\Lambda_n] E_x\Big[\sum_{h=0}^K |N_{h}^{+,1}|\Big],
  \end{equation}
  which by \eqref{eq:K(x)-cond} converges to $0$ as $n \to \infty$.
\end{proof}

The following law of large numbers will be important for the proof of
Proposition~\ref{prop:S_n-scaling-limit}.

\begin{proposition}
  \label{prop:Q-limit}
  Let $f$ be a bounded function and
  $m^f_n \defeq n^{-1}\sum_{i=1}^n f(\varphi_{v_i})$. Then
  \begin{equation}
    \lim_{n\to\infty }m^f_n
    = m_\infty^f := \frac{\inp{\chi}{f}}{\inp{\chi}{1}},
    \quad \text{in $P_x$-probability.}
  \end{equation}
\end{proposition}

\begin{proof}
  Let
  \begin{equation}
    N_{k,n}^{+,i} = N_k^{+,i} \cap \{v_1,\dots, v_n\}
  \end{equation}
  be the part of $N_k^{+,i}$ traversed in the first $n$ steps,
  and let $N^*_{n} = N_{H_n^*,n}^{+,\Lambda_n^*}$. We set
  \begin{equation}
    m^*_n = \abs{N_{n}^*}^{-1} \sum_{v \in N_{n}^*} f(\varphi_v).
  \end{equation}
  Denoting by $\mathcal F_{\mathbf T}$ the $\sigma $-algebra generated by
  the sequence $\mathbf T = (T_i,\varphi^i)_{i\ge 1}$, it holds that
  \begin{equation}
    \label{eq:Q_n-observation}
    m_n^f = E_x[m_n^*| \mathcal F_{\mathbf T}]
  \end{equation}
  (effectively, the expectation here is over $U_n$ only). To show the
  proposition it will thus be sufficient to show that
  \begin{equation}
    \label{eq:Q_n-toshow}
    \lim_{n\to\infty } m^*_n = m_\infty^f
    \quad \text{in $P_x$-probability. }
  \end{equation}
  To see that this is indeed enough, note that since $f$ is bounded
  $m^*_n$ is dominated by a constant. Hence, by the dominated
  convergence theorem, $m_n^*\to m_\infty^f$ also in $L^2(P_x)$. By
  \eqref{eq:Q_n-observation}, since the conditional expectation is a
  contraction on $L^2(P_x)$,  $m_n^f \to m_\infty^f$ in $L^2(P_x)$,
  and thus in $P_x$-probability as claimed.

  To prove \eqref{eq:Q_n-toshow}, we fix $\varepsilon, \delta  >0$ and set
  \begin{equation}
    A_n = \{\abs{m_n^* - m_\infty^f} > \varepsilon \}.
  \end{equation}
  We then fix $C<\infty$ such that $P_x[H_n^* \ge C n^{1/2}] \le \delta/3$
  for all $n\ge 1$, which is possible by
  Lemma~\ref{lem:H_n>Csqrt(n)-tozero}, and take $R(n)$ as in
  Lemma~\ref{lem:P[H*_t<R(x,n)]>delta/3}, so that
  $P_x[H_n^* \le R(n)] \le \delta/3$. Setting
  $B_n = \{R(n) \le H_n^* \le C n^{1/2}\}$, it follows that for all
  $n\ge 1$,
  \begin{equation}
    \label{eq:pxan}
    P_x[A_n]\le 2\delta /3 + P_x[A_n  \cap B_n ]
  \end{equation}
  We further set
  $m_{h,n}^i = |N_{h,n}^{+,i}|^{-1} \sum_{v \in N_{h,n}^{+,i}} f(\varphi_v)$
  and $A_{h, n}^i = \set{\abs{m_{h, n}^i - m_\infty^f}> \varepsilon}$.
  Using the definitions of $U_n$ and $m^*_n$, and then
  decomposing by possible values of $H_k$, $\Lambda_k$, we obtain
  \begin{equation}
    \label{eq:P_x-times-mu_n[A_n]-calc}
    \begin{split}
      P_x [A_n \cap B_n]
      &
      = \sum_{k = 1}^{n} P_x[A_n \cap B_n | U_n = k]P_x[U_n = k]
      \\&= \sum_{k = 1}^{n} \frac 1n P_x[A_{H_k, n}^{\Lambda_k}, R(n)\le H_k
        \le C n^{1/2}]
      \\&= \frac 1n
      \sum_{l=R(n)}^{Cn^{1/2}} \sum_{i=1}^n
      E_x\Big[ 1_{A_{l, n}^{i}} \sum_{k = 1}^{n}
        1_{\set{H_k = l, \Lambda_k = i}} \Big].
      \\& =
      \frac 1n \sum_{l=R(n)}^{Cn^{1/2}} \sum_{i=1}^n
      E_x\big[1_{A_{l, n}^{i}}
        \abs[\big]{N_{l,n}^{+,i}}\big],
    \end{split}
  \end{equation}
  where in the last step we used
  $\sum_{k=1}^{n}  1_{\set{H_k = l, \Lambda_k = i}}  = \abs{N_{l,n}^{+,i}}$.
  To analyse the expectation on the right-hand side,  we define
  $\sigma $-algebras $\mathcal{G}_{i} = \sigma(T^1, T^2, \dots, T^i)$ and set
  $\tau_i = \abs{T^1}+ \dots + \abs{T^{i-1}}$ (which is
    $\mathcal{G}_{i-1}$-measurable). Then, using that
  $(T^i,\varphi^i)_{i\ge 1}$ is an i.i.d.~sequence,
  \begin{equation}
    \begin{split}
      \label{eq:Q_n-calc1}
      E_x\big[1_{A_{l, n}^{i}}\abs{N_{l,n}^{+,i}}\big| \mathcal{G}_{i-1}\big]
      &= 1_{\set{\tau_i \le n}} E_x\big[1_{A^1_{l, n - \tau_i}}
        \abs[\big]{N_{l,n-\tau_i}^{+,1}}\big] \\
      &= 1_{\set{\tau_i\le n}}
      E_x\big[1_{A_{l, n - \tau_i}}
        |N_{l,n-\tau_i}^{+}| \big| N_l^{+} \neq\emptyset\big]
      P_x[N_l^{+} \neq\emptyset],
    \end{split}
  \end{equation}
  where on the last line we omitted the superscript `1' since
  $N^{+,1}_{l,k}$ has the same distribution as $N^+_{l,k}$.
  By Theorem~\ref{thm:diameter}, Proposition~\ref{pro:Harris_L2} and
  \eqref{eq:K(x)-cond} there is a $c = c(x)$ such that
  \begin{equation}
    \label{eq:K-properties}
    P_x\big[N_l^{+} \neq\emptyset\big] \le cl^{-1},
    \qquad E_x\big[\abs[\big]{N_l^{+}}^2 \,\big|\, N_l^{+}
      \neq\emptyset\big]^{1/2} \le c l,
    \qquad E_x\big[\Lambda_n\big] \le cn^{1/2}.
  \end{equation}
  Further, for $B_n(f,\varepsilon, \rho)$ as in
  Lemma~\ref{lem:f-average-tech}, it holds that
  \(
    A_{l,k} \cap \set{|N^+_{l, k}| \ge l\rho}
    \subseteq B_l(f,\varepsilon, \rho)
  \).
  Therefore, choosing $\rho = \delta/(6Cc^2)$ and decomposing on whether
  $|N_{l,n-\tau_i}^{+}|$ is larger than $l \delta/(6Cc^2)$, we obtain
  \begin{equation}
    \begin{split}
      \label{eq:Q_n-calc2}
      E_x\big[&1_{A^1_{l, n - \tau_i}} \abs{N_{l,n-\tau_i}^{+,1}} \big|
        N_l^{+,1} \neq\emptyset\big] \\
      &\le \frac{l\delta}{6Cc^2}
      + E_x\big[1_{B_n(f, \varepsilon, \rho )}
        \abs{N_{l,n-\tau_i}^+} \big| N_l^+ \neq\emptyset\big] \\
      &\le \frac{l\delta}{6Cc^2}
      + E_x\big[\abs{N_l^+}^2 \big| N_l^+ \neq\emptyset\big]^{1/2}
      P_x\big[B_l(f, \varepsilon, \rho )\big| N_l^+ \neq\emptyset\big]^{1/2},
    \end{split}
  \end{equation}
  where in the last step we used the Cauchy--Schwarz inequality and
  $|N_{l,n-\tau_i}^+| \le |N_l^+|$.
  Combining \eqref{eq:Q_n-calc1} and
  \eqref{eq:Q_n-calc2} with the first two claims in
  \eqref{eq:K-properties} we obtain
  \begin{equation}
    E_x\big[1_{A_{l,n}^i}\abs{N_{l,n}^{+,i}}\big] \le
    \Big(\frac{\delta}{6Cc}
      + c^2 P_x\big[B_l(f, \varepsilon, \rho )
        \,\big|\,N_l^+ \neq\emptyset\big]^{1/2}\Big)P_x[\tau_i \le n].
  \end{equation}
  Inserting this into \eqref{eq:pxan},
  \eqref{eq:P_x-times-mu_n[A_n]-calc}, we obtain that $P_x[A_n]$ is
  bounded by
  \begin{equation}
    \frac{2\delta }{3}
    +\frac{1}{n} \Big(\frac{\delta}{6Cc}
      + c^2 \sup_{l \ge R(n)}P_x\big[B_l(f, \varepsilon, \rho )
        \,\big|\,N_l^+ \neq\emptyset\big]^{1/2}\Big)
    \sum_{l=R(n)}^{\floor{C\sqrt{n}}} \sum_{i=1}^n P_x[\Lambda_n \ge i],
  \end{equation}
  Since $R(n)$ diverges, Lemma~\ref{lem:f-average-tech} implies that
  the supremum in this formula tends to zero as
  $n \to \infty$. By the last property in
  \eqref{eq:K-properties},
  $\sum_{l=R(n)}^{\floor{C\sqrt{n}}} \sum_{i=1}^n P_x[\Lambda_n\ge i]
  \le C \sqrt{n}E_x[\Lambda_n] \le Ccn$.
  Therefore
  \begin{equation}
    P_x [A_n] \le \frac{2\delta}{3}
    + \Big(\frac{\delta}{6} + Cc^3 o(1)\Big) \le \delta
  \end{equation}
  for $n$ large enough. Since $\varepsilon$ and $\delta $ are arbitrary,
  this shows \eqref{eq:Q_n-toshow} and with the initial comment completes
  the proof.
\end{proof}

We are now ready to prove Proposition~\ref{prop:S_n-scaling-limit}.

\begin{proof}[Proof of Proposition~\ref{prop:S_n-scaling-limit}]
  We apply a martingale functional central limit theorem, see
  e.g.~\cite[Theorem~8.2.8]{Durrett19}. To check its assumptions we need
  to show that
  \begin{enumerate}
    \item
    \(
      \lim_{n \to \infty}n^{-1}
      \sum_{k=1}^{n} E_x[(S_k - S_{k-1})^2 | \mathcal{H}_{k-1}]
      = \frac{\inp{\chi}{\mathcal{V}}}{\inp{\chi}{1}}
    \)
    in $P_x$-probability,
    \item
    \(
      \lim_{n \to \infty}n^{-1} \sum_{k=1}^{n}
      E_x[(S_k - S_{k-1})^2 1_{\set{|S_k - S_{k-1}|>\varepsilon\sqrt{n}}}]
      = 0.
    \)
  \end{enumerate}

  Condition (a) follows from Lemma~\ref{lem:S_k-square-diff} and
  Proposition~\ref{prop:Q-limit}, using that $\mathcal{V}$ is bounded by
  Lemma~\ref{lem:var-bound}.

  We now show (b). By the Cauchy--Schwarz inequality,
  \(
    E_x[(S_k - S_{k-1})^2
      1_{\set{|S_k - S_{k-1}|>\varepsilon\sqrt{n}}}]
    \le E_x[(S_k - S_{k-1})^4]^{1/2}
    P_x[(S_k - S_{k-1})^2 > \varepsilon^2n]^{1/2}
  \).
  By Lemmas~\ref{lem:S_k-square-diff}, \ref{lem:var-bound},
  $E_x[(S_k - S_{k-1})^2] \le c$ uniformly in $k$ and thus
  $P_x[(S_k - S_{k-1})^2 > \varepsilon^2n] \le c \varepsilon^{-2} n^{-1}$.
  Thus, to prove (b) it is enough to show
  \begin{equation}
    \label{eq:S_n-scaling-toshow-b}
    E_x[(S_k - S_{k-1})^4] \le c
  \end{equation}
  uniformly in $k$. By \eqref{eq:S-diff},
  \(
    S_k - S_{k-1}
    = \sum_{v \in \desc(v_{k-1})} \chi(\varphi_v) - \chi(\varphi_{v_{k-1}})
  \),
  and thus
  \begin{equation}
    \begin{split}
      \label{eq:S-diff^4}
      E_x[(S_k - S_{k-1})^4 | \mathcal{H}_{k-1}]
      &= E_{\varphi_{v_{k-1}}}
      \Big[\Big(\sum_{v\in N_1^+} \chi(\varphi_v)
        - \chi(\varphi_{v_{k-1}})\Big)^4\Big] \\
    &= E_{\varphi_{v_{k-1}}}
    \Big[\Big(\sum_{v\in N_1^+} \chi(\varphi_v)
        - E_{\varphi_{v_{k-1}}}\big[\sum_{v\in N_1^+}
          \chi(\varphi_v)\big]\Big)^4\Big],
    \end{split}
  \end{equation}
  where we used that
  \(
    \chi(\varphi_{v_{k-1}})
    = E_{\varphi_{v_{k-1}}}[\sum_{v \in N_1^+} \chi(\varphi_{v})]
  \)
  in the last line. We write $N_1^+ = \set{w_1, \dots, w_d}$ and follow
  similar arguments as in the proof of Lemma~\ref{lem:var-bound}, using
  that $\chi(x) = 0$ for $x < h^*$ and exploiting the independence of the
  $\varphi_{w_i}$, to obtain that \eqref{eq:S-diff^4} is bounded above by
  \(
    c_d E_{\varphi_{v_{k-1}}}[(\chi(\varphi_{w_1})
        - E_{\varphi_{v_{k-1}}}[\chi(\varphi_{w_1})])^4]
  \).
  Since $\chi$ is Lipschitz on $[h^*, \infty)$,
  $\chi(\varphi_{w_1})-E_x[\chi(\varphi_{w_1})]$ is sub-Gaussian, and thus
  $E_x[(\chi(\varphi_{w_1})-E_x[\chi(\varphi_{w_1})])^4] \le c$, where
  the constant $c$ only depends on $d$ and the Lipschitz constant of
  $\chi$ but not on $x$. Statement \eqref{eq:S_n-scaling-toshow-b} then
  follows by taking the expectation of \eqref{eq:S-diff^4}.
\end{proof}

\subsection{Further invariance principles}

We now prove several further convergence results that are a consequence
of Proposition~\ref{prop:S_n-scaling-limit} and which will be useful
later.

\begin{proposition}
  \label{prop:joint-scaling-limit}
  Let $B_t$ be a standard Brownian motion, let $L_t^0$ be its local time
  at $0$. Then, as $n \to \infty$, in  $P_x$-distribution with respect to
  the Skorokhod topology,
  \begin{equation}
    \bigg(\frac{\sum_{w \in Y(v_{\floor{nt}})} \chi(\varphi_w)}{\sqrt{n}},
      \frac{\Lambda(v_{\floor{nt}})}{\sqrt{n}}\bigg)_{t \ge 0}
    \to \Big(\sigma |B_t|, \frac{\sigma}{\chi(x)}L_t^0\Big)_{t \ge 0}.
  \end{equation}
\end{proposition}

To show Proposition~\ref{prop:joint-scaling-limit} from
Proposition~\ref{prop:S_n-scaling-limit}, we need the following
lemma. It states that (under scaling) the first summand in the
definition \eqref{eq:S_n-def} of $S_n$ is negligible.

\begin{lemma}
  \label{lem:chi/sqrt(n)-conv}
  In $P_x$-distribution with respect to the Skorokhod topology,
  \begin{equation}
    \lim_{n\to\infty }
    \Big(\frac{\chi(\varphi_{v_{\floor{nt}}})}{\sqrt{n}}\Big)_{t \ge 0}
    = 0.
  \end{equation}
\end{lemma}

\begin{proof}
  As $\chi$ grows at most linearly (see  \eqref{eq:crit-chi_bounds}), it
  is enough to show that for any fixed $t > 0$ and $\varepsilon > 0$,
  \begin{equation}
    \label{eqn:hhh}
    \lim_{n\to\infty}
    P_x\Big[\max_{i \le nt}\varphi_{v_i}/\sqrt{n} > \varepsilon\Big]
    = 0.
  \end{equation}
  Let $H(T^i) = \max_{v \in T^i} H(v)$ denote the height of tree $T^i$.
  For every $j,k\in \mathbb N$, the probability in \eqref{eqn:hhh} is
  bounded from above by
  \begin{equation}
    \begin{split}
      \label{eq:jjj}
    P_x\Big[&\max_{i \le nt}\varphi_{v_i} > \varepsilon \sqrt{n},\
      \Lambda_{\floor{nt}} \le j,\ H(T_i) \le k \text{ for }i=1,\dots, j\Big]
    \\ &+
    P_x[\exists i \in \set{1, \dots, j} :
      H(T^i)>k,\ \Lambda_{\floor{nt}}\leq j]
    + P_x[\Lambda_{\floor{nt}} > j].
    \end{split}
  \end{equation}
  To show \eqref{eqn:hhh}, we thus need to choose
  $j_n \to \infty$ and $k_n \to \infty$ so that all three summands
  converge to zero.

  We start with the first summand. Restricted to
  $\set{\Lambda_{\floor{nt}} \le j}$ and
  $\set{H(T^i) \le k \text{ for } i=1,\dots, j}$,
  $\max_{k \le nt} \varphi_{v_k}$ is dominated by the maximum of all
  $\varphi_v$ with $v$ such that $\Lambda(v) \le j$ and $H(v) \le k$.
  Considering not only the maximum over the connected components of the
  level set, but over the first $k$ generations in $j$ copies of the
  whole forward tree $\mathbb{T}^+$, this is dominated by the maximum of
  $j d^k$ non-negatively correlated Gaussian random variables with mean
  at most $x$ and bounded variance. By Gaussian comparison techniques,
  the mean of this maximum is of order
  $x + c \sqrt {\log j d^k} \le c \sqrt {k \log j}$ for $j,k$ large
  enough. Thus, by the Markov inequality,
  \begin{equation}
    P_x\Big[\max_{i \le nt}\varphi_{v_i} > \varepsilon \sqrt{n},\
      \Lambda_{\floor{nt}} \le j,\ H(T_i) \le k
      \text{ for } i=1,\dots, j\Big] \\
    \le \frac{C}{\varepsilon n^{1/2}} \sqrt{k\log(j)}.
  \end{equation}
  By a union bound, using Theorem~\ref{thm:diameter}, the second
  summand in \eqref{eq:jjj} satisfies
  \begin{equation}
    P_x[\exists i \in \set{1, \dots, j}
      : H(T^i) > k,\ \Lambda_{\floor{nt}} \le j]
    \le
    c(x)j/k.
  \end{equation}
  Finally, by the Markov inequality and \eqref{eq:K(x)-cond}, the third
  summand can be bounded by
  \begin{equation}
     P_x[\Lambda_{\floor{nt}} > j] \le c(x)\sqrt{tn}/j.
  \end{equation}
  Setting now, e.g., $j_n = n^{2/3}$ and $k_n = n^{5/6}$, all three
  summands in \eqref{eq:jjj} converge to $0$ as required.
\end{proof}

\begin{proof}[Proof of Proposition~\ref{prop:joint-scaling-limit}]
  We use the continuous mapping theorem together with the already
  established convergence statements. By
  Proposition~\ref{prop:S_n-scaling-limit},
  $n^{-1/2}S_{\floor{n\cdot}} \to \sigma B_\cdot$, and by
  Lemma~\ref{lem:chi/sqrt(n)-conv},
  $n^{-1/2}\chi(\varphi_{v_{\floor{n\cdot}}})\to 0$ as $n \to \infty$, in
  $P_x$-distribution, in the Skorokhod topology. Therefore
  \begin{equation}
    n^{-1/2}(S_{\floor{n\cdot}} - \chi(\varphi_{v_{\floor{n\cdot}}}))
    = n^{-1/2}\Big(\sum_{w\in Y(v_{\floor{n\cdot}})} \chi(\varphi_w)
      - \sum_{i \le \Lambda_{\floor{n\cdot}}} \chi(\varphi_{o^i})\Big)
  \end{equation}
  also converges to $\sigma B_\cdot$ as $n \to \infty$. Next, note that
  \begin{equation}
    \inf_{k \le n}\set{S_{k} - \chi(\varphi_{v_k})}
    = \inf_{k \le n}\set[\Big]{ \sum_{w \in Y(v_k)} \chi(\varphi_w)
      - \sum_{i \le \Lambda_k} \chi(\varphi_{o^i})}
    = -\sum_{i \le \Lambda_n} \chi(\varphi_{o^i}).
  \end{equation}
  Therefore, setting $\underline{B}_t = \inf_{s \le t}B_s$ and using the
  continuous mapping theorem with the map
  $g(X_t) = (X_t - \inf_{s \le t} X_s, -\inf_{s \le t} X_s)$,
  \begin{equation}
    \label{eq:joint-convergence1-step}
    n^{-1/2}\Big(\sum_{w \in Y(v_{\floor{n\cdot}})} \chi(\varphi_w),
      \sum_{i \le \Lambda_{\floor{n\cdot}}} \chi(\varphi_{o^i})\Big)
    \to \big(\sigma (B_\cdot - \underline{B}_\cdot), - \sigma
      \underline{B}_\cdot\big)
  \end{equation}
  in distribution as $n \to \infty$. By Lévy's theorem (see, e.g.,
    Theorem~VI.2.3 in \cite{RevYor99}), the right-hand side of
  \eqref{eq:joint-convergence1-step} has the same distribution as
  $(\sigma |B_\cdot|, \sigma L_\cdot^0(B))$. Together with the fact that
  under $P_x$,
  \(
    \sum_{i \le \Lambda_{\floor{nt}}} \chi(\varphi_{o^i})
    = \chi(x) \Lambda_{\floor{nt}}
  \),
  this implies the proposition.
\end{proof}

Another consequence of
Proposition~\ref{prop:S_n-scaling-limit} is the following
scaling limit result for $S_n$
conditioned to reach a certain height on the first tree $T_1$. To
this end we define
\begin{equation}
  \label{eq:S-bar-def}
  \bar{S}_n =
  \begin{cases}
    \sum_{w \in Y(v_n)} \chi(\varphi_w)\qquad
     &\text{if $n \le |T^1|$}, \\
    0 &\text{if $n > |T^1|$}.
  \end{cases}
\end{equation}
Note that, up to an additive correction $\chi (\varphi_{v_n}) - \chi (x)$,
$\bar S_\cdot$ is equal to $S_\cdot$ restricted to $T_1$.

\begin{proposition}
  \label{prop:conditioned-scaling-limit}
  For $y > 0$, let $(\mathbf{e}^{\ge y/\sigma})_{t\ge 0}$ be a Brownian
  excursion conditioned to reach at least height $y/\sigma$. Then, in
  distribution under $P_x[\, \cdot \, | \sup_k \bar{S}_k \ge \sqrt{n}y]$,
  with respect to the Skorokhod topology,
  \begin{equation}
    \lim_{n\to\infty}
    \big(n^{-1/2}\bar{S}_{\floor{nt}}\big)_{t \ge 0}
    = (\sigma \mathbf{e}^{\ge y/\sigma})_{t\ge 0}.
  \end{equation}
\end{proposition}

We omit the proof of this proposition as it would be identical to the
proof of Proposition~6.13 in \cite{Pow19}, which itself follows
\cite[Proposition 2.5.2]{DLG02}.

\section{\texorpdfstring{Relation of $S_n$ and $H_n$}{Relation of Sn and Hn}}
\label{sec:H_n-S_n-connection}

The aim of this section will be to establish a connection between the
martingale $S_n$ and the height process $H_n$ (see \eqref{eq:S_n-def}
  and \eqref{eq:def-Lambda,H} for definitions). This will be useful in
the proof of Theorem~\ref{thm:inv-principle} in
Section~\ref{sec:inv-principle}. Throughout the whole section, we only
consider the field on the first tree $(T, \varphi) =(T^1, \varphi^1)$ in
the infinite tree sequence
$\mathbf{T} = ((T^1, \varphi^1), (T^2, \varphi^2), \dots)$ introduced in
Section~\ref{sec:S_n-process}. As a consequence, we only consider $\bar{S}$
introduced in \eqref{eq:S-bar-def} instead of $S$. We will see that,
approximately, $\bar{S}(v) \approx H(v)/ C_1$ with
$C_1$ as in \eqref{eq:C_1-def}.

Motivated by this, for $\eta > 0$ we say that $v \in T$ is \emph{$\eta$-bad}
if
\begin{equation}
  \abs[\Big]{\frac{\bar{S}(v)}{H(v)} - C_1^{-1} } > \eta.
\end{equation}
Fixing in addition $R > 0$, we say that $v \in T$ is \emph{$(\eta, R)$-bad}
if there exists a $w \prec v$ such that $H(w) \ge R$ and $w$ is $\eta$-bad.
This means that $v$ is $(\eta, R)$-good (i.e.~not $(\eta ,R)$-bad) if all
its ancestors in generations at least $R$ are $\eta$-good. We set
\begin{equation}
  \label{eq:Nbad}
  N_n^{(\eta,R)} = \set{v \in N_n^+ : \text{ $v$ is $(\eta, R)$-bad}}.
\end{equation}

The first main result of this section shows that this set is relatively
small.

\begin{proposition}
  \label{prop:|N_n^bad|/|N_n|-fraction}
  For every $\varepsilon > 0$ and $x \ge h^*$,
  \begin{equation}
    \label{eq:psevenone}
    \lim_{R\to\infty}
    \sup_{n \ge R}
    P_x\Big[\frac{\abs{N_n^{(\eta,R)}}}{\abs{N_n^+}} > \varepsilon
      \,\Big|\, N_n^+ \neq\emptyset \Big] = 0.
  \end{equation}
\end{proposition}

\begin{proof}
  The proof follows a strategy similar to the proof of Proposition~6.17
  in \cite{Pow19}, with adaptations that are necessary in order to handle
  the unbounded domain in our setting. For $\varepsilon > 0$, $R > 0$ and
  $n \ge 0$ we define the event
  \begin{equation}
    E_{R, n}^\varepsilon
    = \Big\{ \frac{\sum_{v \in N_n^+} \chi(\varphi_v)
        1_{\set{\text{$v$ is $(\eta, R)$-bad}}}}
      {\sum_{v \in N_n^+} \chi(\varphi_v)} > \varepsilon \Big\}.
  \end{equation}
  We first claim that it suffices to show that
  \begin{equation}
    \label{eq:(eta,R)-bad-fraction_toshow}
    \lim_{R\to\infty}
    \sup_{n \ge R} P_x[E_{R,n}^\varepsilon | N_n^+ \neq\emptyset] = 0.
  \end{equation}
  To see that this indeed implies the lemma, we write
  \begin{equation}
    \frac{\abs{N_n^{(\eta, R)}}}{\abs{N_n^+}} =
    \frac{\abs{N_n^{(\eta,R)}}}{\sum_{v \in N_n^+} \chi(\varphi_v)}
    \cdot
    \frac{\sum_{v \in N_n^+} \chi(\varphi_v)} {\abs{ N_n^+}}.
  \end{equation}
  Therefore, for every $\delta > 0$,
  the probability in \eqref{eq:psevenone} is
  bounded from above by
  \begin{equation}
    \label{eq:|N_n^(eta,R)|/|N_n|-bound}
    P_x\Big[\frac{|N_n^{(\eta,R)}|}{\sum_{v \in N_n^+}\chi(\varphi_v)}
      > \delta \,\Big|\, N_n^+ \neq\emptyset \Big]
    +P_x\Big[\frac{|N_n^+|}{\sum_{v \in N_n^+}\chi(\varphi_v)}
      < \delta/\varepsilon \,\Big|\, N_n^+ \neq\emptyset \Big].
  \end{equation}
  By Lemma~\ref{lem:sum_f_average} (with $f=1$), the second probability
  tends to $0$ as $n \to \infty$ if $\delta = \delta(\varepsilon)$ is
  small enough. Since $\chi (x) \ge \chi (h^*) > 0$ for every $x\ge h^*$,
  it holds that
  \(
    \abs{N_n^{(\eta,R)}}
    \le c \sum_{v \in N_n^+}\chi(\varphi_v)
    1_{\set{\text{$v$ is $(\eta,R)$-bad}}}
  \),
  and thus
  \begin{equation}
    P_x\Big[\frac{|N_n^{+,(\eta,R)}|}{\sum_{v \in N_n^+}\chi(\varphi_v)}
      > \delta \Big| N_n^+ \neq\emptyset \Big]
      \le P_x\big[E^{\delta/c}_{R,n} \big| N_n^+ \neq\emptyset \big],
  \end{equation}
  which together with \eqref{eq:(eta,R)-bad-fraction_toshow} implies the
  claim of the lemma.

  We now prove \eqref{eq:(eta,R)-bad-fraction_toshow}. By first using the
  Markov inequality, and then Lemma~\ref{lem:many-to-few_general} with
  $k = 1$ and
  \(
    Y(v_1) = \chi(\varphi_{v_1})1_{\set{\text{$v_1$ is
          $(\eta,R)$-bad}}}{1_{\set{N_n^+ \neq\emptyset}}}/({\sum_{w \in
          N_n^+}\chi(\varphi_w)})
  \),
  it holds
  \begin{equation}
    \begin{split}
      \label{eq:E_R,n^e-bound}
      P_x[E_{R,n}^\varepsilon| N_n^+ \neq\emptyset]
      &\le \varepsilon^{-1}
      E_x\Big[\frac{\sum_{v \in N_n^+}
          \chi(\varphi_v)1_{\set{\text{$v$ is $(\eta, R)$-bad}}}}
        {\sum_{v \in N_n^+} \chi(\varphi_v)}
        \Big| N_n^+ \neq\emptyset\Big] \\
      &= \varepsilon^{-1} E_x\Big[
        \frac{1_{\set{N_n^+ \neq\emptyset}}\sum_{v \in N_n^+} \chi(\varphi_v)
          1_{\set{\text{$v$ is $(\eta, R)$-bad}}}
          }
          {\sum_{v \in N_n^+} \chi(\varphi_v)}\Big]
        P_x[N_n^+ \neq\emptyset]^{-1} \\
        &= \varepsilon^{-1} \Q_x\Big[
          \frac{\chi(x) P_x[N_n^+ \neq\emptyset]^{-1}}
          {\sum_{v \in N_n^+} \chi(\varphi_v)}
        1_{\set{\text{$\sigma_n$ is $(\eta,R)$-bad}}}\Big],
    \end{split}
  \end{equation}
  where, as defined in Section~\ref{sec:notation-and-results}, $\sigma_n$
  is the vertex on the spine in the $n$-th generation. To continue, we first
  show the following two claims:
  \begin{itemize}
    \item[(i)]
    $\sup_{n \ge R} \Q_x[\text{$\sigma_n$ is $(\eta, R)$-bad}] \to 0$ as
    $R \to \infty$.
    \item[(ii)]
    Set
    \(
      Z_n =\frac{\chi(x) P_x[N_n^+ \neq\emptyset]^{-1}}
      {\sum_{v \in N_n^+} \chi(\varphi_v)}
    \).
    Then for all $\delta > 0$, there exist $R', K>0$ such that
    $\Q_x[Z_n 1_{\set{Z_n>K}}] \le \delta$ for all $n \ge R'$.
  \end{itemize}

  Claim (i) will follow from the ergodic behaviour of the field along the
  spine $(\sigma_n)_{n\ge 0}$ under $\Q_x$. Let
  $\sib^{<}(v) = \{w\in \sib(v): v<_{\mathbf{v}} w\}$. Then
  $\bar{S}(\sigma_n)=\sum_{k\le n}\sum_{v\in \sib^<(\sigma_k)}\chi(\varphi_v)$.
  Recall that $\xi_k := \varphi (\sigma_k)$. By conditioning on
  $\xi_{k-1}$, due to the fact that under $Q_x$ the spine mark at
  generation $k$ is uniformly distributed on descendants of its position
  at generation $k-1$ and that non-spine vertices behave as under $P_x$,
  \begin{equation}
    \begin{split}
      \Q_x&\Big[\sum_{v \in \sib^{<}(\sigma_k)} \chi(\varphi_v)\Big]
      = \Q_x\Big[\Q_{\xi_{k-1}}\Big[
            \sum_{v \in \sib^{<}(\sigma_1)} \chi(\varphi_v)\Big]\Big]
      \\&=
      Q_x\Big[\frac{1}{2} \Q_{\xi_{k-1}}
        \Big[\sum_{v \in N_1^+ \setminus \set{\sigma_1}}
          \chi(\varphi_v)\Big]\Big]
      \\&= Q_x\Big[\frac{d-1}{2d} P_{\xi_{k-1}}\Big[\sum_{v \in N_1^+}
          \chi(\varphi_v)\Big]\Big]
      = \frac{d-1}{2d} Q_x[ \chi(\xi_{k-1})].
    \end{split}
  \end{equation}
  By Lemma~\ref{lem:xi_inv-measure} the invariant measure of the Markov
  chain $(\xi_k)_{k\ge 0}$ under $\Q_x$ is given by $\chi(y)^2\nu(\D y)$,
  so the last expression converges to $C_1^{-1}$ as $k\to \infty$.
  Claim (i) then follows from an ergodicity argument.

  To see (ii), note that again by Lemma~\ref{lem:many-to-few_general},
  \begin{equation}
    \begin{split}
      \Q_x[Z_n 1_{\set{Z_n > K}}]
      &= \frac{P_x[1_{\set{Z_n>K}}1_{\set{N_n^+ \neq\emptyset}}]}
      {P[N_n^+ \neq\emptyset]}
      = P_x[Z_n>K| N_n^+ \neq\emptyset].
     \\& \le
     P_x\Big[\frac{\sum_{v \in N_n^+}\chi(\varphi_v)}{n}
       < \frac{1}{Kc} \Big| N_n^+ \neq\emptyset\Big],
    \end{split}
  \end{equation}
  where in the last inequality we used the fact that
  $\chi(x)P_x[N_n^+ \neq\emptyset]^{-1} \ge c$,
  by Theorem~\ref{thm:diameter}.
  By Theorem~\ref{thm:conv-to-exp}, this converges to the probability that an
  exponentially distributed random variable is smaller than $(K c_x)^{-1}$.
  It is thus straightforward to find $K$ and $R'$ such that the last
  probability is smaller than $\delta$ for all $n \ge R'$, proving (ii).

  Returning back to
  \eqref{eq:E_R,n^e-bound}, its right-hand side is bounded by
  \begin{equation}
    \begin{split}
      \Q_x&[Z_n 1_{\set{\text{$\sigma_n$ is $(\eta,R)$-bad}}}] \\
      &= \Q_x[Z_n 1_{\set{Z_n > K}}
        1_{\set{\text{$\sigma_n$ is $(\eta,R)$-bad}}}] +
      \Q[Z_n 1_{\set{Z_n \le K}}
        1_{\set{\text{$\sigma_n$ is $(\eta,R)$-bad}}}] \\
      &\quad\le \Q_x[Z_n 1_{\set{Z_n > K}}]
      + K \Q[\text{$\sigma_n$ is $(\eta,R)$-bad}],
    \end{split}
  \end{equation}
  which together with (i) and (ii) implies
  \eqref{eq:(eta,R)-bad-fraction_toshow} and concludes the proof.
\end{proof}

To state the second main result of this section, we introduce (enlarging
  the probability space if necessary) random variables
$\tau^1, \dots, \tau^k$ which under $P_x$, conditionally on $T$, are
independent and uniformly distributed on $\set{1, 2, \dots, |T|}$. Using
these random variables, we define two $k \times k$ random matrices
\begin{equation}
  \begin{split}
    \label{eq:D-mat-def}
    (D_n^{\bar{S}})_{i,j}
    &= n^{-1}\big(\bar{S}(v_{\tau^i}) + \bar{S}(v_{\tau^j})
      - 2 \bar{S}(v_{\tau^i} \wedge v_{\tau^j})\big) \\
    (D_n^H)_{i,j}
    &= n^{-1}\big(H(v_{\tau^i}) + H(v_{\tau^j})
      - 2 H(v_{\tau^i} \wedge v_{\tau^j})\big),
  \end{split}
\end{equation}
where, as usual, $(v_1, v_2, \dots)$ denotes the depth-first traversal of
$T$, and $v \wedge w$ denotes the most recent common ancestor of $v$ and
$w$ in $T$.

\begin{proposition}
  \label{prop:S-H-mat}
  For every $k \ge 1$ and $\varepsilon > 0$,
  \begin{equation}
    \lim_{n\to\infty}
    P_x\big[\norm{C_1^{-1} D_n^H - D_n^{\bar{S}}} > \varepsilon
      \big| N_n^+ \neq\emptyset\big] = 0,
  \end{equation}
  where $\norm{\cdot}$ denotes the Frobenius norm of
  $k\times k$ matrices.
\end{proposition}

To prove this proposition, we need two lemmas. The first one estimates
from below the size of $T$ conditionally on $\set{N_n^+ \neq \emptyset}$.

\begin{lemma}
  \label{lem:P[|T|<qn^2]-to-zero}
  For every $x \ge h^*$,
  \begin{equation}
    \lim_{q\to 0}\sup_{n\ge 0}
    P_x\big[\abs{T} \le q n^2 | N_n^+ \neq\emptyset\big] = 0.
  \end{equation}
\end{lemma}

\begin{proof}
  It will be sufficient to show that for every $\delta > 0$ there exist
  $q$ and $n_0$ such that for every $n \ge n_0$,
  \begin{equation}
    \label{eq:|T|<Cn^2-toshow}
    P_x[|T| \le q n^2| N_n^+ \neq\emptyset] \le \delta.
  \end{equation}
  Indeed, since there are
  only finitely many $n < n_0$, by decreasing the value of $q$,
  the inequality \eqref{eq:|T|<Cn^2-toshow} can be made true for all
  $n\ge 0$. This then directly implies the claim of the lemma.

  To show \eqref{eq:|T|<Cn^2-toshow}, observe that for every $\eta >0$
  \begin{equation}
    \begin{split}
      \label{eqn:decaa}
      P_x&[|T| \le q n^2| N_n^+ \neq\emptyset]
      \\&\le
      P_x\big[|T| \le q n^2, \abs{N_n^+} \ge \eta n| N_n^+
        \neq\emptyset\big]
      +P_x\big[\abs{N_n^+} < \eta n | N_n^+ \neq\emptyset\big].
    \end{split}
  \end{equation}
  By Theorem~\ref{thm:conv-to-exp}, there is $\eta >0$ small such
  that the second probability on the right-hand side of
  \eqref{eqn:decaa} is bounded by $\delta /2$ for all $n$ large enough. It
  is thus sufficient to show that for every $\delta >0$ and
  $\eta >0$ there is $n_0$ so that for all $n\ge n_0$
  \begin{equation}
    \label{eq:decbb}
    P_x\big[\abs T \le q n^2, \abs{N_n^+} \ge \eta n\big| N_n^+
      \neq\emptyset\big]
    \le \delta /2.
  \end{equation}
  Note first that
  \begin{equation}
    P_x\big[\abs T \le q n^2, \abs{N_n^+} \ge \eta n \big| N_n^+
      \neq\emptyset\big]
    \le
    P_x\big[\abs T \le q n^2 \big| \abs{N_n^+} \ge \eta n\big].
  \end{equation}
  Given $\{\abs{N_n^+} \ge \eta n\}$, let $w_1,\dots, w_{\floor{\eta n}}$ be
  the first $\floor{\eta n}$ vertices in $N_n^+$, and let $T_w$ be the
  subtree of $T$ rooted at $w$. Obviously,
  $\abs T \le \sum_{i=1}^{\floor {\eta n}} \abs {T_{w_i}}$.
  Under $P[\, \cdot\, | \abs{N_n^+} \ge \eta n]$, the random variables
  $\abs{T_{w_i}}$ are independent. Moreover, since $w \in N_n^+$ implies
  $\varphi_w \ge h^*$, by stochastic domination \eqref{eq:stoch_dom}, for
  any $u>0$,
  \begin{equation}
    P_x\big[ \abs{T_{w_i}} \ge u  \big| \abs{N_n^+} \ge \eta n\big]
    \ge P_{h^*}\big[\abs{\mathcal C_o \cap \mathbb T^+} \ge u\big].
  \end{equation}
  Denoting $T_i'$, $i\ge 1$, i.i.d.~random variables distributed as
  $\abs{\mathcal C_o \cap \mathbb T^+}$ under $P_{h^*}$, it follows that
  \begin{equation}
    \label{eq:stable}
    P_x\big[|T| \le q n^2 \big| \abs{N_n^+} \ge \eta n\big] \le
    P\Big[\sum_{i=1}^{\floor{\eta n}}T'_i \le q n^2\Big].
  \end{equation}
  By Theorem~\ref{pro:cl23}, $T'_i$ are in the domain of attraction of
  the $1/2$-stable random distribution. Therefore,
  $n^{-2}\sum_{i\le \floor{\eta n}} T'_i$ converges in distribution to a
  non-negative $1/2$-stable random variable (see e.g. \cite[Theorem
      XIII.6.2]{Fel71}). As a consequence, since the distribution of this
  random variable has no atom at $0$, for any $\delta, \eta >0$ there
  exists $q$ small so that for all $n$ large enough the right-hand side
  of \eqref{eq:stable} is bounded by $\delta/2$. This shows
  \eqref{eq:decbb} and completes the proof.
\end{proof}

The second lemma needed to show Proposition~\ref{prop:S-H-mat} studies
the probability that a randomly chosen vertex
$v_{\tau_1}$ is $(\eta, R)$-bad.

\begin{lemma}
  \label{lem:P[v_tau-is-bad]}
  Let $V := v_{\tau_1}$ be a uniformly chosen vertex of $T$. Then, for
  every $x\ge h^*$ and $\eta >0$,
  \begin{equation}
    \lim_{R \to \infty} \sup_{n \ge R}
    P_x\big[\text{$V$ is $(\eta, R)$-bad }
      \big| N_n^+ \neq\emptyset \big] = 0.
  \end{equation}
\end{lemma}

\begin{proof}
  For arbitrary positive constants $q,h_1<h_2$, it holds
  \begin{equation}
    \begin{split}
      \label{eq:oooo}
      P_x&\big[V\text{ is $(\eta, R)$-bad }
        \big| N_n^+ \neq\emptyset \big]
      \\&\le P_x\big[\abs T \le q n^2
        \big| N_n^+ \neq\emptyset \big]
      + P_x\big[\abs T > qn^2, H(V)< h_1 n
        \big| N_n^+ \neq\emptyset \big]
      \\&\quad+ P_x\big[\abs T > qn^2, H(V)> h_2 n
        \big| N_n^+ \neq\emptyset \big]
      \\&\quad+ P_x\big[\abs T > qn^2, H(V)\in [h_1n ,h_2n ],
        \text{$V$ is $(\eta, R)$-bad }
        \big| N_n^+ \neq\emptyset \big].
    \end{split}
  \end{equation}
  We will show that for every $\delta > 0$
  there are $q, h_1, h_2$ such that for every $n\ge 0$ the first three
  summands on the right-hand side are smaller than $\delta /3$, and that
  for every fixed $q, h_1, h_2$ the fourth one satisfies
  \begin{equation}
    \label{eq:v_tau-is-bad_case4_toshow}
    \lim_{R \to \infty} \sup_{n \ge R}
    P_x\big[\abs T > qn^2, H(V) \in [h_1n, h_2n],
      V\text{ is $(\eta,R)$-bad } \big| N_n^+ \neq\emptyset\big] = 0.
  \end{equation}
  This will imply the statement of the lemma.

  Concerning the first summand, the fact that it is possible to choose
  $q>0$ small so that
  $P_x[\abs T \le q n^2 \big| {N_n^+ \neq\emptyset} ] \le \delta /3$ for
  all $n\ge 0$ follows
  immediately from Lemma~\ref{lem:P[|T|<qn^2]-to-zero}. From now on we
  keep $q$ fixed in this way.

  For the second summand, it holds
  \begin{equation}
    \begin{split}
      \label{eq:mmm}
      P_x\big[\abs T > qn^2, H(V)< h_1 n \big| N_n^+ \neq\emptyset \big]
      & \le \frac{1}{P_x[N_n^+ \neq\emptyset] }
      P_x\big[\abs T > qn^2, H(V)< h_1 n \big]
      \\ & \le \frac{1}{P_x[N_n^+ \neq\emptyset] }
      \frac 1{q n^2} E_x\Big[\sum_{k=0}^{h_1 n} \abs{N_k^+}\Big],
    \end{split}
  \end{equation}
  where the last inequality follows from the fact that
  $V$ is a uniformly chosen vertex of $T$, and thus on $\{\abs T \ge qn^2\}$,
  \begin{equation}
    P_x[H_V\le h_1n | \sigma (T)]
    = \abs T^{-1} \sum_{k=0}^{h_1 n}\abs {N_k^+}
    \le \frac 1{qn^2}\sum_{k=0}^{h_1 n}\abs {N_k^+}.
  \end{equation}
  By Theorem~\ref{thm:diameter},
  $P_x[N_n^+ \neq\emptyset]^{-1} \le c(x) n$, and, by
  Proposition~\ref{pro:Harris_L2}, $E_x[|N_k^+|] \le c'(x)$ for all
  $k\ge 0$. Therefore, the right-hand side of \eqref{eq:mmm} can be made
  smaller than $\delta /3$ by choosing $h_1$ small.

  For the third summand in \eqref{eq:oooo}, note that
  $\{H(V) \ge h_2 n\} \subset \{N_{h_2 n}^+ \neq \emptyset\}$, and thus
  \begin{equation}
    P_x\big[\abs T > qn^2, H(V)> h_2 n \big| N_n^+ \neq\emptyset \big]
    \le
    \frac 1  {P_x[N_n^+ \neq\emptyset ]} P_x[N_{h_2 n}^+ \neq \emptyset],
  \end{equation}
  which can be made arbitrarily small by using Theorem~\ref{thm:diameter}
  and choosing $h_2$ large.

  Finally, we show \eqref{eq:v_tau-is-bad_case4_toshow}. Recall the
  notation from \eqref{eq:Nbad}. Since $V$ is a
  uniformly chosen vertex of $T$, for arbitrary $q, h_1, h_2, n, R >0$
  and $\varepsilon  \in (0,1)$, on the event $\{\abs T\ge q n^2\}$ it
  holds that
  \begin{equation}
    \begin{split}
      P_x\big[ & H(V) \in [h_1n, h_2n],
        V\text{ is $(\eta,R)$-bad } \big| \sigma(T,\varphi ) \big]
      = \frac 1{\abs T}
      \sum_{k=h_1 n}^{h_2 n} \abs{N_k^{+,(\eta ,R)}}
      \\& =
      \frac 1{\abs T}
      \sum_{k=h_1 n}^{h_2 n}
      \frac{\abs{N_k^{+,(\eta ,R)}}}{\abs{N_k^+}} \abs{N_k^+}
      \le
      \varepsilon  + \frac 1{qn^2}
      \sum_{k=h_1 n}^{h_2 n}
      1_{\set{{\abs{N_k^{+,(\eta ,R)}}}/{\abs{N_k^+}}\ge \varepsilon}}
      \abs{N_k^+}.
    \end{split}
  \end{equation}
  As a consequence, the probability in \eqref{eq:v_tau-is-bad_case4_toshow}
  satisfies
  \begin{equation}
    \label{eq:v_tau-is-bad_mid-bound}
    \begin{split}
      P_x\big[&\abs T > qn^2, H(V) \in [h_1n, h_2n],
        V\text{ is $(\eta,R)$-bad } \big| N_n^+ \neq\emptyset\big]
      \\&\le
      \varepsilon  + \frac{1}{qn^2}
      \sum_{k=h_1 n}^{h_2 n}
      E_x\big[ 1_{\set{{\abs{N_k^{+,(\eta ,R)}}}
            /{\abs{N_k^+}}\ge \varepsilon}}
        \abs{N_k^+} \big| N_n^+ \neq\emptyset\big].
    \end{split}
  \end{equation}
  By the Cauchy--Schwarz inequality, every summand in
  \eqref{eq:v_tau-is-bad_mid-bound} is bounded by
  \begin{equation}
    E_x\big[\abs{N_k^+}^2 \big| N_n^+ \neq\emptyset\big]^{1/2}
    P_x\big[\abs{N_k^{+,(\eta,R)}}/\abs{N_k^+}
      > \varepsilon, N_k^+ \neq\emptyset
      \big| N_n^+ \neq\emptyset\big]^{1/2}.
  \end{equation}
  By Theorem~\ref{thm:diameter} and Proposition~\ref{pro:Harris_L2},
  $E_x[|N_k^+|^2\,|\, N_n^+\neq\emptyset]^{1/2} \le c(x) n$. The second
  term satisfies,
  \begin{equation}
    P_x\Big[\frac{\abs{N_k^{+,(\eta,R)}}}{\abs{N_k^+}}
      > \varepsilon, N_k^+ \neq\emptyset \Big| N_n^+ \neq\emptyset\Big]
    \le P_x\Big[ \frac{\abs{N_k^{+,(\eta,R)}}}{\abs{N_k^+}}
      > \varepsilon\Big| N_k^+ \neq\emptyset\Big]
    \frac{P_x[N_k^+ \neq\emptyset]}{P_x[N_n^+ \neq\emptyset]}.
  \end{equation}
  By Theorem~\ref{thm:diameter}, for $k\ge h_1 n$, the fraction on the
  right-hand side can be bounded by a constant $c(x)$. Noting also that
  $\abs{N_k^{+,(\eta,R)}} = 0$ for $R > k$, we obtain that
  \eqref{eq:v_tau-is-bad_mid-bound} is bounded by
  \begin{equation}
    \varepsilon  + \frac 1{qn^2} c(x) n^2 (h_2 - h_1)
    \sup_{k\ge R}
    P_x\Big[ \frac{\abs{N_k^{+,(\eta,R)}}}{\abs{N_k^+}}
      > \varepsilon\Big| N_k^+ \neq\emptyset\Big]^{1/2}.
  \end{equation}
  By Proposition~\ref{prop:|N_n^bad|/|N_n|-fraction}, the supremum here
  converges to $0$ as $R \to \infty$. Since $\varepsilon >0$ is
  arbitrary, this shows \eqref{eq:v_tau-is-bad_case4_toshow} and
  completes the proof.
\end{proof}

We can now prove Proposition~\ref{prop:S-H-mat}. The proof
resembles the one of Proposition~6.21 in \cite{Pow19}, with modifications
required to account for the unbounded domain in our setting.

\begin{proof}[Proof of Proposition~\ref{prop:S-H-mat}]
  It is enough to show the statement for $k = 2$. The general case where
  $k > 2$ is then obtained by a union bound. In the case $k=2$, since the
  matrices are symmetric and the diagonal entries are zero, it is
  sufficient to control one off-diagonal entry. To this end, for
  $\varepsilon >0$ let $G_n$ be the event
  \begin{equation}
    G_n : = \set[\big]{|C_1^{-1} (D_n^H)_{1,2} - (D_n^{\bar{S}})_{1,2}| >
      \varepsilon}.
  \end{equation}
  In order to bound the probability of $G_n$, we will define events
  $A_n(\delta)$ (see \eqref{eq:P[mat-norm]-A_n(delta)-def}) satisfying
  $P[A_n(\delta)| N_n^+ \neq\emptyset] \le \delta$ and show that for $n$
  large enough and the right choice of $\delta_n \to 0$, $G_n$ is a
  subset of $A_n(\delta_n)$. This will imply the statement of the lemma.

  In order to define the events $A_n(\delta)$ for $\delta > 0$, we first fix
  $M = M(\delta ) \ge 1$ so that
  \begin{equation}
    \lim_{n \to \infty}
    P_x[N_{Mn}^+\neq \emptyset  |
      N_n^+ \neq\emptyset] \le \delta/3,
  \end{equation}
 which is possible by Theorem~\ref{thm:diameter}. We then set
 \begin{equation}
   \eta = \eta(\delta) := \varepsilon/(4M),
 \end{equation}
 and, using Lemma~\ref{lem:P[v_tau-is-bad]} and the
 independence of $\tau^1$ and $\tau^2$, we fix $R = R(\delta)$ so that
 \begin{equation}
   \lim_{n \to \infty}
   P_x\big[\text{$v_{\tau^1}$ or $v_{\tau^2}$ is $(\eta, R)$-bad}
     \big| N_n^+ \neq\emptyset\big] \le \delta/3.
 \end{equation}
 Next, we fix $K = K(\delta)$ such that
 \begin{equation}
	 \label{eq:S-H-mat_K-exists}
   \lim_{n \to \infty}
   P_x[\sup \set{\chi(\varphi_u) : u
	 \in \cup_{k=0}^R N_k^+} \ge K | N_n^+ \neq\emptyset] \le \delta/3.
 \end{equation}
 This is possible since the event in \eqref{eq:S-H-mat_K-exists} is
 $\mathcal{F}_R$-measurable, and thus, by
 Proposition~\ref{prop:P-conditioned-is-Q}, the left-hand side of
 \eqref{eq:S-H-mat_K-exists} equals
 $Q_x[\sup \set{\chi(\varphi_u) : u \in \cup_{k=0}^R N_k^+} \ge K]$
 (with $Q_x$ as in Section~\ref{ss:spined-bp}). It
 is then straightforward to choose $K$ large enough so that
 \eqref{eq:S-H-mat_K-exists} is satisfied.
 Finally, we define the events $A_n$ by
  \begin{equation}
    \begin{split}
      \label{eq:P[mat-norm]-A_n(delta)-def}
      A_n = A_n(\delta)
      ={} &\set{N_{Mn}^+ \neq \emptyset }
        \cup \set{\text{$v_{\tau^1}$ or $v_{\tau^2}$ is $(\eta,R)$-bad}}
        \\&  \cup \set{\sup \set{\chi(\varphi_u) : u \in \cup_{k=0}^R N_k^+}
          \ge K}.
    \end{split}
  \end{equation}
  By the choice of $M$, $\eta$, $R$, and $K$, it holds that
  \(
    \lim_{n \to \infty} P_x[A_n(\delta) | N_n^+ \neq\emptyset] \le \delta
  \).
  Therefore, for sequences $\delta_n$ converging to zero sufficiently
  slowly, it also holds that
  \begin{equation}
    \label{eq:P[mat-norm]-step}
    \lim_{n \to \infty} P_x[A_n(\delta_n)| N_n^+ \neq\emptyset] = 0.
  \end{equation}

  Next, we show that $G_n \subseteq A_n(\delta_n)$ for a well-chosen
  $\delta_n$ and large $n$. Let
  $B(\delta ) = \set{H(v_{\tau^1} \land v_{\tau^2}) > R}$.
  We will show that for large enough $n$ and a suitable choice of
  $\delta_n$,
  \begin{equation}
    \label{eq:P[mat-norm]-step2}
    G_n \cap A_n^c(\delta_n) \cap B(\delta_n) = \emptyset
    \quad \text{and} \quad
    G_n \cap A_n^c(\delta_n) \cap B^c(\delta_n) = \emptyset.
  \end{equation}
  From this, it easily follows that $G_n \cap A_n^c(\delta_n) = \emptyset$,
  and thus, $G_n \subseteq A_n(\delta_n)$. Then, if $\delta_n \to 0$
  sufficiently slowly, by \eqref{eq:P[mat-norm]-step},
  \begin{equation}
    \lim_{n \to \infty} P_x[G_n| N_n^+ \neq\emptyset]
    \le \lim_{n \to \infty} P_x[A_n(\delta_n) | N_n^+ \neq\emptyset] = 0,
  \end{equation}
  which is sufficient to prove the statement of the proposition.

  It remains to show the existence of $\delta_n$ such that
  \eqref{eq:P[mat-norm]-step} and \eqref{eq:P[mat-norm]-step2} hold. By
  definition, on $A_n^c$ it holds
  \begin{equation}
    \label{eq:P[mat-norm]-A^c-props}
    H(v_{\tau^i}) \le Mn \quad\text{and}\quad
    \Big|\frac{\bar{S}(v_{\tau^i})}{H(v_{\tau^i})} - C_1^{-1} \Big| \le \eta
    \qquad \text{ for $i =1,2$.}
  \end{equation}
  Thus, on $A^c_n \cap B$ it holds that
  $R < H(v_{\tau^1} \land v_{\tau^2}) \le Mn$. Since $v_{\tau^1}$ is not
  $(\eta,R)$-bad,
  \(
    \abs{\bar{S}(v_{\tau^1} \wedge v_{\tau^2})
      /H(v_{\tau^1} \wedge v_{\tau^2}) - C_1^{-1}} \le \eta
  \).
  Together with \eqref{eq:P[mat-norm]-A^c-props}, the definition of the
  matrices in \eqref{eq:D-mat-def}, and the definition of $\eta$, this
  implies
  \begin{equation}
    \abs{(D_n^{\bar{S}})_{1,2} - C_1^{-1} (D_n^H)_{1,2}}
    \le \frac{1}{n} 4Mn \eta \le \varepsilon,
  \end{equation}
  and thus $G_n \cap B \cap A_n^c = \emptyset$, proving the first half of
  \eqref{eq:P[mat-norm]-step2}. Next, on event $B^c \cap A_n^c$ it holds
  $H(v_{\tau^1} \land v_{\tau^2}) \le R$ and
  $\sup \set{\chi (\varphi_u) : u \in \cup_{k \le R} N^+_k } < K$. Therefore,
  $\bar{S}(v_{\tau^1} \land v_{\tau^2}) \le K d^R$, and as a consequence
  $|\bar{S}(v_{\tau^1} \land v_{\tau^2})
  - C_1^{-1} H(v_{\tau^1} \land v_{\tau^2})| \le K d^R + R C_1^{-1} $.
  We choose a sequence $\delta_n$ converging to $0$ sufficiently
  slowly so that \eqref{eq:P[mat-norm]-step} is satisfied and
  $K(\delta_n) d^{R(\delta_n) +1} + R(\delta_n) C_1^{-1}  \le \varepsilon n$,
  and thus
  \begin{equation}
    |(D_n^{\bar{S}})_{1,2} - C_1^{-1} (D_n^H)_{1,2}|
    \le \frac{\varepsilon}{2}
    + \frac{K(\delta_n)d^{R(\delta_n)}+R(\delta_n)C_1^{-1} }{n}
    \le \varepsilon
  \end{equation}
  for $n$ large. Hence,
  $G_n \cap B^c(\delta_n) \cap A_n^c(\delta_n) = \emptyset$ for such $n$,
  completing the proof of \eqref{eq:P[mat-norm]-step2}.
\end{proof}

\section{Proof of Theorem~\ref{thm:inv-principle}}
\label{sec:inv-principle}

In this section, we prove Theorem~\ref{thm:inv-principle}, relying on
Lemmas~\ref{lem:GHP-tightness} and~\ref{lem:GP-convergence}. The argument
follows the framework of \emph{random metric measure spaces}, with key
results originating from \cite{GrePfaWin09} and \cite{AbDeHo13}. The
proof closely mirrors that of Theorem~1.1 in \cite[Section~6.3.2]{Pow19},
which itself builds on the aforementioned results.

Recall from Section~\ref{sec:model} that $\mathbf{e}$ denotes a Brownian
excursion conditioned to reach height $1$. Also define $\hat{\mathbf{e}}$
to be $(\sigma C_1)$ times a Brownian excursion conditioned to reach
height $(C_1 \sigma)^{-1}$, where the constants $\sigma$ and $C_1$ are
given in \eqref{eq:sigma-def} and \eqref{eq:C_1-def}, and write
$(T_{\hat{\mathbf{e}}}, d_{\hat{\mathbf{e}}})$ for the real tree with
contour function $\hat{\mathbf{e}}$. By Brownian scaling,
$(T_{\hat{\mathbf{e}}}, d_{\hat{\mathbf{e}}})$ is isometrically
equivalent to $(T_{\mathbf{e}}, d_{\mathbf{e}})$. Therefore, to show
Theorem~\ref{thm:inv-principle}, it is enough to prove that if
$(T_{x,n}, d_{x,n})$ is a random metric space whose law coincides with
that of $(\mathcal{C}_o^{h^*} \cap \mathbb{T}^+, d/n)$ under
$P_x[ \cdot | N_n^+ \neq\emptyset]$ then
\begin{equation}
  \label{eq:inv-princ_toshow}
  (T_{x,n}, d_{x,n}) \to (T_{\hat{\mathbf{e}}}, d_{\hat{\mathbf{e}}})
  \quad \text{as $n \to \infty$}
\end{equation}
in distribution with respect to the Gromov-Hausdorff topology.

It will prove useful to additionally equip $(T_{x,n}, d_{x,n})$ and
$(T_{\hat{\mathbf{e}}}, d_{\hat{\mathbf{e}}})$ with measures and work in
the space of metric measure spaces which we will call $\mathbb{X}$. For
this purpose, we equip $(T_{x,n}, d_{x,n})$ with the measure $\mu_{x,n}$
which is the uniform measure among the vertices in
$\mathcal{C}_o^{h^*} \cap \mathbb{T}^+$. We equip
$(T_{\hat{\mathbf{e}}}, d_{\hat{\mathbf{e}}})$ with the measure
$\mu_{\hat{\mathbf{e}}}$ which is the uniform measure on $[0, \hat{\tau})$
and $\hat{\tau}$ is the length of $\hat{\mathbf{e}}$. We write $P_x^n$
for the law of $(T_{x,n}, d_{x,n}, \mu_{x,n})$ and $P_{\hat{\mathbf{e}}}$
for the law of
$(T_{\hat{\mathbf{e}}}, d_{\hat{\mathbf{e}}}, \mu_{\hat{\mathbf{e}}})$,
and use $E_x^n$ and $E_{\hat{\mathbf{e}}}$ for the corresponding
expectations.

\begin{lemma}
  \label{lem:GHP-tightness}
  For fixed $x \ge h^*$, the family $(T_{x,n}, d_{x,n}, \mu_{x,n})_{n \ge
  0}$ is tight with respect to the Gromov-Hausdorff-Prokhorov metric.
\end{lemma}

\begin{proof}
  Fix $x \ge h^*$. Our aim is to show that for every $\varepsilon > 0$
  there exists a relatively compact subset $\mathbb{K} \subset \mathbb{X}$
  (with respect to the Gromov-Hausdorff-Prokhorov metric) such that
  \begin{equation}
    \label{eq:GHP-tightness}
    \inf_{n \ge 0} P_x^n[(T_{x,n}, d_{x,n}, \mu_{x,n}) \in \mathbb{K}]
    \ge 1-\varepsilon.
  \end{equation}
  To do this, we define the relatively compact set (where the relative
    compactness follows from Theorem~2.6 in \cite{AbDeHo13})
  \begin{equation}
    \mathbb{K}_{R, M}
    = \bigg\{ (X,r,\mu) \in \mathbb{X} \ \bigg\vert
      \begin{array}{l}
      \mu(X)=1,\ \diam(X) \le 2R,\ \text{for all $k\ge 1$, $X$ can be} \\
      \text{covered by fewer than $2^{4k}M$ balls of radius $2^{-k}$}
    \end{array}\bigg\}
  \end{equation}
  and show that for $\varepsilon > 0$ there are $R(\varepsilon)$ and
  $M(\varepsilon)$ such that for $\mathbb{K} = \mathbb{K}_{R, M}$,
  \eqref{eq:GHP-tightness} is satisfied.
  Define the events
  \begin{align}
    &A_{x,n}(R) = \set{\diam(T_{x,n}) \le 2R} \quad \text{and} \\
    &B_{x,n}^\delta(M) = \set{\text{$T_{x,n}$ can be covered with
        $<\delta^{-4} M$ balls of radius $\delta$}}.
  \end{align}
  Then, since $\mu_{x,n}(T_{x,n}) = 1$,
  \(
    \set{(T_{x,n}, d_{x,n}, \mu_{x,n}) \in \mathbb{K}_{R, M}}
    = A_{x,n}(R) \cap (\cap_{k \ge 1} B_{x,n}^{2^{-k}}(M))
  \),
  and thus
  \(
    P_x^n[(T_{x,n}, d_{x,n}, \mu_{x,n}) \not\in \mathbb{K}_{R, M}]
    \le P_x^n[A_{x,n}(R)^c]
    + \sum_{k \ge 1} P_x^n[B_{x,n}^{2^{-k}}(M)^c \cap A_{x,n}(R)]
  \).
  To see \eqref{eq:GHP-tightness}, it is therefore enough to show that
  for $M = M(\varepsilon)$ and $R = R(\varepsilon)$ large enough
  \begin{equation}
    \label{eq:GHP-tightness_toshow}
    P_x^n[A_{x,n}(R)^c] \le \varepsilon/2
    \quad \text{and} \quad
    P_x^n[B_{x,n}^\delta(M)^c \cap A_{x,n}(R)] \le \delta \varepsilon/2,
  \end{equation}
  for all $\delta > 0$ and $n \ge 0$.

  We will first find a $R(\varepsilon)$ such that the first equation of
  \eqref{eq:GHP-tightness_toshow} is satisfied. Recall that the law of
  $(T_{x,n}, d_{x,n})$ under $P_x^n$ is that of $(T, d/n)$ under
  $P_x[\,\cdot\, | N_n^+ \neq\emptyset]$. Thus,
  $P_x^n[\diam(T_{x,n}) \ge 2R] \le P_x[|N_{Rn}^+|>0 | N_n^+ \neq\emptyset]$,
  which by Theorem~\ref{thm:diameter}, is smaller than $\varepsilon/2$
  for $R(\varepsilon)$ large enough. For such $R(\varepsilon)$ the first
  part of \eqref{eq:GHP-tightness_toshow} is valid.

  Next, we will find a $M(\varepsilon)$ such that also the second
  equation of \eqref{eq:GHP-tightness_toshow} is met. We will distinguish
  thereby between the cases where $n \delta < 1$ and $n \delta \ge 1$.
  Assume first that $n \delta < 1$. Then, every $\delta$-ball around a
  vertex $v \in T_{x,n}$ only contains the vertex $v$ itself. This means
  that in this case $T_{x,n}$ can be covered with $<\delta^{-4} M$ balls
  of radius $\delta$ if and only if the cardinality of $T_{x,n}$ is
  $<\delta^{-4} M$. Therefore
  \begin{equation}
    P_x^n[B_{x,n}^\delta(M)^c \cap A_{x,n}(R)]
    \le P_x[|T| \ge \delta^{-4}M | N_n^+ \neq \emptyset]
    \le \frac{P_x[|T| \ge \delta^{-4}M]}{P_x[N_n^+ \neq \emptyset]}.
  \end{equation}
  Using Theorem~\ref{thm:diameter} and Proposition~\ref{pro:cl23} this can be
  further bounded from above by
  \begin{equation}
    C\frac{\delta^2 n}{\sqrt{M}} \le C\frac{\delta}{\sqrt{M}},
  \end{equation}
  for some constant $C$, where we also used that by assumption
  $\delta n < 1$. Choosing some
  $M = M(\varepsilon) \ge (2C/\varepsilon)^2$, the second equation of
  \eqref{eq:GHP-tightness_toshow} is true for all $n, \delta$ with
  $n \delta < 1$.

  Next we consider the case $n \delta \ge 1$. Set
  $b_{n, \delta} = \floor{n \delta}$ and define the sets of vertices
  \(
    V_j = \set{v \in N_{j b_{n,\delta}}^+ :
      \exists w \in N_{(j+1) b_{n, \delta}}^+ \text{ with } v \prec w}
  \)
  for $j \ge 0$. We will show that on
  \begin{equation}
    \label{eq:GHP-balls_condition}
    \set{\diam(T_{x,n}) \le 2R} \cap
    \big(\bigcap_{j=0}^{2Rn/b_{n,\delta}-1}
      \big\{|V_j| \le \frac{M b_{n,\delta}}{2Rn\delta^4}\big\}\big)
  \end{equation}
  $(T_{x,n}, d_{x,n})$ can be covered with fewer than $\delta^{-4}M$ balls
  of radius $\delta$. Specifically, $T_{x,n}$ is covered by
  $\set{B_\delta(v)}_{v \in V}$, where
  $V = \cup_{j = 0}^{2Rn/b_{n,\delta}-1} V_j$. To see this, note that
  since on $\set{\diam(T_{x,n}) \le 2R}$ the height of $T_{x,n}$ is
  trivially bounded by $2Rn$ and thus, for every $v \in T_{x,n}$ there is
  a $w \in V_j$ for some $0\le j \le 2Rn/b_{n,\delta}-1$ so that
  $d_{x,n}(v,w) \le \delta$. Further, it is clear that on
  \eqref{eq:GHP-balls_condition} the cardinality of $V$ is bounded by
  $(2Rn/b_{n,\delta})(M b_{n, \delta}/2Rn\delta^4) = \delta^{-4}M$. This
  shows the claimed covering property. To prove the second part of
  \eqref{eq:GHP-tightness_toshow}, it thus suffices to show that for
  $M = M(R, \varepsilon)$ large enough,
  \begin{equation}
    \label{eq:GHP-tightness_step}
    P_x\Big[\bigcup_{j = 0}^{2Rn/b_{n,\delta}-1} \big\{|V_j|
        > \frac{M b_{n, \delta}}{2Rn\delta^4}\big\}\Big| N_n^+
      \neq \emptyset \Big]
    \le \delta\varepsilon/2.
  \end{equation}
  By conditioning on $\mathcal{F}_{jb_{n,\delta}}$ (recalling the
    definition of $\mathcal{F}_n$ from \eqref{eq:filtrations}) and using
  Theorem~\ref{thm:diameter} and Proposition~\ref{pro:Harris_L2},
  \begin{align}
    E_x\big[|V_j|\big]
    &= E_x\big[E_x\big[|V_j| \big| \mathcal{F}_{jb_{n,\delta}}\big]\big]
    = E_x\Big[\sum_{v \in N_{jb_{n,\delta}}^+}
      P_{\varphi_v}[N_{b_{n, \delta}}^+ \neq \emptyset]\Big] \\
    &= E_x\Big[\sum_{v \in N_{jb_{n,\delta}}^+}
      C \chi(\varphi_v)b_{n, \delta}^{-1}
      (1 + \varepsilon_{b_{n, \delta}}(\varphi_v))\Big]
    \le q(x)b_{n, \delta}^{-1}
  \end{align}
  for some function $q(x)$ independent of $n$ and $\delta$. Therefore, by
  conditioning on $\set{N_n^+ \neq\emptyset}$ and using the Markov
  inequality and Theorem~\ref{thm:diameter},
  \begin{equation}
    P_x\Big[|V_j|
      > \frac{M b_{n, \delta}}{2Rn\delta^4} \Big| N_n^+ \neq \emptyset \Big]
    \le \frac{2Rn\delta^4}{M b_{n, \delta}}
    \frac{E_x[|V_j|]}{P_x[N_n^+ \neq \emptyset]}
    \le q'(x) \frac{2Rn^2\delta^4}{M b_{n, \delta}^2},
  \end{equation}
  where $q'$ is some other function independent of $n$ and $\delta$. This means
  that by a union bound, the left-hand side of \eqref{eq:GHP-tightness_step}
  is bounded from above by
  \begin{equation}
    \label{eq:GHP-tightness_step2}
    q'(x) \frac{2Rn^2\delta^4}{M b_{n,\delta}^2}(2Rn/b_{n,\delta})
    = q'(x)\frac{4R^2n^3\delta^4}{M b_{n, \delta}^3}
    = \frac{q'(x) 4R^2 \delta}{M}
    \Big(\frac{n\delta}{\floor{n \delta}}\Big)^3.
  \end{equation}
  Using that for $x \ge 1$, $1 \le x/\floor{x} \le 2$, we see that by
  choosing $M \ge 64R^2q'(x)/\varepsilon$, the right-hand side of
  \eqref{eq:GHP-tightness_step2} is bounded by $\delta \varepsilon/2$.
  For such choices of $M$, \eqref{eq:GHP-tightness_step} is satisfied.

  Finally, first choosing $R$ as described above and then $M$ as the
  maximum of the two $M$-values obtained for the cases $n\delta < 1$ and
  $n \delta \ge 1$, gives $(R,M)$ such that
  \eqref{eq:GHP-tightness_toshow} is satisfied. For such $(R, M)$,
  \eqref{eq:GHP-tightness} holds true for $\mathbb{K} = \mathbb{K}_{R,M}$
  and the proof is finished.
\end{proof}

\begin{lemma}
  \label{lem:GP-convergence}
  For fixed $x \ge h^*$,
  \begin{equation}
    \label{eq:GP-convergence}
    (T_{x,n}, d_{x,n}, \mu_{x,n})
    \to (T_{\hat{\mathbf{e}}}, d_{\hat{\mathbf{e}}}, \mu_{\hat{\mathbf{e}}})
    \quad \text{as $n \to \infty$}
  \end{equation}
  in distribution with respect to the Gromov-Prokhorov metric.
\end{lemma}

\begin{proof}
  By Corollary~3.1 in \cite{GrePfaWin09}, the convergence
  \eqref{eq:GP-convergence} holds true if and only if
  \begin{enumerate}
    \item The family $\set{P_x^n}_{n \ge 0}$ is relatively compact in the
    space of probability measures on $\mathbb{X}$ (with respect to the
      Gromov-weak topology).
    \item For every function $\Psi: \mathbb{X} \to \mathbb{R}$ of the form,
    \begin{equation}
      \label{eq:Psi-def}
      \Psi((X, r, \mu))
      = \int \psi((r(x_i,x_j))_{1\le i < j \le k})
      \mu^{\otimes k}(\D(x_1, \dots, x_k)),
    \end{equation}
    where $\psi: [0, \infty)^{\binom{k}{2}} \to \mathbb{R}$ is a
    continuous and bounded function,
    $E^n_x[\Psi] \to E_{\hat{\mathbf{e}}}[\Psi]$ as $n \to \infty$.
  \end{enumerate}
  Since convergence in the Gromov-Hausdorff-Prokhorov sense implies
  convergence in the Gromov-weak sense, part (a) is implied by
  Lemma~\ref{lem:GHP-tightness} and only part (b) is left to be shown.

  To see (b), fix a function $\Psi$ as described in \eqref{eq:Psi-def}
  and recall the definition of $D^H_n$ from \eqref{eq:D-mat-def}. Note
  that $d_{x,n}(v,w) = n^{-1} (H(v) + H(w) - 2H(v \wedge w))$, and thus
  \begin{equation}
    \label{eq:P_x^n[psi]-eqiv}
    E^n_x[\Psi] = E_x\big[\psi(D^H_n) \big| N_n^+ \neq\emptyset\big].
  \end{equation}
  By Proposition~\ref{prop:S-H-mat},
  \begin{equation}
    \label{eq:P_x^n[psi]-conv1}
    \lim_{n\to \infty}
    \big(E_x\big[\psi(D^H_n) \big| N_n^+ \neq\emptyset \big]
      - E_x\big[\psi(D^{\bar{S}}_n / C_1^{-1} )
        \big| N_n^+ \neq\emptyset \big]\big) = 0.
  \end{equation}
  Assume for now that also
  \begin{equation}
    \label{eq:P_x^n[psi]-conv2}
    \lim_{n\to\infty} \big(
      E_x\big[\psi(D^{\bar{S}}_n/C_1^{-1}) \big| N_n^+ \neq\emptyset \big]
      -E_x\big[\psi(D^{\bar{S}}_n/C_1^{-1} ) \big|
        \sup_m \bar{S}_m \ge C_1^{-1} n \big] \big) = 0.
  \end{equation}
  By Proposition~\ref{prop:conditioned-scaling-limit},
  \(
    \lim_{n\to\infty} E_x[\psi(D^{\bar{S}}_n / C_1^{-1} )
      | \sup_m \bar{S}_m \ge C_1^{-1} n]
    = E_{\hat{\mathbf{e}}}[\Psi]
  \),
  and thus in combination with \eqref{eq:P_x^n[psi]-eqiv},
  \eqref{eq:P_x^n[psi]-conv1} and \eqref{eq:P_x^n[psi]-conv2}, it follows
  $\lim_{n \to \infty} E^n_x[\Psi] = E_{\hat{\mathbf{e}}}[\Psi]$,
  concluding the proof.

  It remains to prove \eqref{eq:P_x^n[psi]-conv2}. Since $\psi$ is
  bounded, it is enough to show that
  \begin{equation}
    \lim_{n\to\infty}
    P_x \big[N_n^+ \neq\emptyset
      \big| \sup_m \bar{S}_m \ge C_1^{-1} n \big] = 1
    \quad \text{and} \quad
    \lim_{n\to\infty}
    P_x \big[\sup_m \bar{S}_m \ge C_1^{-1} n
      \big| N_n^+ \neq\emptyset \big] = 1.
  \end{equation}
  Using that $P[A|B]=P[B|A]P[A]/P[B]$, it is enough to show that
  \begin{equation}
    \label{eq:GP-conv_lefttoshow}
    \lim_{n \to \infty}
    P_x[\sup_m \bar{S}_m \ge C_1^{-1} n | N_n^+ \neq\emptyset] = 1
    \text{\quad and \quad}
    \lim_{n\to\infty}
    \frac{P_x[\sup_m \bar{S}_m \ge C_1^{-1} n]}{P_x[N_n^+ \neq\emptyset]}
    = 1.
  \end{equation}
  Observe that for every $\delta > 0$,
  \begin{equation}
    \begin{split}
      P_x&\big[\sup_m \bar{S}_m \ge C_1^{-1} n
        \big| N_n^+ \neq\emptyset\big]
      \\&\ge P_x\big[\sup_m \bar{S}_m \ge C_1^{-1} n
        \big| |N_{\floor{n(1+\delta)}}^+|>0\big]
      \frac{P_x[|N_{\floor{n(1+\delta)}}^+|>0]}{P_x[N_n^+ \neq\emptyset]}.
    \end{split}
  \end{equation}
  By Proposition~\ref{prop:|N_n^bad|/|N_n|-fraction}, the first term in
  the product on the right-hand side converges to $1$ as $n \to \infty$,
  and the second term converges to $(1+\delta)^{-1}$ by
  Theorem~\ref{thm:diameter}. This shows the first convergence of
  \eqref{eq:GP-conv_lefttoshow}. By Theorem~\ref{thm:diameter},
  $P_x[N_n^+ \neq\emptyset] \sim \chi(x)C_1 n^{-1}$ as $n \to \infty$.
  Therefore, we are left to prove that
  \begin{equation}
    \label{eq:GP-conv_lefttoshow_2}
    \lim_{n\to\infty}
    \frac{P_x[\sup_m \bar{S}_m \ge C_1^{-1} n]}{\chi(x) C_1 n^{-1}} = 1
  \end{equation}
  to conclude the second limit of \eqref{eq:GP-conv_lefttoshow}. We will
  do this in a similar way as, e.g., in \cite[Section 1.4, p.~263]{Gal05}.

  Consider the depth-first traversal $(v_1, v_2, \dots)$ of the sequence
  of trees $(T^1, T^2, \dots)$. We set
  $S'_n = \sum_{w \in Y(v_n)} \chi(\varphi_w)$ and
  $S^i = \sup_n S'_n 1_{v_n \in T^i}$ for $i\ge 1$ (so that
    $\sup_n \bar{S}_n = S^1$). Recall that by
  Proposition~\ref{prop:joint-scaling-limit},
  \begin{equation}
    \lim_{n \to \infty}
    \Big(\frac{S'_{\floor{nt}}}{\sqrt{n}},
      \frac{\Lambda_{\floor{nt}}}{\sqrt{n}}\Big)_{t \ge 0}
    = \Big(\sigma |B_t|, \frac{\sigma}{\chi(x)}L_t^0\Big)_{t \ge 0},
  \end{equation}
  where the limit is in $P_x$-distribution with respect to the Skorokhod
  topology. Defining $\gamma_r = \inf\set{t \ge 0: L_t^0 > r}$ and
  $\tau_n = \inf\set{t \ge 0: n^{-1}\Lambda_{\floor{n^2 t}} > 1}$ and
  using Brownian scaling, this gives the joint convergence
  \begin{equation}
    \lim_{n\to\infty}
    \Big(\Big(\frac{1}{n} S'_{\floor{n^2 t}}\Big)_{t \ge 0}, \tau_n \Big)
    = \Big(\big(\sigma |B_t|\big)_{t \ge 0}, \gamma_{\chi(x)/\sigma}\Big),
  \end{equation}
  and therefore also
  \begin{equation}
    \lim_{n\to\infty}
    \Big(\frac{1}{n} S'_{\floor{n^2(t \wedge \tau_n)}} \Big)_{t \ge 0}
    = \big(\sigma |B_{t \wedge \gamma_{\chi(x)/\sigma}}|\big)_{t \ge 0}.
  \end{equation}
  From this we deduce that for every $y > 0$, as $n \to \infty$,
  \begin{equation}
    \label{eq:GP-convergence_step1}
    \lim_{n\to\infty} P_x\Big[\sup_{1 \le i \le n} S^i > ny \Big]
    = P_x\Big[\sup_{t \le \gamma_{\chi(x)/\sigma}} \sigma |B_t| > y\Big]
    = 1 - \exp\Big(-\frac{\chi(x)}{y}\Big),
  \end{equation}
  where the last equality is a consequence of excursion theory for
  Brownian motion (see e.g. Chapter XII in \cite{RevYor99}). By the
  independence of the trees $T^i$ (and thus the variables $S^i$), we also
  have that
  \begin{equation}
    \label{eq:GP-convergence_step2}
    P_x\Big[\sup_{1 \le i \le n} S^i > ny \Big]
    = 1 - \big(1- P_x\big[S^1 > ny\big]\big)^n
    = 1 - \Big(1- P_x\Big[\sup_{m} \bar{S}_m > ny\Big]\Big)^n.
  \end{equation}
  Combining \eqref{eq:GP-convergence_step1} and
  \eqref{eq:GP-convergence_step2} then gives
  \begin{equation}
    P_x\Big[\sup_{m} \bar{S}_m > ny\Big] \sim \frac{\chi(x)}{y}n^{-1}
    \quad \text{as $n \to \infty$},
  \end{equation}
  which by setting $y = C_1^{-1}$ concludes
  \eqref{eq:GP-conv_lefttoshow_2} and finishes the proof.
\end{proof}

We now present the proof of Theorem~\ref{thm:inv-principle}, which is
identical to that of Theorem~1.1 in \cite{Pow19}.

\begin{proof}[Proof of Theorem~\ref{thm:inv-principle}]
  Since convergence in the Gromov-Hausdorff-Prokhorov metric implies
  convergence in the Gromov-Prokhorov metric,
  Lemma~\ref{lem:GP-convergence} characterizes subsequential limits with
  respect to the Gromov-Hausdorff-Prokhorov topology. Hence, we obtain
  the convergence in distribution
  \begin{equation}
    (T_{x,n}, d_{x,n}, \mu_{x,n})
    \to (T_{\hat{\mathbf{e}}}, d_{\hat{\mathbf{e}}}, \mu_{\hat{\mathbf{e}}})
    \quad \text{as $n \to \infty$}
  \end{equation}
  with respect to the Gromov-Hausdorff-Prokhorov metric. Now, since
  Gromov-Hausdorff-Prokhorov convergence further implies convergence in
  the Gromov-Hausdorff sense, claim \eqref{eq:inv-princ_toshow} follows,
  completing the proof.
\end{proof}

\appendix
\section{\texorpdfstring{Properties of $\chi $}{Properties of chi}}
\label{app:proof-chi-lipsch}

We prove here that $\chi$ is Lipschitz continuous on $[h^*, \infty)$, as stated in
\eqref{eq:chi-Lipschitz}. The proof uses similar ideas as the proof of
Proposition~3.1 of \cite{AbaCer19}.

\begin{proof}[Proof of \eqref{eq:chi-Lipschitz}]
  We will show that the derivative $\chi'$ is bounded on $[h^*, \infty)$.
  Recall that for $x\ge h^*$ it holds
  $\chi(x) = d\int_{[h^*, \infty)} \chi(z) \rho_Y(z - x/d)\D z$, where
  $\rho_Y$ is the density of the Gaussian random variable $Y$.
  Differentiating this expression, using an integration by parts and the
  fact that $\chi'(z)=0$ for $z< h^*$, we obtain that for $x\ge h^*$
  \begin{equation}
    \begin{split}
      \label{eq:chi'-1}
      \chi'(x)
      & = - \int_{h^*}^\infty \chi(z)\rho_Y'\Big(z - \frac{x}{d}\Big)\D z
      \\ & = \int_{[h^*, \infty)}\chi'(z)\rho_Y\Big(z - \frac xd\Big)\D z
      + \chi(h^*)\rho_Y\Big(h^* - \frac xd\Big)
      \\ & \eqdef E_Y\Big[\chi'\Big(Y + \frac xd\Big)\Big] + e(x).
    \end{split}
  \end{equation}
  Defining $e(x) = 0$ for $x< h^*$, using that $\chi$ is increasing and
  thus $\chi '\ge 0$, this equality can be extended to inequality
  \begin{equation}
    \label{eq:chi-Lipsch-simple}
    \chi'(x) \le E_Y\Big[\chi'\Big(Y + \frac xd\Big)\Big] + e(x)
    \qquad \text{for all } x\in \mathbb R.
  \end{equation}
  Similarly as in the proof of Proposition~3.1 in \cite{AbaCer19}, we
  obtain an upper bound on $\chi'$ by iterating
  \eqref{eq:chi-Lipsch-simple} an appropriate amount of times. For this,
  let $Y_1, \dots, Y_k$ be independent Gaussian random variables having
  the same distribution as $Y$, and define
  $Z_i = Y_1/d^{i-1} + Y_2/d^{i-2} + \dots + Y_i$. Note that $Z_i$ is a
  centred Gaussian random variable whose variance, denoted $\sigma^2_{i}$,
  is bounded uniformly in $i$. Then, by applying
  \eqref{eq:chi-Lipsch-simple} $k$-times,
  \begin{equation}
    \begin{split}
      \label{eq:chi'-bound}
      \chi'(x)
      &\le E_{Y_1}\Big[
        E_{Y_2}\Big[
          \chi'\Big(Y_2 + \frac{Y_1 + x/d}{d}\Big)
          + e(Y_1 + x/d) + e(x)\Big]\Big]
      \\ &\le \dots \le
      E\Big[\chi'\Big(\frac{x}{d^k} + \frac{Y_1}{d^{k-1}}
          + \frac{Y_2}{d^{k-2}} + \dots + Y_k\Big)\Big]
      + E\Big[\sum_{i=0}^{k-1} e\Big(\frac{x}{d^i} + Z_i\Big)\Big]
      \\ &\le E\Big[\chi'\Big(\frac{x}{d^k} + Z_k\Big)\Big]
      + \sum_{i=0}^{k-1} E\Big[e\Big(\frac{x}{d^i} + Z_i\Big)\Big].
    \end{split}
  \end{equation}
  We now choose $k=k(x) = \floor{\log_d(x)}$. Then the boundedness of the
  first summand can be proved exactly in the same way as the boundedness
  of $E[\chi(x/d^k + Z_k)]$ in the proof of Proposition~3.1 in
  \cite{AbaCer19} (note that $(d-1)$ in \cite{AbaCer19} corresponds to $d$
    in our setting), one only needs to verify that $\chi' \in L^2(\nu)$.
  This can be easily proved using \eqref{eq:chi'-1} and the facts
  $\rho_Y'(x) = -c x \rho_Y(x)$, $L[\chi] = \chi$ and
  \eqref{eq:crit-chi_bounds} which imply
  $|\chi'(x)| \le c L[\chi^2](x) + c x^2$ for $x \ge h^*$. From this
  $\chi'\in L^2(\nu)$ follows from
  Proposition~\ref{prop:hypercontractivity}(b).

  To bound the second summand on the right hand side of
  \eqref{eq:chi'-bound}, we first observe that by the choice of $k(x)$ it
  holds that $x_0 := x/d^{k(x)} \in [1,d]$. With this notation, by
  rearranging the sum,
  \begin{equation}
    \sum_{i=0}^{k(x)-1} E[e(x/d^i + Z_i)]
    = \sum_{i=0}^{k(x)-1} E[e(x_0 d^i + Z_{k(x)-i})].
  \end{equation}
  Moreover, it is easy to see that for some constants $c,c'$
  \begin{equation}
    0\le e(x) := 1_{x\ge h^*} \chi (h^*) \rho_Y\Big(h^* - \frac xd\Big)
    \le c' 1_{x\ge h^*} e^{-c x}.
  \end{equation}
  Therefore,
  \begin{equation}
    \sum_{i=0}^{k(x)-1} E[e(x_0 d^i + Z_{k(x)-i})]
    \le
    c' \sum_{i=0}^\infty E\big[e^{-c(x_0d^i + Z_{k(x)-i})}\big]
    \le
    c' \sum_{i=0}^\infty e^{-c x_0 d^i} e^{c^2 \sigma_{k(x)-i}^2/2},
  \end{equation}
  which, since $\sigma_i^2$ are bounded and $x_0 \in [1,d]$, is clearly
  bounded uniformly in $x\ge h^*$.
\end{proof}

\section{Proof of the many-to-few formulas}
\label{sec:B}

We give here a proof of Proposition~\ref{pro:many-to-few_model}.
Throughout the proof we use the notation from Section~\ref{ss:spined-bp}.
We start by introducing a useful standard martingale related to
$\varphi$ when viewed as branching process.

\begin{lemma}
  \label{lem:martingale_zeta}
  For $x\in \mathbb R$, $n\in \mathbb N$, let
  \begin{equation}
    \zeta(x, n) \defeq d^n\chi(x) \qquad \text{and} \qquad
    \zeta_n  \defeq \zeta(\xi_n^1, n).
  \end{equation}
  Then, for every $k\in \mathbb N$, $(\zeta_n)_{n \in \mathbb{N}}$ is a
  $P_x^k$-martingale with respect to $\mathcal F_n^k$.
\end{lemma}

\begin{proof}
  By definition of $\zeta_n$, for $n\ge 1$,
  \begin{equation}
    d^{-(n-1)}  E_x^k [\zeta_n \mid \mathcal F_{n-1}^k]
    =  d E_x^k [ \chi (\varphi_{\sigma_n^1}) \mid \mathcal F_{n-1}^k],
  \end{equation}
  where $\sigma_n^1$ is the vertex carrying the first spine mark at level $n$.
  Conditionally on $\mathcal F_{n-1}^k$,
  $\sigma_n^1$ is uniformly distributed on $\desc(\sigma_{n-1}^1)$ and
  independent of $\varphi$. Therefore, this
  equals
  \begin{equation}
    \sum_{v\in \desc(\sigma_{n-1}^1)}
    E_x^k[\chi (\varphi_v) \mid \mathcal F_{n-1}^k] =
    L[\chi](\varphi_{\sigma^1_{n-1}}) = \chi (\varphi_{\sigma^1_{n-1}})
    = \chi (\xi^1_{n-1}).
  \end{equation}
  where in for the first equality we used \eqref{eq:L_h-BP1} and the fact
  that $\chi (x) = 0$ on $(-\infty,h^*)$, and where the second equality
  follows from \eqref{eq:chieigenfunction}. The martingale property of
  $\zeta_n$ then follows directly from this computation.
\end{proof}

Due to Lemma~\ref{lem:martingale_zeta}, Lemma~8 of \cite{HarRob17} can
directly be applied to our process. We restate it here for reader's
convenience, with very minor adaptations coming from the fact that in our
process every node has always $d$ descendants (some of them might be in
  the cemetery state).

\begin{lemma}[Lemma~8 in \cite{HarRob17}]
  \label{lem:many-to-few_general}
  For any $k\geq 1$, let $Y$ be a $\F_n^k$-measurable random variable
  which can be written as
  \begin{equation}
    Y = \sum_{v_1, \dots, v_k \in S_n^+} Y(v_1, \dots, v_k)
    1_{\set{\sigma_n^1=v_1, \dots, \sigma_n^k=v_k}},
  \end{equation}
  where, for every $v_1, \dots, v_k \in S_n^+$, $Y(v_1, \dots, v_k)$ is a
  $\F_n$-measurable random variable. Then
  \begin{equation}
    \label{eq:many-to-few_general}
    E_x\Big[\sum_{v_1, \dots, v_k \in N^+_n} Y(v_1, \dots, v_k)\Big]
    = \Q_x^k\Big[ Y
      \prod_{v \in \skel(n)\setminus\set{o}}
      \frac{\zeta(\varphi_{p(v)}, |v|-1)}{\zeta(\varphi_v, |v|)}d^{l_{p(v)}}
      \Big].
  \end{equation}
\end{lemma}

We are now ready to give the proof of
Proposition~\ref{pro:many-to-few_model}.

\begin{proof}[Proof of Proposition~\ref{pro:many-to-few_model}]
  Statement \eqref{eq:many-to-few_model1} follows directly from
  Lemma~\ref{lem:many-to-few_general} with $k=1$ and
  $Y(v_1) = f(\varphi_{v_1})$. Indeed, since $k=1$, there is only one mark
  on every level and thus $Y = f(\xi_n^1)$,
  $\skel(n)\setminus \set{o} = \set{\sigma_1^1, \dots, \sigma_n^1}$, and
  $l_{p(v)} = 1$ for all $v \in \skel(n)\setminus\set{o}$. Therefore, by
  \eqref{eq:many-to-few_general}
  \begin{equation}
    \begin{split}
      E_x\Big[\sum_{v\in N_n^+} f(\varphi_v)\Big]
      &= \Q_x\bigg[f(\xi_n) \prod_{v\in \{\sigma_1,\dots,\sigma_n\}}
        \Big(\frac{\zeta(\varphi_{p(v)}, |v|-1)}{\zeta(\varphi_v, |v|)}
          d\Big)\bigg] \\
      &= \Q_x\bigg[f(\xi_n) \prod_{i=1}^n
        \Big(\frac{\chi(\xi_{i-1})d^{i-1}}{\chi(\xi_i)d^i} d\Big)\bigg]
      = \Q_x\bigg[f(\xi_n^1)
        \frac{\chi(x)}{\chi(\xi_n)}\bigg],
    \end{split}
  \end{equation}
  as claimed in \eqref{eq:many-to-few_model1}.

  Similarly, statement \eqref{eq:many-to-few_model2} is a consequence
  of Lemma~\ref{lem:many-to-few_general} with $k=2$. We set
  $Y(v_1,v_2)= f(\varphi_{v_1})g(\varphi_{v_2})$. Then
  $Y = f(\xi_n^1)g(\xi_n^2)$ and by \eqref{eq:many-to-few_general},
  \begin{equation}
    \label{eq:many-to-two_init}
    E_x\Big[\sum_{v,w\in N^+_n} f(\varphi_v)  g(\varphi_w)\Big]
    = \Q_x^2\bigg[f(\xi_{n}^1)g(\xi_n^2)
      \prod_{v \in \skel(n)\setminus\set{o}}
      \frac{\zeta(\varphi_{p(v)}, \abs{v}-1)}{\zeta(\varphi_v, \abs{v})}
      d^{l_{p(v)}}\bigg].
  \end{equation}
  To consider the possible structures of $\skel(n)\setminus\set{o}$,
  set $s = \max\set{k : \sigma_k^1 = \sigma_k^2}$ to be the last
  time where the two spines agree. Note that due to the dynamics of marks
  under~$Q^2_x$,
  \begin{equation}
    \Q_x^2[s = k] = \begin{cases}
      (d-1)d^{-(k+1)},\quad&\text{if } k \in \{0,\dots, n-1\},\\
      d^{-n},& \text{if } k = n,
    \end{cases}
  \end{equation}
  and, by a simple computation,
  \begin{equation}
    \begin{split}
      \prod_{v \in \skel(n)\setminus \set{o}}
      \frac{\zeta(\varphi_{p(v)}, \abs{v}-1)}
      {\zeta(\varphi_v, \abs{v})}d^{l_{p(v)}}
      &= \prod_{v \in \skel(n)\setminus \set{o}}
      \left(\frac{\chi(\varphi_{p(v)})}{\chi(\varphi_v)}
        \frac{1}{d}\right) d^{l_{p(v)}}
      = \frac{\chi(\xi^1_s)\chi(x)}{\chi(\xi_n^1)\chi(\xi_n^2)}d^s.
    \end{split}
  \end{equation}
  Therefore, \eqref{eq:many-to-two_init} can be written as
  \begin{equation}
    \begin{split}
      \label{eq:many-to-two_2}
      E_x&\bigg[\sum_{v\in N_n^+} f(\varphi_v) \sum_{v \in N_n^+}
        g(\varphi_v)\bigg]
      = \sum_{k=0}^n d^k \Q_x^2[s=k]\Q_x^2\bigg[
        f(\xi_n^1) g(\xi_n^2)
        \frac{\chi(\xi^1_k)\chi(x)}{\chi(\xi_n^1)\chi(\xi_n^2)}
        \bigg| s= k\bigg] \\
      &= \frac{d-1}{d} \sum_{k=0}^{n-1} \Q_x^2\bigg[
        \frac{f(\xi_n^1)}{\chi(\xi_n^1)}
        \frac{g(\xi_n^2)}{\chi(\xi_n^2)}\chi(\xi^1_s)\chi(x)
        \bigg| s= k\bigg]
      \\&\quad+ \Q_x^2\bigg[
        f(\xi_n^1)g(\xi_n^1)\frac{\chi(x)}{\chi(\xi^1_n)}
        \bigg| s= n\bigg].
    \end{split}
  \end{equation}
  By construction, under $Q^2_x$, conditional on $s=k$, for times
  $i=1, \dots, k$ the processes $\xi_i^1$ and $\xi_i^2$ are Markov chains
  which follow the same trajectory and have the same dynamics as $\xi_i$
  under $Q^1_x$. Further, for later times $i=k+1,\dots, n$, they are
  independent Markov chains distributed according to $Q_{\xi_k^1}$.
  Therefore,
  \begin{equation}
    \Q_x^2\Big[\frac{f(\xi_n^1)}{\chi(\xi_n^1)}
      \frac{g(\xi_n^2)}{\chi(\xi_n^2)}\chi(\xi^1_s)\chi(x)\Big| s= k\Big]
    = \chi(x) Q_x\bigg[\chi (\xi_k)
      \Q_{\xi_k^1}^2\Big[\frac{f(\xi_{n-k}^1)}{\chi(\xi_{n-k}^1)}\Big]
      \Q_{\xi_k^1}^2\Big[\frac{g(\xi_{n-k}^2)}{\chi(\xi_{n-k}^2)}\Big]
      \bigg],
  \end{equation}
  and similarly
  \begin{equation}
    \label{eq:end}
    \Q_x^2\bigg[
      f(\xi_n^1)g(\xi_n^1)\frac{\chi(x)}{\chi(\xi^1_n)}
      \bigg| s= n\bigg] =
    \chi(x) Q_x\Big[\frac{f(\xi_n)g(\xi_n)}{\chi
        (\xi_n)}\Big].
  \end{equation}
  Combining \eqref{eq:many-to-two_2}--\eqref{eq:end} directly implies
  \eqref{eq:many-to-few_model2}.
\end{proof}

\bibliographystyle{jcamsalpha}
\bibliography{gffinvprinc}

\def\arxiv#1{Preprint, available at \href{http://arxiv.org/abs/#1}{arXiv:#1}}
\providecommand{\bysame}{\leavevmode\hbox to3em{\hrulefill}\thinspace}
\providecommand\MR{}
\renewcommand\MR[1]{\relax\ifhmode\unskip\spacefactor3000
\space\fi \MRhref{#1}{#1}}
\providecommand\MRhref{}
\renewcommand{\MRhref}[2]%
{\href{http://www.ams.org/mathscinet-getitem?mr=#1}{MR#2}}
\providecommand{\href}[2]{#2}
\begin{thebibliography}{DCGRS23}

\bibitem[A{\v C}20]{AbaCer19}
Angelo Ab\"{a}cherli and Ji{\v r}{\'\i} {\v C}ern{\'y}, \emph{Level-set
  percolation of the {G}aussian free field on regular graphs {I}: regular
  trees}, Electron. J. Probab. \textbf{25} (2020), Paper No. 65, 24.
  \MR{4115734}

\bibitem[ADH13]{AbDeHo13}
Romain Abraham, Jean-Fran\c{c}ois Delmas, and Patrick Hoscheit, \emph{A note on
  the {G}romov-{H}ausdorff-{P}rokhorov distance between (locally) compact
  metric measure spaces}, Electron. J. Probab. \textbf{18} (2013), no. 14, 21.
  \MR{3035742}

\bibitem[Ald91]{Ald91}
David Aldous, \emph{The continuum random tree. {I}}, Ann. Probab. \textbf{19}
  (1991), no.~1, 1--28. \MR{1085326}

\bibitem[BDIM23]{BDIM23}
Dariusz Buraczewski, Congzao Dong, Alexander Iksanov, and Alexander Marynych,
  \emph{Critical branching processes in a sparse random environment}, Mod.
  Stoch. Theory Appl. \textbf{10} (2023), no.~4, 397--411. \MR{4655407}

\bibitem[BFRS24]{BFS24arXiv}
Florin Boenkost, Félix Foutel-Rodier, and Emmanuel Schertzer, \emph{The
  genealogy of nearly critical branching processes in varying environment},
  \arxiv{2207.11612}, 2024.

\bibitem[BLM87]{BriLebMae87}
Jean Bricmont, Joel~L. Lebowitz, and Christian Maes, \emph{Percolation in
  strongly correlated systems: the massless {G}aussian field}, J. Statist.
  Phys. \textbf{48} (1987), no.~5-6, 1249--1268. \MR{914444}

\bibitem[CD24]{CaiDin24arXiv}
Zhenhao Cai and Jian Ding, \emph{One-arm probabilities for metric graph
  gaussian free fields below and at the critical dimension},
  \arxiv{2406.02397}, 2024.

\bibitem[CD25]{CD25}
Zhenhao Cai and Jian Ding, \emph{One-arm exponent of critical level-set for
  metric graph {G}aussian free field in high dimensions}, Probability Theory
  and Related Fields \textbf{191} (2025), no.~3, 1035--1120. \MR{4898098}

\bibitem[CKKM24]{CKM24}
Guillaume Conchon-Kerjan, Daniel Kious, and C\'ecile Mailler, \emph{Scaling
  limit of critical random trees in random environment}, Electron. J. Probab.
  \textbf{29} (2024), Paper No. 112, 53. \MR{4779872}

\bibitem[{\v C}L25]{CerLoc23}
Ji{\v r}{\'i} {\v C}ern{\'y} and Ramon Locher, \emph{Critical and near-critical
  level-set percolation of the {G}aussian free field on regular trees}, 2025,
  pp.~746--767. \MR{4863059}

\bibitem[CN20]{ChiNit20}
Alberto Chiarini and Maximilian Nitzschner, \emph{Entropic repulsion for the
  {G}aussian free field conditioned on disconnection by level-sets}, Probab.
  Theory Related Fields \textbf{177} (2020), no.~1-2, 525--575. \MR{4095021}

\bibitem[CR90]{ChaRou90}
B.~Chauvin and A.~Rouault, \emph{Supercritical branching {B}rownian motion and
  {K}-{P}-{P} equation in the critical speed-area}, Math. Nachr. \textbf{149}
  (1990), 41--59. \MR{1124793}

\bibitem[CTJP24]{CJP24}
Natalia Cardona-Tob\'on, Arturo Jaramillo, and Sandra Palau, \emph{Rates on
  {Y}aglom's limit for {G}alton-{W}atson processes in a varying environment},
  ALEA Lat. Am. J. Probab. Math. Stat. \textbf{21} (2024), no.~1, 1--23.
  \MR{4703767}

\bibitem[DCGRS23]{DGRS23}
Hugo Duminil-Copin, Subhajit Goswami, Pierre-Fran\c{c}ois Rodriguez, and Franco
  Severo, \emph{Equality of critical parameters for percolation of {G}aussian
  free field level sets}, Duke Math. J. \textbf{172} (2023), no.~5, 839--913.
  \MR{4568695}

\bibitem[DLG02]{DLG02}
Thomas Duquesne and Jean-Fran\c{c}ois Le~Gall, \emph{Random trees, {L\'evy}
  processes and spatial branching processes}, Ast\'erisque, no. 281,
  Soci\'et\'e math\'ematique de France, 2002 (en). \MR{1954248}

\bibitem[DPR18]{DrePreRod18}
Alexander Drewitz, Alexis Pr\'{e}vost, and Pierre-Fran\c{c}cois Rodriguez,
  \emph{The sign clusters of the massless {G}aussian free field percolate on
  {$\mathbb{Z}^d, d \geqslant 3$} (and more)}, Comm. Math. Phys. \textbf{362}
  (2018), no.~2, 513--546. \MR{3843421}

\bibitem[DPR25]{DPR25}
Alexander Drewitz, Alexis Pr{\'e}vost, and Pierre-Fran{\c{c}}ois Rodriguez,
  \emph{Critical one-arm probability for the metric {G}aussian free field in
  low dimensions}, Probability Theory and Related Fields (2025).

\bibitem[dR17]{Rap17}
Lo\"{\i}c de~Raph\'{e}lis, \emph{Scaling limit of multitype {G}alton-{W}atson
  trees with infinitely many types}, Ann. Inst. Henri Poincar\'{e} Probab.
  Stat. \textbf{53} (2017), no.~1, 200--225. \MR{3606739}

\bibitem[DRS14]{DreRatSap14}
Alexander Drewitz, Bal\'{a}zs R\'{a}th, and Art\"{e}m Sapozhnikov, \emph{On
  chemical distances and shape theorems in percolation models with long-range
  correlations}, J. Math. Phys. \textbf{55} (2014), no.~8, 083307, 30.
  \MR{3390739}

\bibitem[Dur19]{Durrett19}
Rick Durrett, \emph{Probability---theory and examples}, fifth ed., Cambridge
  Series in Statistical and Probabilistic Mathematics, vol.~49, Cambridge
  University Press, Cambridge, 2019. \MR{3930614}

\bibitem[Fel71]{Fel71}
William Feller, \emph{An introduction to probability theory and its
  applications. {V}ol. {II}.}, Second edition, John Wiley \& Sons Inc., New
  York, 1971. \MR{0270403}

\bibitem[GLLP22]{GLL22}
Ion Grama, Ronan Lauvergnat, and \'Emile Le~Page, \emph{Limit theorems for
  critical branching processes in a finite-state-space {M}arkovian
  environment}, Adv. in Appl. Probab. \textbf{54} (2022), no.~1, 111--140.
  \MR{4397862}

\bibitem[GPW09]{GrePfaWin09}
Andreas Greven, Peter Pfaffelhuber, and Anita Winter, \emph{Convergence in
  distribution of random metric measure spaces ({$\Lambda$}-coalescent measure
  trees)}, Probab. Theory Related Fields \textbf{145} (2009), no.~1-2,
  285--322. \MR{2520129}

\bibitem[Gri99]{Gri99}
Geoffrey Grimmett, \emph{Percolation}, Grundlehren der Mathematischen
  Wissenschaften, vol. 321, Springer-Verlag, Berlin, 1999. \MR{1707339}

\bibitem[GRS22]{GosRodSev22}
Subhajit Goswami, Pierre-Fran\c{c}ois Rodriguez, and Franco Severo, \emph{On
  the radius of {G}aussian free field excursion clusters}, Ann. Probab.
  \textbf{50} (2022), no.~5, 1675--1724. \MR{4474499}

\bibitem[HH07]{HarHar07}
J.~W. Harris and S.~C. Harris, \emph{Survival probabilities for branching
  {B}rownian motion with absorption}, Electron. Comm. Probab. \textbf{12}
  (2007), 81--92. \MR{2300218}

\bibitem[HHKW22]{HHK22}
Simon~C. Harris, Emma Horton, Andreas~E. Kyprianou, and Minmin Wang,
  \emph{Yaglom limit for critical nonlocal branching {M}arkov processes}, Ann.
  Probab. \textbf{50} (2022), no.~6, 2373--2408. \MR{4499840}

\bibitem[HR17]{HarRob17}
Simon~C. Harris and Matthew~I. Roberts, \emph{The many-to-few lemma and
  multiple spines}, Ann. Inst. Henri Poincar\'{e} Probab. Stat. \textbf{53}
  (2017), no.~1, 226--242. \MR{3606740}

\bibitem[Kol38]{Kol38}
Andrey~Nikolaevich Kolmogorov, \emph{Zur {L}{\"o}sung einer biologischen
  {A}ufgabe}, Comm. Math. Mech. Chebyshev Univ. Tomsk \textbf{2} (1938), no.~1,
  1--12.

\bibitem[LG05]{Gal05}
Jean-Fran\c{c}ois Le~Gall, \emph{Random trees and applications}, Probab. Surv.
  \textbf{2} (2005), 245--311. \MR{2203728}

\bibitem[LS86]{LebSal86}
Joel~L. Lebowitz and H.~Saleur, \emph{Percolation in strongly correlated
  systems}, Phys. A \textbf{138} (1986), no.~1-2, 194--205. \MR{865243}

\bibitem[Mie08]{Mie08}
Gr\'{e}gory Miermont, \emph{Invariance principles for spatial multitype
  {G}alton-{W}atson trees}, Ann. Inst. Henri Poincar\'{e} Probab. Stat.
  \textbf{44} (2008), no.~6, 1128--1161. \MR{2469338}

\bibitem[Mod71]{Mode71}
Charles~J. Mode, \emph{Multitype branching processes. {T}heory and
  applications}, Modern Analytic and Computational Methods in Science and
  Mathematics, No. 34, American Elsevier Publishing Co., Inc., New York, 1971.
  \MR{0279901}

\bibitem[MS83]{MolSte83}
S.~A. Molchanov and A.~K. Stepanov, \emph{Percolation in random fields. {I}},
  Teoret. Mat. Fiz. \textbf{55} (1983), no.~2, 246--256. \MR{734878}

\bibitem[MS22]{MaiSch22}
Pascal Maillard and Jason Schweinsberg, \emph{Yaglom-type limit theorems for
  branching {B}rownian motion with absorption}, Ann. H. Lebesgue \textbf{5}
  (2022), 921--985. \MR{4526243}

\bibitem[Pow19]{Pow19}
Ellen Powell, \emph{An invariance principle for branching diffusions in bounded
  domains}, Probab. Theory Related Fields \textbf{173} (2019), no.~3-4,
  999--1062. \MR{3936150}

\bibitem[PR15]{PopRat15}
Serguei Popov and Bal\'{a}zs R\'{a}th, \emph{On decoupling inequalities and
  percolation of excursion sets of the {G}aussian free field}, J. Stat. Phys.
  \textbf{159} (2015), no.~2, 312--320. \MR{3325312}

\bibitem[PS22]{PanSev22}
Christoforos Panagiotis and Franco Severo, \emph{Analyticity of {G}aussian free
  field percolation observables}, Comm. Math. Phys. \textbf{396} (2022), no.~1,
  187--223. \MR{4499015}

\bibitem[RS13]{RodSzn13}
Pierre-Fran\c{c}ois Rodriguez and Alain-Sol Sznitman, \emph{Phase transition
  and level-set percolation for the {G}aussian free field}, Comm. Math. Phys.
  \textbf{320} (2013), no.~2, 571--601. \MR{3053773}

\bibitem[RY99]{RevYor99}
Daniel Revuz and Marc Yor, \emph{Continuous martingales and {B}rownian motion},
  third ed., Grundlehren der mathematischen Wissenschaften, vol. 293,
  Springer-Verlag, Berlin, 1999. \MR{1725357}

\bibitem[Szn16]{Szn15}
Alain-Sol Sznitman, \emph{Coupling and an application to level-set percolation
  of the {G}aussian free field}, Electron. J. Probab. \textbf{21} (2016),
  1--26. \MR{3492939}

\bibitem[Szn19]{Szn19.2}
Alain-Sol Sznitman, \emph{On macroscopic holes in some supercritical strongly
  dependent percolation models}, Ann. Probab. \textbf{47} (2019), no.~4,
  2459--2493. \MR{3980925}

\bibitem[Yag47]{Yag47}
A.~M. Yaglom, \emph{Certain limit theorems of the theory of branching random
  processes}, Doklady Akad. Nauk SSSR (N.S.) \textbf{56} (1947), 795--798.
  \MR{22045}

\end{thebibliography}

\end{document}